\documentclass[a4paper]{article}

\usepackage{graphicx} 
\usepackage{amsthm, amsmath, amssymb, amsfonts}
\usepackage{latexsym}
\usepackage{mathrsfs} 
\usepackage{mathtools}
\usepackage{tensor} 
\usepackage{url}
\usepackage{enumerate}
\usepackage[margin=29mm]{geometry}
\usepackage{authblk}
\usepackage{hyperref}
\usepackage[dvipsnames]{xcolor}
\usepackage{stmaryrd} 
\usepackage{docmute}
\usepackage{centernot}

\usepackage{appendix}
\AtBeginEnvironment{subappendices}{%
\chapter*{Appendix}
\addcontentsline{toc}{chapter}{Appendices}
\counterwithin{figure}{section}
\counterwithin{table}{section}
}

\definecolor{OLgreen}{rgb}{0.074,0.541,0.027}

\mathtoolsset{showonlyrefs}

\allowdisplaybreaks

\hypersetup{
    colorlinks=true,
    linkcolor={blue!80!black},
    citecolor={OLgreen!90!black},
    urlcolor=magenta,
    linktocpage=true,
}

\numberwithin{equation}{section}

\theoremstyle{plain}
\newtheorem{Thm}{Theorem}[section]
\newtheorem{Lem}[Thm]{Lemma}
\newtheorem{Prp}[Thm]{Proposition}
\newtheorem{Cor}[Thm]{Corollary}
\theoremstyle{definition}

\newtheorem{Rem}[Thm]{Remark}

\theoremstyle{plain}
\newtheorem*{Thm*}{Theorem}
\newtheorem*{Lem*}{Lemma}
\newtheorem*{Prp*}{Proposition}
\newtheorem*{Cor*}{Corollary}
\theoremstyle{definition}
\newtheorem*{Def*}{Definition}
\newtheorem*{Rem*}{Remark}
\newtheorem*{Exa*}{Example}

\DeclareMathOperator\Dom{Dom}

\DeclarePairedDelimiter{\norm}{\lVert}{\rVert} 
\DeclarePairedDelimiter{\abra}{\langle}{\rangle} 
\DeclarePairedDelimiter{\floor}{\lfloor}{\rfloor} 
\DeclarePairedDelimiter{\ceil}{\lceil}{\rceil} 

\newcommand{\E}{\mathbb{E}}
\newcommand{\N}{\mathbb{N}}
\newcommand{\R}{\mathbb{R}}

\newcommand{\Z}{\mathbb{Z}}

\newcommand{\bD}{\mathbb{D}}

\newcommand{\bW}{\mathbb{W}}

\newcommand{\cA}{\mathcal{A}}

\newcommand{\cC}{\mathcal{C}}

\newcommand{\cF}{\mathcal{F}}

\newcommand{\cM}{\mathcal{M}}
\newcommand{\cN}{\mathcal{N}}

\newcommand{\cP}{\mathcal{P}}

\newcommand{\cS}{\mathcal{S}}

\newcommand{\cW}{\mathcal{W}}

\newcommand{\sC}{\mathscr{C}}
\newcommand{\sD}{\mathscr{D}}

\newcommand{\sF}{\mathscr{F}}

\newcommand{\sH}{\mathscr{H}}

\newcommand{\sL}{\mathscr{L}}

\newcommand{\sS}{\mathscr{S}}

\newcommand{\fE}{\mathfrak{E}}

\newcommand{\fH}{\mathfrak{H}}

\newcommand{\ind}{\mathbf{1}}

\newcommand{\dis}{\displaystyle}
\newcommand{\wt}{\widetilde}
\newcommand{\relmiddle}[1]{\mathrel{}\middle#1\mathrel{}}

\newcommand{\dpair}[4]{\tensor[_{{#1}}]{\langle{#2}, {#3}\rangle}{_{#4}}}

\newcommand{\dpairLR}[4]{\tensor[_{{#1}}]{\left\langle{#2}, {#3}\right\rangle}{_{#4}}}

\title{Optimal local central limit theorems on Wiener chaos}
\author[1]{Masahisa Ebina}
\author[2]{Ivan Nourdin}
\author[2]{Giovanni Peccati}

\affil[1]{Kyushu University}
\affil[2]{University of Luxembourg}

\date{}

\begin{document}
\maketitle

\begin{abstract}
This paper investigates a local central limit theorem for a normalized sequence of random variables belonging to a fixed order Wiener chaos and converging to the standard normal distribution.
We prove, without imposing any additional conditions, that the optimal rate of convergence of their density functions to the standard normal density in the Sobolev space $W^{k,r}(\mathbb{R})$, for every $k \in \mathbb{N} \cup \{0\}$ and $r \in [1,\infty]$, is determined by the maximum of the absolute values of their third and fourth cumulants. 
We also obtain exact asymptotics for this convergence under an additional assumption.
Our approach is based on Malliavin--Stein techniques combined with tools from the theory of generalized functionals in Malliavin calculus.
\end{abstract}

\noindent
\textbf{Keywords:} Breuer--Major theorem, composition of tempered distributions, Edgeworth expansions, exact asymptotics, local central limit theorem, Malliavin calculus, optimal rates of convergence, Stein's method, Wiener chaos.\\
\textbf{2020 Mathematics Subject Classification:} 60F05, 60G15, 60H07.

\tableofcontents

\section{Introduction}
Since the discovery of the so-called Fourth Moment Theorem \cite{NualartPeccati}, the study of normal approximations for nonlinear Gaussian functionals has undergone substantial development.
One of the subsequent milestones was the formulation of what is now known as the Malliavin--Stein approach \cite{NPStein'smethodonWienerchaos}, which merges Stein's method with Malliavin's integration-by-parts formula. 
This framework provides explicit quantitative bounds (including total variation and 1-Wasserstein distance) between a Gaussian functional and the standard normal law, and has become a standard tool for quantitative central limit theorems.
A detailed overview of its many applications can be found in \cite{NourdinPeccati} and in the webpage\footnote{\url{https://sites.google.com/site/malliavinstein/home}} maintained by the second-named author. 

Given the success of this approach, it is natural to ask whether similar techniques can handle stronger forms of convergence than convergence in distribution.
Since Malliavin calculus gives precise conditions for the existence and smoothness of densities, and since convergence in total variation already implies convergence of densities in $L^1(\R)$, it is natural to explore whether one can use this framework to study the \emph{uniform} convergence of such densities.

This direction of research was initiated and thoroughly studied by Hu, Lu, and Nualart \cite{HuLuNualart}. 
They showed that under some nondegeneracy assumptions, estimates analogous to those for the total variation and 1-Wasserstein distances also hold for the uniform norms between the densities of Gaussian functionals and the standard normal law, as well as for their derivatives.
Their work laid the foundation for subsequent studies, and has been applied to spatial averages of stochastic heat and wave equations \cite{KuzgunNualartSHE,KuzgunNualartPAM,sun2025densityconvergencespatialaverage} and to the density convergence in the Breuer--Major theorem \cite{DensityBreuerMajor}. 
Moreover, their approach has been extended to other non-normal target densities, including the Gamma density \cite{bourguin2024noncentrallimitdensitiesfunctionals} and densities of more general distributions \cite{dang2025densityconvergencemarkovdiffusion}.
In view of the success of this approach, research on the convergence of density functions via the Malliavin--Stein method has attracted growing attention in recent years.

The aim of this paper is to further develop this line of research and to establish optimal rates for the convergence of density functions by introducing a new method.
We first review the Malliavin--Stein method and summarize the approaches to density convergence considered in \cite{HuLuNualart} and in the present work, and then state our main results.

\subsection{Overview of the approaches}
The Malliavin--Stein method is built on Stein's method and aims to assess the distance between the law of a Gaussian functional $F$ and the standard normal distribution $\cN(0,1)$. 
Typically, the probabilistic distances covered by this method have the form
\begin{equation}
    d_{\sH}(F, \cN) \coloneqq \sup_{h \in \sH}\left|\E[h(F)] - \E[h(\cN)]\right|, \label{d_H}
\end{equation}
where $\cN$ is a standard normal random variable (\textit{i.e.}, $\cN \sim \cN(0,1)$) and $\sH$ is an appropriate class of test functions. 
Note that this type of distances includes the total variation and 1-Wasserstein distances (see \textit{e.g.}, \cite[Appendix C]{NourdinPeccati}). 
To estimate \eqref{d_H}, the starting point of the method is to consider \textit{Stein's equation} for each test function $h \in \sH$ 
\begin{equation}
    f'(x) - xf(x) = h(x) - \E[h(\cN)], \label{Stein equation}
\end{equation} 
and to express \eqref{d_H} as 
\begin{equation}
    d_{\sH}(F, \cN) = \sup_{h \in \sH}\left|\E[f_h'(F) - Ff_{h}(F)]\right| \label{reexp d_H}
\end{equation}
by using an appropriate solution $f_h$ to \eqref{Stein equation}. 
We refer the reader to \cite{ChenGoldsteinShao} for further details on Stein's method.
Now suppose that $F$ is Malliavin differentiable and has the same mean and variance as $\cN$, and recall that the Malliavin derivative $DF$ is a random variable with values in a suitable Hilbert space $(\fH, \abra{\cdot, \cdot}_{\fH})$. 
The Malliavin--Stein method exploits the integration by parts formula in Malliavin calculus to derive an explicit upper bound of \eqref{d_H}:
\begin{align}
    d_{\sH}(F, \cN) 
    &= \sup_{h \in \sH} \left|\E\left[ f_{h}'(F)(1- \abra{DF, -DL^{-1}F}_{\fH}) \right]\right|\\ 
    &\leq \sup_{h \in \sH}\norm{f_h'}_{\infty} \E\left[\left|1 - \abra{DF, -DL^{-1}F}_{\fH} \right|\right], \label{d_H upper bound}
\end{align}
where $\norm{\cdot}_{\infty}$ denotes the uniform norm on $\R$, and $L^{-1}$ is the pseudo-inverse of the Ornstein--Uhlenbeck operator $L$ (details of the notation are provided in Section \ref{subsection Malliavin calculus}). 
Although the random variable $\abra{DF, -DL^{-1}F}_{\fH}$ is not always easy to handle, it is remarkable that in \eqref{d_H upper bound} the standard normal variable $\cN$ has disappeared, and that the supremum can be separated from the expectation through a uniform control of the norms $\norm{f_h'}$. 
As a result, this bound often becomes more tractable than directly estimating \eqref{d_H} or \eqref{reexp d_H}, and has been applied in various settings.

In \cite{HuLuNualart}, Hu, Lu, and Nualart applied the Malliavin--Stein method to the study of the density convergence for Gaussian functionals. 
Their main tool is a representation formula for densities in Malliavin calculus, and their approach can be summarized as follows. 
Let $F$ be a suitable Gaussian functional whose density function $\rho_F^{}$ is sufficiently smooth. 
Then the representation formula used in \cite{HuLuNualart} yields
\begin{equation}
    \rho_{F}^{(j)}(a) \coloneqq \left. \frac{d^j}{dx^j}\rho_F(x)\right|_{x=a} =  (-1)^j\E[\ind_{(a,\infty)}(F)G_{j+1}(DF)], \qquad a \in \R, \ j \in \N \cup \{0\}, \label{standard density formula}
\end{equation}
where $\ind_{(a,\infty)}$ is the indicator function of $(a, \infty)$ and $G_{j+1}(DF)$ is a random variable that can be expressed in terms of $DF$. 
In the simplest case where $F$ is a standard normal random variable $\cN$ (\textit{i.e.}, $\cN \sim \cN(0,1)$), the formula reduces to 
\begin{equation}
    \rho_{\cN}^{(j)}(a) = (-1)^j\E[\ind_{(a, \infty)}(\cN)H_{j+1}(\cN)],
\end{equation}
where $H_{j+1}$ denotes the $(j+1)$th Hermite polynomial.  
To control $\norm{\rho_{F}^{(j)} - \rho_{\cN}^{(j)}}_{\infty}$, we use a composite bound such as
\begin{align}
    \sup_{a \in \R}\left| \rho_{F}^{(j)}(a) - \rho_{\cN}^{(j)}(a) \right|
    &\leq \E\left[\left|G_{j+1}(DF) - H_{j+1}(F)\right| \right] \\
    &\quad + \sup_{a \in \R}\left| \E[\ind_{(a,\infty)}(F)H_{j+1}(F)] - \E[\ind_{(a,\infty)}(\cN)H_{j+1}(\cN)] \right|.  \label{111}
\end{align} 
Observe that the second term involves only the random variables $F$ and $\cN$. 
Thus, it can be controlled by applying the Malliavin--Stein method, where $\ind_{(a,\infty)}(\cdot)H_{j+1}(\cdot)$ is viewed as a test function. 
Consequently, once the remaining first term is estimated directly, we can derive a uniform upper bound for $\rho_{F}^{(j)} - \rho_{\cN}^{(j)}$.

This approach is simple yet powerful for deriving upper bounds, and it can also be applied to other target densities (see \cite{bourguin2024noncentrallimitdensitiesfunctionals, dang2025densityconvergencemarkovdiffusion}).
However, in this way, the Malliavin--Stein method is only partially applied, and it is often unclear in advance how to estimate the first term on the right-hand side of \eqref{111}, particularly when $F$ is general and $j$ is large.
Moreover, since the estimate \eqref{111} is based on the splitting of the difference of the densities into two parts and on the use of the triangle inequality, this approach is not suitable for obtaining lower bounds and, consequently, for establishing the optimal rate of convergence of density functions.

In this paper, we develop a novel approach to study local central limit theorems for Gaussian functionals, with the aim of further extending the scope of the Malliavin--Stein method.
To evaluate $\rho_{F}^{(j)} - \rho_{\cN}^{(j)}$,  
we also exploit a representation formula for densities as in \cite{HuLuNualart}.
More precisely, our analysis relies on a different representation for $\rho_F^{(j)}$, formally given by 
\begin{equation}
    \rho_{F}^{(j)}(a) = (-1)^{j}\E\left[(\sD^j \delta_a)(F)\right], \qquad a \in \R, \  j \in \N \cup \{0\}. \label{density formula Introduction}
\end{equation}
Here, $\sD$ denotes the distributional derivative, $\delta_a$ denotes the delta function (or the Dirac measure) concentrated at $a \in \R$, and $(\sD^j \delta_a)(F)$ represents the composition of a tempered distribution $\sD^j \delta_a$ with $F$. 
Although this composition is, of course, ill-defined in the classical sense, one can rigorously define it and justify \eqref{density formula Introduction} by using the theory of generalized functionals in Malliavin calculus (see Sections \ref{subsection Malliavin calculus} and \ref{section Composition of tempered distributions with Gaussian functionals} for details, or \cite{IkedaWatanabe, SW-AWF} for a comprehensive account of this representation). 
A key observation here is that, unlike the previous representation \eqref{standard density formula}, the Malliavin derivative $DF$ does not appear inside the expectation in \eqref{density formula Introduction}. This allows us to regard the family $\{(-1)^j (\sD^j\delta_a)\}_{a \in \R}$ as test functions. 
Our approach consists in solving Stein's equation 
\begin{equation}
    \sD T - g T = (-1)^j\sD^j\delta_a - \rho_{\cN}^{(j)}(a) \qquad\qquad \left(\text{where \  $g(x) = x$}\right)  \label{Stein eq introduction}
\end{equation}
for $T$ in the space of tempered distributions $\cS'(\R)$ and to estimate $|\rho_{F}^{(j)}(a) - \rho_{\cN}^{(j)}(a)|$ directly using the Malliavin--Stein machinery, rather than splitting the bound into two parts.
This seemingly small difference nevertheless allows for a more refined analysis.
It is also worth noting that, as will become clear later, many of the arguments and computations developed so far in the Malliavin--Stein approach can be applied to our framework without essential modification. 
This in turn allows a more transparent and unified analysis.

While the idea behind this approach is intuitive, several technical challenges must be addressed to make this approach successful.
As expected, solving \eqref{Stein eq introduction} in $\cS'(\R)$ already is a nontrivial task. 
Even if we could obtain a solution, say $T = f_{a,j}$, certain estimates for $(\sD^kf_{a,j})(F)$ with some $k \in \N$ would still be required to apply the Malliavin--Stein method effectively.
In addition, if one wish to bound $\norm{\rho_{F}^{(j)} - \rho_{\cN}^{(j)}}_{\infty}$, it is necessary to control $(\sD^kf_{a,j})(F)$ uniformly in $a \in \R$.
In the present paper, we provide an explicit solution to \eqref{Stein eq introduction} and exploit estimates related to Sobolev-type spaces for tempered distributions defined via their Hermite coefficients to overcome these difficulties.

We remark that taking $(-1)^j (\sD^j\delta_a)$ as a test function is not crucial, and different choices can be made depending on the purpose.
Even for commonly used test functions (\textit{e.g.}, indicator functions, Lipschitz functions, or smooth functions), one can still apply our framework by deliberately regarding them and the corresponding classical solutions to Stein's equation as tempered distributions. Thus, our approach can be seen as a generalization of the Malliavin--Stein method.

Despite the technical difficulties, a significant advantage of lifting Stein's equation to $\cS'(\R)$ is that it enables us to work with a wide class of test functions, including locally integrable functions with at most polynomial growth, as well as genuine tempered distributions such as $\sD^j\delta_a$. 
In the standard Stein’s method or Malliavin--Stein framework, one needs to ensure that the solution $f_h$ to Stein's equation \eqref{Stein equation} possesses a suitable regularity and that $\sup_{h \in \sH}\norm{f_{h}^{(k)}}_{\infty} < \infty$ for some $k \in \N_0$, typically $k \in \{0,1,2\}$. 
As a result, the choice of the class $\sH$ of test functions is rather restricted.
On the other hand, in our approach, once Stein's equation is solved in $\cS'(\R)$, the solution is, by definition, smooth in the sense of distributional derivatives, so that its regularity is no longer an issue. 
Moreover, it turns out that the distributional derivatives become bounded operators within our framework, and we only need to control the solution in a certain norm, which can be much weaker than the uniform norm. 
For these reasons, we can handle a broader class of test functions $h$ and irregular solutions $f_h$.

The price to pay here is that a Gaussian functional $F$ needs to have much stronger integrability and Malliavin differentiability than usual, and to be nondegenerate in the sense that $\norm{DF}_{\fH}^{-1}$ has finite $p$-moments for some sufficiently large $p \geq 1$. 
In fact, these assumptions are necessary for defining the composition of tempered distributions with $F$ and ensuring the existence and the regularity of the density $\rho_F^{}$.
It should be noted that verifying the nondegeneracy itself is generally a highly nontrivial task, and consequently our approach imposes strong restrictions on the choice of Gaussian functionals. 

In this paper, as a first step in studying local limit theorems based on our approach, we focus on the case of Wiener chaos, where the analysis can be considerably simplified (see Section \ref{subsection Properties of Wiener chaos}).
We emphasize, however, that this approach can also treat more general Gaussian functionals provided the above assumptions are met. 
Moreover, it could potentially be extended to the approximation of multidimensional normal distributions and densities, or even to non-normal cases, and further developments along these lines are left for future research.

\subsection{Main results}
Let $\{F_n\}_{n \in \N}$ be a sequence of random variables, and let $d(\cdot, \cdot)$ be a distance between the laws of real-valued random variables. 
Assume that $d(F_n, \cN) \xrightarrow{n \to \infty} 0$.
In this case, we say that a deterministic sequence $\{\varphi(n)\}_{n \in \N}$ provides an optimal convergence rate for $d(F_n, \cN)$ if there exist constants $0< C_1 < C_2 < \infty$, independent of $n$, such that 
\begin{equation}
    C_1 \varphi(n) \leq d(F_n, \cN) \leq C_2 \varphi(n) \label{optimal rate def}
\end{equation} 
for all sufficiently large $n$. 
Similarly, if $\norm{\rho_{F_n}^{} - \rho_{\cN}^{}} \xrightarrow{n \to \infty} 0$ with respect to some norm and if \eqref{optimal rate def} holds with $d(F_n, \cN)$ replaced by $\norm{\rho_{F_n}^{} - \rho_{\cN}^{}}$, then we say that $\{\varphi(n)\}_{n \in \N}$ is an optimal convergence rate for $\norm{\rho_{F_n}^{} - \rho_{\cN}^{}}$.

If $\{F_n\}_{n \in \N}$ is a sequence in the $m$th Wiener chaos ($m \geq 2$) with unit variance and $F_n \xrightarrow{n \to \infty} \cN(0,1)$, then it is proved in \cite{OptimalBErates} that the optimal convergence rate of 
\begin{equation}
    d_{2}(F_n, \cN) \coloneqq \sup_{h \in C^2, \norm{h''}_{\infty} \leq 1} \left|\E[h(F_n)] - \E[h(\cN)]\right| \xrightarrow{n \to \infty} 0
\end{equation}
is given by the sequence
\begin{equation}
    \mathbf{M}(F_n) \coloneqq \max \{|\E[F_n^3]|, \E[F_n^4] - 3 \} = \max \{|\kappa_3(F_n)|, \kappa_4(F_n)\},
\end{equation}
where $\kappa_3(F_n) = \E[F_n^3]$ and $\kappa_4(F_n) = \E[F_n^4]-3$ are the third and fourth cumulants of $F_n$, respectively.
Moreover, the same optimal convergence rate has also been established in \cite{NPoptimalTV} for the total variation distance.
It should be noted that we have $\kappa_4(F_n) > 0$ and $\lim_{n \to \infty}\mathbf{M}(F_n) = 0$ under this setting. 
Combining our approach with the idea of \cite{OptimalBErates}, we can further show that $\{\mathbf{M}(F_n)\}_{n \in \N}$ is also the optimal convergence rate for the uniform convergence of the density functions and their derivatives, which is our first main result.

Set $\N_0 = \N \cup \{0\}$. 
Let $C^k_{\mathrm{b}}(\R)$ denote the space of $k$-times continuously differentiable functions on $\R$ whose derivatives up to order $k$ are bounded. 
In particular, $C^0_{\mathrm{b}}(\R)$ denotes the space of bounded continuous functions.  
We equip $C^k_{\mathrm{b}}(\R)$ with the norm 
\begin{equation}
    \norm{f}_{C_{\mathrm{b}}^{k}(\R)} = \sum_{j=0}^{k} \norm{f^{(j)}}_{\infty}, \qquad k \in \N_0.  
\end{equation}
Our result reads as follows.
\begin{Thm}\label{Thm Main results 1}
Let $\mathbf{F} = \{F_n\}_{n \in \N}$ be a sequence of random variables living in the $m$th Wiener chaos with $m \geq 2$ such that $\E[F_n^2] = 1$ and $F_n \xrightarrow[n \to \infty]{d} \cN(0,1)$.
Then, for any $j \in \N_0$, there exists $N_{\mathbf{F}, j} \in \N$ such that $\rho_{F_n} \in C_{\mathrm{b}}^{j}(\R)$ and
\begin{equation}
    C_{\mathbf{F}, j,m} \mathbf{M}(F_n) \leq \left\lVert\rho_{F_n}^{(j)} - \rho_{\cN}^{(j)}\right\rVert_{\infty} \leq \wt{C}_{\mathbf{F}, j, m} \mathbf{M}(F_n)
\end{equation}
hold for every $n \geq N_{\mathbf{F}, j}$, where the constants $C_{\mathbf{F}, j, m}$ and $\wt{C}_{\mathbf{F}, j, m}$ depend on $\mathbf{F}$,  $j$, and $m$ but not on $n$.
As a consequence, $\{\mathbf{M}(F_n)\}_{n \in \N}$ provides an optimal convergence rate for $\norm{\rho_{F_n}^{} - \rho_{\cN}^{}}_{C_{\mathrm{b}}^{k}(\R)}$, for every $k \in \N_0$.
\end{Thm}

\begin{Rem}
\begin{enumerate}
    \item[(1)] By the Fourth Moment Theorem, $F_n \to \cN(0,1)$ holds if and only if $\E[F_n^4] \to 3$ in the setting of Theorem \ref{Thm Main results 1}. For other equivalent conditions, see \textit{e.g.} \cite[Theorem 5.2.7]{NourdinPeccati}.
    \item[(2)] Under the assumptions of Theorem \ref{Thm Main results 1}, a previously known result, which can be obtained by combining \cite{HuLuNualart} and \cite{Superconvergence}, shows that  
    \begin{equation}
        \left\lVert\rho_{F_n}^{(j)} - \rho_{\cN}^{(j)}\right\rVert_{\infty} \leq \wt{C}_{\mathbf{F}, j, m} \kappa_4(F_n)^{\frac{1}{2}}
    \end{equation}
    for all $n$ large enough. Thus, Theorem \ref{Thm Main results 1} provides an improved convergence rate, for instance, when $|\kappa_3(F_n)| \leq \kappa_4(F_n)$.
    The sequence $\{\kappa_4(F_n)^{\frac{1}{2}}\}_{n \in \N}$ can also become the optimal convergence rate, precisely when $\kappa_4(F_n)^{\frac{1}{2}}$ and $|\kappa_3(F_n)|$ are of the same order (see \cite{NPexact} for several examples of this situation). 
\end{enumerate}
\end{Rem}

We next turn our attention to the optimal convergence rate in Sobolev spaces. 
Recall that for $k \in \N_0$ and $r \in [1,\infty]$, the Sobolev space $W^{k,r}(\R)$ is a Banach space of (equivalence classes of) functions $f \in L^r(\R)$ whose weak derivatives $f^{(j)}, \ j = 0, \ldots, k,$ belong to $L^r(\R)$, equipped with the norm
\begin{equation}
    \norm{f}_{W^{k,r}(\R)} = \sum_{j=0}^{k}\norm{f^{(j)}}_{L^r(\R)}. 
\end{equation}
In \cite{Superconvergence}, it is shown that if $\{F_n\}_{n \in \N}$ is as in Theorem \ref{Thm Main results 1}, then for every $k \in \N_0$ and $r \in [1,\infty]$, we have $\rho_{F_n}^{} \in W^{k,r}(\R)$ for all $n$ large enough and 
\begin{equation}
    \rho_{F_n}^{} \xrightarrow{n \to \infty} \rho_{\cN}^{} \qquad \text{in \ $W^{k,r}(\R)$}. 
\end{equation}
However, the rate of the convergence remained unknown.
Our second result, together with Theorem \ref{Thm Main results 1}, reveals that the optimal rate of this convergence is determined by the same quantities $\{\mathbf{M}(F_n)\}_{n \in \N}$.

\begin{Thm}\label{Thm Main results 2}
Let $\mathbf{F} = \{F_n\}_{n \in \N}$ be a sequence of random variables living in the $m$th Wiener chaos with $m \geq 2$ such that $\E[F_n^2] = 1$ and $F_n \xrightarrow[n \to \infty]{d} \cN(0,1)$.
Then, for any $j \in \N_0$, there exists $N_{\mathbf{F},j} \in \N$ such that $\rho_{F_n} \in C_{\mathrm{b}}^{j}(\R)$ and
\begin{equation}
    C_{\mathbf{F},j,m,r} \mathbf{M}(F_n) \leq \left\lVert \rho_{F_n}^{(j)} - \rho_{\cN}^{(j)}\right\rVert_{L^r(\R)} \leq \wt{C}_{\mathbf{F},j, m,r} \mathbf{M}(F_n)
\end{equation}
hold for every $n \geq N_{\mathbf{F},j}$ and $r \in [1,\infty)$, where the constants $C_{\mathbf{F},j, m, r}$ and $\wt{C}_{\mathbf{F},j, m, r}$ depend on $\mathbf{F}$, $j$, $m$, and $r$ but not on $n$.
As a consequence, $\{\mathbf{M}(F_n)\}_{n \in \N}$ is an optimal convergence rate for $\norm{\rho_{F_n}^{} - \rho_{\cN}^{}}_{W^{k,r}(\R)}$, for every $k \in \N_0$ and $r \in [1,\infty)$.
\end{Thm}

\begin{Rem}
The result for $j=0$ and $r=1$ is already proved in \cite{NPoptimalTV} by a different argument.  
Since the proof of Theorem \ref{Thm Main results 2} does not rely on this result, it provides an alternative proof. 
\end{Rem}

While Theorems \ref{Thm Main results 1} and \ref{Thm Main results 2} provide the order of $\norm{\rho_{F_n}^{} - \rho_{\cN}^{}}_{W^{k,r}(\R)}$ as $n \to \infty$, it is difficult (in general) to assess the constants $C_{\mathbf{F},j,m,r}$ and $\wt{C}_{\mathbf{F},j, m,r}$ explicitly. 
Our final main result derives the exact asymptotics for the error $\rho_{F_n}^{} - \rho_{\cN}^{}$ under an additional assumption. 
Let $\cN(0,\cC)$ denote the centered Gaussian distribution with covariance matrix $\cC$.

\begin{Thm}\label{Thm Main results 3}
Let $\{F_n\}_{n \in \N}$ be a sequence of random variables belonging to the $m$th Wiener chaos with $m \geq 2$ such that $\E[F_n^2] = 1$. 
Assume that there exists a sequence $\{\varphi(n)\}_{n \in \N}$ such that $\varphi(n) > 0$ for all sufficiently large $n$, $\dis \lim_{n \to \infty}\varphi(n) = 0$, and 
\begin{equation}
    \left(F_n, \frac{1- \abra{DF_n, -DL^{-1}F_n}_{\fH}}{\varphi(n)}\right) \xrightarrow[n \to \infty]{d} \cN\left(0, \begin{pmatrix} 1 & \zeta \\\zeta & \eta \\\end{pmatrix}\right), \qquad \eta \in (0, \infty). 
    \label{two dim conv assump}
\end{equation}
Then, for every $k \in \N_0$ and $r \in [1,\infty]$, we have $\rho_{F_n} \in C_{\mathrm{b}}^k(\R) \cap W^{k,r}(\R)$ for all sufficiently large $n$ and 
\begin{equation}
    \lim_{n \to \infty}  \left\lVert \frac{\rho_{F_n}^{} - \rho_{\cN}^{}}{\varphi(n)} - \frac{\zeta}{3}\rho_{\cN}^{(3)} \right\rVert = 0, \label{THm3 norm conv}
\end{equation}
 where $\norm{\cdot} \in \{\norm{\cdot}_{C_{\mathrm{b}}^k(\R)}, \norm{\cdot}_{W^{k,r}(\R)}\}$.
\end{Thm}

Theorem \ref{Thm Main results 3} is particularly interesting when $\zeta \neq 0$. 
In this case, $\{\varphi(n)\}_{n \in \N}$ becomes an optimal convergence rate for $\norm{\rho_{F_n}^{} - \rho_{\cN}}$, and more precise information about the constants can be obtained. 
We note that when $m$ is odd and \eqref{two dim conv assump} holds, $\zeta$ must be zero (see Remark \ref{Rem end}). 
In practice, one can verify assumption \eqref{two dim conv assump} by applying estimates from the multidimensional Malliavin--Stein method.
For a sequence $\{F_n\}_{n \in \N}$ in a fixed order Wiener chaos, several sufficient conditions under which \eqref{two dim conv assump} holds are known. 
See \cite[Sections 3.2 and 3.3]{NPexact} for details.

\begin{Rem}
In a slightly more general setting, \cite[Theorem 3.1]{NPexact} establishes that \eqref{two dim conv assump} with $\varphi(n) = \sqrt{\E[|1- \abra{DF_n, -DL^{-1}F_n}_{\fH}|^2]}$ and $\eta = 1$ implies
\begin{equation}
    \frac{P(F_n \leq x) - P(\cN \leq x)}{\varphi(n)} \xrightarrow{n \to \infty} \frac{\zeta}{3}\rho_{\cN}^{(2)}(x), \qquad \text{for every $x \in \R$},
\end{equation}
and that the Kolmogorov distance $d_{\mathrm{Kol}}(F_n, \cN)$ is bounded below by $\varphi(n)$ times a positive constant whenever $\zeta \neq 0$ and $n$ is large enough. 
In this result, $\varphi(n)$ is chosen so that $d_{\mathrm{Kol}}(F_n, \cN) \leq \varphi(n)$. 
While our result can be seen as the density analog of this result, Theorem \ref{Thm Main results 3} does not need to specify the sequence $\varphi(n)$ because the convergence \eqref{THm3 norm conv} itself already implies $\norm{\rho_{F_n}^{} - \rho_{\cN}} \lesssim \varphi(n)$ regardless of the value of $\zeta$. 
\end{Rem}

\begin{Rem}
Although our main results focus on a sequence belonging to a fixed order Wiener chaos, we remark that a result analogous to Theorem \ref{Thm Main results 3} (for instance, the pointwise convergence of $\frac{\rho_{F_n}^{}(a) - \rho_{\cN}^{}(a)}{\varphi(n)} \to \frac{\zeta}{3}\rho_{\cN}^{(3)}(a)$) can be established for a more general sequences, provided that they satisfy sufficiently strong uniform integrability and nondegeneracy conditions. 
In contrast, Theorems \ref{Thm Main results 1} and \ref{Thm Main results 2} essentially rely on the specific structure of Wiener chaos, and thus such extensions cannot be expected. 
Nevertheless, an upper bound similar to \eqref{d_H upper bound} for general Gaussian functionals can still be derived under appropriate assumptions using our approach. 
We leave this point in its full generality for future research. 
\end{Rem}

After the introduction of some notation, the rest of the paper is organized as follows. In Section \ref{section Preliminaries}, we recall Hermite polynomials and functions, review fundamental elements of Malliavin calculus, and collect several properties of Wiener chaos that are used throughout the paper. 
Section \ref{section Stein's equation in the space of tempered distributions} discusses Stein's equation \eqref{Stein eq introduction} in the space of tempered distributions and provides its solution explicitly. 
To define the composition of tempered distributions with Gaussian functionals, it is necessary to endow a certain subclass of tempered distributions with an appropriate topology, which is carried out in Section \ref{section Sobolev spaces for tempered distributions}.
In Section \ref{section Composition of tempered distributions with Gaussian functionals}, we define the composition and justify the relation \eqref{density formula Introduction}. 
Finally, Section \ref{section Optimal local central limit theorems} proves Theorems \ref{Thm Main results 1} and \ref{Thm Main results 2}, and Section \ref{section Exact asymptotics} proves Theorem \ref{Thm Main results 3}.

\vskip\baselineskip
\noindent
\textbf{Notation.}
Let $\N = \{1,2, \ldots \}$ and $\N_0 \coloneqq \N \cup \{0\}$. 
The indicator function of a Borel set $A \subset \R$ is written as $\ind_{A}$. 
For $a, b \in \R$, set $a \land b = \min \{a, b\}$ and $a \lor b = \max \{a, b\}$.
As usual, the symbols $\floor{a}$ and $\ceil{a}$ stand for the floor and ceiling of $a \in \R$, respectively.

We use the notation $a \lesssim b$ to mean that $a \leq C b$ for some constant $C > 0$. Sometimes we write $a \lesssim_{Q_1, \ldots, Q_L} b$ for some $L$ quantities $Q_1, \ldots, Q_L \ (L \in \N)$ to emphasize that the implicit constant $C$ depends on $Q_1, \ldots, Q_L$.

The symbol $\cN$ always denotes a standard normal random variable on the underlying probability space (\textit{cf.} Section \ref{subsection Malliavin calculus}). 
For a random variable $F$, its density function, if exists, is denoted by $\rho_{F}$. 

Let $\cS(\R)$ and $\cS'(\R)$ denote the space of rapidly decreasing smooth functions and the space of tempered distributions, respectively.  
The dual pairing of $T \in \cS'(\R)$ and $\varphi \in \cS(\R)$ is denoted by $\dpair{\cS'(\R)}{T}{\varphi}{\cS(\R)}$.

\vspace{2mm}
\noindent
\textbf{Acknowledgment.}
The first-named author thanks his PhD supervisor, Professor Seiichiro Kusuoka, for valuable comments on an earlier version of this work. 
He is also deeply grateful to the second- and third-named authors for their kind hospitality during his stay at the University of Luxembourg in October--November 2024, where part of this research was carried out.
M.E. was supported by the Japan Society for the Promotion of Science (JSPS), KAKENHI Grant Numbers JP22J21604, JP22H05105.
I.N. and G.P. are supported by the Luxembourg National Research Fund (Grants: O22/17372844/FraMStA and O24/18972745/GFRF).

\section{Preliminaries}\label{section Preliminaries}
In this section, we review the Hermite polynomials, functions, and basic elements of Malliavin calculus that will be used throughout this paper.

\subsection{Hermite polynomials and functions}
For $n \in \N_0$, the $n$th Hermite polynomial is given by 
\begin{equation}
    H_n(x) \coloneqq (-1)^ne^{\frac{x^2}{2}}\frac{d^n}{dx^n}e^{-\frac{x^2}{2}}, \quad x \in \R.
\end{equation}
The following properties of $H_n$ are well known and will be used in this paper. 
See \textit{e.g.} \cite[Chapter 1]{NourdinPeccati}.
\begin{Lem}
\label{Lem Hermite polynomials properties}
For every $n, m \in \N_0$, the following holds. 
\begin{enumerate}[\normalfont(i)]
    \item $\dis H_{n+1}'(x) = (n+1)H_n(x)$ \  and \  $\dis \left(-\frac{d}{dx}\right)\left(H_n(x)e^{-\frac{x^2}{2}}\right) = H_{n+1}(x)e^{-\frac{x^2}{2}}$.
    \item $\dis \int_{\R}H_{n}(x)H_{m}(x)\rho_{\cN}^{}(x)dx = n!\delta_{nm}$, where $\delta_{nm}$ is the Kronecker delta.
    \item $\rho_{\cN}^{(n)}(x) = (-1)^{n}H_{n}(x)\rho_{\cN}^{}(x)$. 
    \item $\dis H_m(x)H_n(x) = \sum_{j=0}^{m \land n} \binom{m}{j}\binom{n}{j}j!H_{m+n-2j}(x)$.
\end{enumerate}
\end{Lem}

We recall that the Hermite functions $\{\phi_n\}_{n \in \N_0}$ are defined as follows: 
\begin{equation}
    \phi_n(x) \coloneqq (-1)^n(\sqrt{\pi} 2^n n!)^{-\frac{1}{2}}e^{\frac{x^2}{2}}\frac{d^n}{dx^n}e^{-x^2}, \quad x \in \R, \ n \in \N_0. 
\end{equation}
For convenience, we set $\phi_n \equiv 0$ if $n \in \Z \setminus \N_0$.
It is well known that $\{\phi_n\}_{n \in \N_0}$ is a complete orthonormal system in $L^2(\R)$. 
Moreover, $\phi_n$ satisfies the following properties.

\begin{Lem}\label{Lem 1d Hermite functions property}
For every $n \in \N_0$, the following holds.
\begin{enumerate}[\normalfont(i)]
    \item $\dis \left(x^2-\frac{d^2}{dx^2}\right)\phi_n(x) = (2n+1)\phi_n(x)$.
    \item $\dis \phi_n'(x) = \sqrt{\frac{n}{2}}\phi_{n-1}(x) - \sqrt{\frac{n+1}{2}}\phi_{n+1}(x)$ \quad \text{and} \quad $\dis x\phi_n(x) = \sqrt{\frac{n}{2}}\phi_{n-1}(x) + \sqrt{\frac{n+1}{2}}\phi_{n+1}(x)$
    \item $\dis \int_{\R}e^{-2\pi \sqrt{-1}\xi x}\phi_n(x)dx = \sqrt{2\pi}(-\sqrt{-1})^n\phi_n(\xi)$.
    \item As $n \to \infty$, $\phi_n(x) = O\left(n^{-\frac{1}{4}}\right)$ for each $x \in \R$, $\norm{\phi_n}_{L^{\infty}(\R)} = O\left(n^{-\frac{1}{12}}\right)$, and $\norm{\phi_n}_{L^1(\R)} = O\left(n^{\frac{1}{4}}\right)$. 
\end{enumerate}
\end{Lem}

Equations (i), (ii), and (iii) of Lemma \ref{Lem 1d Hermite functions property} follow from straightforward calculations.
Estimates (iv) can be found in \cite[Section 21.3, p.571]{HillePhillips}.

We also state the following lemma for later use.

\begin{Lem}\label{Lem Hermite integral}
For any $\alpha \geq 0$, we have 
\begin{align}
    \int_{\R}\phi_n(x)e^{-\alpha x^2}dx =
    \begin{cases}
        \dis0, &\quad \text{if \  $n$ is odd},\\[1ex]
        \dis \frac{2^{\frac{1}{2}}\pi^{\frac{1}{4}}(n!)^{\frac{1}{2}}}{2^{\frac{n}{2}}\left(\frac{n}{2}\right)!} \sqrt{\frac{1}{1+2\alpha}}\left(\frac{1-2\alpha}{1+2\alpha}\right)^{\frac{n}{2}}, &\quad \text{if \  $n$ is even}.
    \end{cases}\label{int w m}
\end{align}
\end{Lem}
\begin{proof}
It is known that 
\begin{equation}
    e^{2xt - t^2} = \sum_{n=0}^{\infty}\wt{H_n}(x)\frac{t^n}{n!}, \qquad \text{where \  $\wt{H_n}(x) = (-1)^n e^{x^2}\frac{d^n}{dx^n}e^{-x^2},$}
\end{equation}
and it follows from Fubini's theorem that for any $\lambda > 0$, 
\begin{equation}
    \int_{\R}e^{- \lambda x^2 + 2xt - t^2}dx = \sum_{n=0}^{\infty}\int_{\R}\wt{H_n}(x)e^{-\lambda x^2}dx\frac{t^n}{n
    !}.
\end{equation}
A simple calculation yields 
\begin{equation}
    \int_{\R}e^{-\lambda x^2 + 2xt - t^2}dx = \sqrt{\frac{\pi}{\lambda}}e^{\frac{t^2(1-\lambda)}{\lambda}} = \sum_{k=0}^{\infty}\sqrt{\frac{\pi}{\lambda}}\left(\frac{1-\lambda}{\lambda}\right)^{k}\frac{(2k)!}{k!}\frac{t^{2k}}{(2k)!},
\end{equation}
and hence 
\begin{align}
    \int_{\R}\wt{H_n}(x)e^{-\lambda x^2}dx =
    \begin{cases}
        \dis 0, &\quad \text{if $n$ is odd},\\[1ex]
        \dis \sqrt{\frac{\pi}{\lambda}}\left(\frac{1-\lambda}{\lambda}\right)^{\frac{n}{2}}\frac{n!}{(\frac{n}{2})!}, &\quad \text{if $n$ is even}.
    \end{cases}
\end{align}
Therefore, 
\begin{align}
    \int_{\R}\phi_n(x)e^{-\alpha x^2}dx 
    &= (\sqrt{\pi}2^n n!)^{-\frac{1}{2}}\int_{\R}\wt{H_n}(x)e^{-\frac{1+2\alpha}{2}x^2}dx\\
    &= 
    \begin{cases}
        \dis0, &\quad \text{if $n$ is odd},\\[1ex]
        \dis \frac{2^{\frac{1}{2}}\pi^{\frac{1}{4}}(n!)^{\frac{1}{2}}}{2^{\frac{n}{2}}\left(\frac{n}{2}\right)!} \sqrt{\frac{1}{1+2\alpha}}\left(\frac{1-2\alpha}{1+2\alpha}\right)^{\frac{n}{2}}, &\quad \text{if $n$ is even}.
    \end{cases}
\end{align}
\end{proof}

\subsection{Malliavin calculus}\label{subsection Malliavin calculus}
In this section, we recall some basic notions and tools of Malliavin calculus that will be used in later sections. 
For more details on Malliavin calculus, the reader is referred to \cite{IkedaWatanabe, Shigekawa, Nualartbook, NourdinPeccati}. 

Let $\fH$ be a real separable Hilbert space, and let $X = \{X(h) \mid h \in \fH\}$ denote an isonormal Gaussian process on a suitable complete probability space $(\Omega, \sF, P)$. 
Recall that $X$ is a centered Gaussian process such that for any $h, \wt{h} \in \fH$, $\E[X(h)X(\wt{h})] = \langle h, \wt{h} \rangle_{\fH}$.
Throughout the paper, we shall assume that $\sF$ is generated by $X$ and work on this probability space. 

For $m \in \N_0$ and a real separable Hilbert space $\fE$, we write $\cW_m(\fE)$ for the $m$th $\fE$-valued Wiener chaos of $X$. 
By definition, $\cW_m(\fE)$ is the closed linear subspace of $L^2(\Omega; \fE)$ generated by 
\begin{equation}
    \{H_m(X(h))e \mid h \in \fH, \norm{h}_{\fH} = 1, \  e \in \fE \},
\end{equation}
where $H_m$ is the $m$th Hermite polynomial. 
By the orthogonality of Hermite polynomials, $\cW_m(\fE)$ and $\cW_n(\fE)$ are orthogonal in $L^2(\Omega; \fE)$ whenever $m \neq n$.
The importance of Wiener chaos in Malliavin calculus stems from the Wiener--It\^{o} decomposition
\begin{equation}
    L^2(\Omega; \fE) = \bigoplus_{m=0}^{\infty}\cW_m(\fE).
\end{equation}
Such an expression means that every $F \in L^2(\Omega; \fE)$ can be uniquely expanded as
\begin{equation}
    F = \sum_{m=0}^{\infty}F_m,
\end{equation}
where $F_m \in \cW_m(\fE)$ and the series converges in $L^2(\Omega;\fE)$. 
In particular, $F_0 = \E[F] \in \fE$.
The orthogonal projection onto $\cW_m(\fE)$ in $L^2(\Omega;\fE)$ will be denoted by $J_m^{\fE}$.
To simplify notation, we write $\cW_m$, $L^2(\Omega)$, $J_m$ instead of $\cW_m(\R)$, $L^2(\Omega;\R)$, $J_m^{\R}$, and we simply call $\cW_m$ the $m$th Wiener chaos associated with $X$.

Let $\cP$ denote the set of all real-valued random variables of the form $f(X(h_1), \ldots, X(h_m))$, where $m \in \N$, $h_1, \ldots, h_m \in \fH$, and $f \colon \R^m \to \R$ is a polynomial function of $m$-variables.
More generally, for a real separable Hilbert space $\fE$, we write $\cP(\fE)$ for the set of all $\fE$-valued random variables of the form $F_1 e_1 + \cdots + F_m e_m$, where $m \in \N$, $F_1, \ldots, F_m \in \cP$, and $e_1, \ldots, e_m \in \fE$.
It is well known that $\cP(\fE)$ is dense in $L^p(\Omega;\fE)$ for every $p \in [1, \infty)$. (See \textit{e.g.} \cite[Exercise 1.1.7]{Nualartbook} for the case $\fE = \R$. 
The density for general separable Hilbert space case follows by the same argument.)

The Ornstein--Uhlenbeck operator $L$ and its pseudo-inverse $L^{-1}$ on $\cP(\fE)$ into itself are the linear operators defined by
\begin{equation}
    LF = -\sum_{n=1}^{\infty} n J_n^{\fE}F \quad \text{and} \quad L^{-1}F = -\sum_{n=1}^{\infty}\frac{1}{n}J_n^{\fE}F, \qquad F \in \cP(\fE).
\end{equation}
Since $F \in \cP(\fE)$, both right-hand sides are actually finite sums, and we have
\begin{equation}
    LL^{-1}F = L^{-1}LF = F - \E[F], \qquad F \in \cP(\fE).
\end{equation}

For $\alpha \in \R$ and $p \in [1, \infty)$, let $\bW^{\alpha, p}(\fE)$ denote the Sobolev space obtained by completing $\cP(\fE)$ with respect to the norm 
\begin{equation}
    \norm{F}_{\bW^{\alpha, p}(\fE)} \coloneqq \norm{(I-L)^{\frac{\alpha}{2}}F}_{L^p(\Omega; \fE)},
\end{equation}
where $(I-L)^{\frac{\alpha}{2}}$ is the linear operator on $\cP(\fE)$ given by 
\begin{equation}
    (I-L)^{\frac{\alpha}{2}}F = \sum_{n=0}^{\infty} (n+1)^{\frac{\alpha}{2}}J_n^{\fE}F, \qquad F \in \cP(\fE).
\end{equation}
It is easily seen that $\bW^{0,p}(\fE) = L^p(\Omega;\fE)$, and that if $1 \leq p_1 \leq p_2 < \infty$ and $\alpha_1, \alpha_2 \in \R$ with $\alpha_1 \leq \alpha_2$, then 
\begin{equation}
    \bW^{\alpha_2, p_2}(\fE) \subset \bW^{\alpha_1, p_1}(\fE) \quad \text{and} \quad \norm{F}_{\bW^{\alpha_1, p_1}(\fE)} \leq \norm{F}_{\bW^{\alpha_2, p_2}(\fE)}, \quad F \in \bW^{\alpha_2, p_2}(\fE).
\end{equation}
When $\alpha \in [0,\infty)$ and $p \in (1, \infty)$, the dual space of $\bW^{\alpha, p}(\fE)$ can be naturally identified with $\bW^{-\alpha, \frac{p}{p-1}}(\fE)$.
Under the identification, if $F \in \bW^{\alpha,p}(\fE)$ and $G \in L^{\frac{p}{p-1}}(\Omega;\fE) \subset \bW^{-\alpha, \frac{p}{p-1}}(\fE)$, we have
\begin{equation}
    \dpair{\bW^{\alpha,p}(\fE)}{F}{G}{\bW^{-\alpha, \frac{p}{p-1}}(\fE)} = \dpair{\bW^{-\alpha, \frac{p}{p-1}}(\fE)}{G}{F}{\bW^{\alpha,p}(\fE)} = \E[\abra{G, F}_{\fE}],
\end{equation}
where $\dpair{\bW^{-\alpha, \frac{p}{p-1}}(\fE)}{\cdot}{\cdot}{\bW^{\alpha,p}(\fE)}$ and $\dpair{\bW^{\alpha,p}(\fE)}{\cdot}{\cdot}{\bW^{-\alpha, \frac{p}{p-1}}(\fE)}$ denote the dual pairing between $\bW^{\alpha,p}(\fE)$ and $\bW^{-\alpha, \frac{p}{p-1}}(\fE)$. 
It should be noted that if $\alpha < 0$, an element of $\bW^{\alpha, p}(\fE)$ may not be an $\fE$-valued random variable on the probability space $(\Omega, \sF, P)$. 
Nonetheless, it can be interpreted as a generalized random variable through the above duality. 

The following result can be used to identify elements of $\bW^{\alpha, p}(\fE)$ for $\alpha > 0$ and $p \in (1,\infty)$.
\begin{Lem}[\textit{cf.} {\cite[Lemma 1.5.3]{Nualartbook}}]
\label{Lem criterion for elements in Sobolev}
Let $\alpha > 0$ and $p \in (1,\infty)$. 
Let $\fE$ be a real separable Hilbert space. 
If a sequence $\{F_n\}_{n \in \N} \subset \bW^{\alpha, p}(\fE)$ converges weakly to $F$ in $L^p(\Omega; \fE)$ and satisfies
\begin{equation}
    \sup_{n \in \N} \norm{F_n}_{\bW^{\alpha, p}(\fE)} < \infty,
\end{equation}
then $F \in \bW^{\alpha, p}(\fE)$ and $F_n$ converges weakly to $F$ in $\bW^{\alpha,p}(\fE)$.
In particular, 
\begin{equation}
    \norm{F}_{\bW^{\alpha, p}(\fE)} \leq \liminf_{n \to \infty} \norm{F_{n}}_{\bW^{\alpha,p}(\fE)} \leq \sup_{n \in \N} \norm{F_n}_{\bW^{\alpha, p}(\fE)}.
\end{equation}
\end{Lem}

\begin{proof}
Observe that the operator $(I - L)^{-\frac{\alpha}{2}}$, extended from $\cP(\fE)$ to $L^p(\Omega; \fE)$, is an isometric isomorphism from $L^p(\Omega; \fE)$ to $\bW^{\alpha, p}(\fE)$. 
Consequently, $\bW^{\alpha, p}(\fE)$ is a reflexive Banach space. 
Thus, there is a subsequence $\{F_{n_k}\}_{k \in \N}$ such that $F_{n_k}$ converges weakly to some $\wt{F}$ in $\bW^{\alpha, p}(\fE)$, and it holds that for any $G \in L^{\frac{p}{p-1}}(\Omega;\fE) \subset \bW^{-\alpha, \frac{p}{p-1}}(\fE)$, $\lim_{k \to \infty}\E[\langle F_{n_k}, G\rangle_{\fE}] = \E[\langle \wt{F}, G \rangle_{\fE}]$.
Since $F_{n_k}$ converges weakly to $F$ in $L^p(\Omega; \fE)$, $\wt{F}$ coincides with $F$ in $L^p(\Omega;\fE)$, and therefore, $F \in \bW^{\alpha, p}(\fE)$.
The same argument shows that for any subsequence of $\{F_n\}_{n \in \N}$, we can further extract a subsequence converging weakly to $F$ in $\bW^{\alpha, p}(\fE)$. 
It follows that $F_n$ converges weakly to $F$ in $\bW^{\alpha, p}(\fE)$, and the proof is complete.
\end{proof}

Set 
\begin{equation}
     \bW^{\infty}(\fE) \coloneqq \bigcap_{\alpha > 0} \bigcap_{1 < p  <\infty} \bW^{\alpha, p}(\fE) \quad \text{and} \quad \bW^{-\infty}(\fE) \coloneqq \bigcup_{\alpha > 0} \bigcup_{1 < p < \infty} \bW^{-\alpha, p}(\fE).
\end{equation}
Then $\bW^{-\infty}(\fE)$ can be regarded as the dual space of $\bW^{\infty}(\fE)$.
The dual pairing $\dpair{\bW^{-\infty}(\fE)}{\cdot}{\cdot}{\bW^{\infty}(\fE)}$ is defined in such a way that if $F \in \bW^{\infty}(\fE)$ and $G \in \bW^{-\alpha, p}(\fE) \subset \bW^{-\infty}(\fE)$ for some $\alpha > 0$ and $p \in (1, \infty)$, then 
\begin{equation}
    \dpair{\bW^{-\infty}(\fE)}{G}{F}{\bW^{\infty}(\fE)} \coloneqq \dpair{\bW^{-\alpha, p}(\fE)}{G}{F}{\bW^{\alpha, \frac{p}{p-1}}(\fE)}. 
\end{equation}
In particular, the pairing $\dpair{\bW^{-\infty}}{G}{1}{\bW^{\infty}}$ with the constant function $1 \in \bW^{\infty}$ is called the generalized expectation of $G \in \bW^{-\infty}$, as it coincides with the standard expectation whenever $G \in L^p(\Omega)$ for some $p \in (1, \infty)$.
For our analysis based on the duality, it is convenient to work with the Sobolev spaces $\bW^{\alpha, p}(\fE)$, but the operator $(I - L)^{\frac{\alpha}{2}}$ is difficult to handle in actual computations. For this reason, we also introduce Sobolev spaces associated with the Malliavin derivative.

The Malliavin derivative on $\cP(\fE)$ is the linear operator $D \colon \cP(\fE) \to \cP(\fH \otimes \fE)$ such that for $F \coloneqq \sum_{i=1}^{n}f_{i}(X(h_1), \ldots, X(h_m))e_i \in \cP(\fE)$,  
\begin{equation}
    DF = \sum_{i=1}^{n}\sum_{j=1}^{m}\frac{\partial f_i}{\partial x_j}(X(h_1), \ldots, X(h_m)) h_j \otimes e_i.
\end{equation}
For each integer $k \geq 2$, the $k$th Malliavin derivative $D^k \colon \cP(\fE) \to \cP(\fH^{\otimes k} \otimes \fE)$ can be defined successively in the same way. 
The corresponding Sobolev spaces $\bD^{k,p}(\fE)$, defined for $k \in \N_0$ and $p \in [1, \infty)$, are given by the completion of $\cP(\fE)$ with respect to the norm
\begin{equation}
    \norm{F}_{\bD^{k,p}(\fE)} \coloneqq \left(\norm{F}_{L^p(\Omega;\fE)}^p + \sum_{i=1}^{k}\norm{D^iF}_{L^p(\Omega;\fH^{\otimes i} \otimes \fE)}^p\right)^{\frac{1}{p}}.
\end{equation}
Clearly, $\bD^{0,p}(\fE) = L^p(\Omega;\fE)$, and one can verify that for any $1 \leq p_1 \leq p_2 < \infty$ and $k_1, k_2 \in \N_0$ with $k_1 \leq k_2$,
\begin{equation}
    \bD^{k_2, p_2}(\fE) \subset \bD^{k_1, p_1}(\fE) \quad \text{and} \quad \norm{F}_{\bD^{k_1,p_1}(\fE)} \lesssim_{k_1, p_1, p_2} \norm{F}_{\bD^{k_2, p_2}(\fE)}, \quad F \in \bD^{k_2, p_2}(\fE).
\end{equation}
Moreover, it is well known that the two Sobolev norms $\norm{\cdot}_{\bD^{k,p}(\fE)}$ and $\norm{\cdot}_{\bW^{k,p}(\fE)}$ are equivalent on $\cP(\fE)$ for any $k \in \N_0$ and $p \in (1, \infty)$ by Meyer's inequality (see \textit{e.g.} \cite[Section 1.5]{Nualartbook}), and consequently, 
\begin{equation}
    \bD^{k, p}(\fE) = \bW^{k, p}(\fE) \qquad \text{for every $k \in \N_0$ and $p \in (1, \infty)$}.
\end{equation}
Note that whether the equivalence holds for $p = 1$ remains unknown.
From now on, we shall write $\bD^{k, p}$ and $\norm{\cdot}_{\bD^{k, p}}$ for $\bD^{k,p}(\R)$ and $\norm{\cdot}_{\bD^{k, p}(\R)}$, respectively.

The operator $D \colon L^p(\Omega; \fE) \to L^p(\Omega;\fH \otimes \fE)$, initially defined on $\cP(\fE)$, is closable for all $p \in [1, \infty)$ and extends to a closed operator with domain $\bD^{1, p}(\fE)$ by taking its closure. 
On the other hand, the Malliavin derivative can also be extended to a continuous operator $D \colon \bW^{\alpha +1, p}(\fE) \to \bW^{\alpha, p}(\fH \otimes \fE)$ for every $\alpha \in \R$ and $p \in (1,\infty)$. 
In fact, it can be extended uniquely to $D \colon \bW^{-\infty}(\fE) \to \bW^{-\infty}(\fH \otimes \fE)$ so that its restriction $D \colon \bW^{\alpha + 1, p}(\fE) \to \bW^{\alpha, p}(\fH \otimes \fE)$ is continuous for all $\alpha \in \R$ and $p \in (1, \infty)$ (see \cite[Chapter V, Theorem 8.5]{IkedaWatanabe}). 
The Malliavin derivative is more tractable than $(I-L)^{\frac{\alpha}{2}}$ partly because the following chain rule holds: 
Let $F \in \bD^{1,p}$ for some $p \in [1, \infty)$ and let $\psi \colon \R \to \R$ be a continuously differentiable function with a bounded derivative.
Then $\psi(F) \in \bD^{1,p}$ and 
\begin{equation}
    D\psi(F) = \psi'(F)DF.
\end{equation}
Moreover, if $F, G \in \cP$ and $\fE_1$ and $\fE_2$ are real separable Hilbert spaces, the Leibniz rule 
\begin{equation}
    D(F e_1 \otimes G e_2) = GDF \otimes e_1 \otimes e_2  + F DG \otimes e_1 \otimes e_2 \qquad e_1 \in \fE_1, \ e_2 \in \fE_2,
\end{equation}
holds in $\cP(\fH \otimes \fE_1 \otimes \fE_2)$.
From this, we can further obtain the following estimate.
\begin{Lem}[\textit{cf.} {\cite[Chapter V, Proposition 8.8]{IkedaWatanabe}}]\label{Lem Sobolev holder ineq}
Let $k \in \N_0$, and let $\fE_1$ and $\fE_2$ be real separable Hilbert spaces. 
For any $F \in \bD^{k, p}(\fE_1)$ and $G \in \bD^{k,q}(\fE_2)$ with $\frac{1}{p} + \frac{1}{q} \eqqcolon \frac{1}{r} < 1$, we have $F \otimes G \in \bD^{k,r}(\fE_1 \otimes \fE_2)$ and 
\begin{equation}
    \norm{F \otimes G}_{\bD^{k,r}(\fE_1 \otimes \fE_2)} \lesssim_{k,p,q} \norm{F}_{\bD^{k,p}(\fE_1)} \norm{G}_{\bD^{k,q}(\fE_2)}.
\end{equation}
\end{Lem}
\begin{proof}
The claim holds when $F \in \cP(\fE_1)$ and $G \in \cP(\fE_2)$ (see \cite[Chapter V, Proposition 8.8]{IkedaWatanabe}). 
The desired conclusion follows from an approximation argument together with Lemma \ref{Lem criterion for elements in Sobolev}.
\end{proof}

The following integration by parts formula will be used in Section \ref{section Exact asymptotics}.
See \cite[Theorem 2.9.1]{NourdinPeccati} for the proof.
\begin{Lem}\label{Lem NP IBP}
Let $F, G \in \bD^{1,2}$. 
Let $\psi \colon \R \to \R$ be a continuously differentiable function with a bounded derivative. 
Then 
\begin{equation}
    \E[F\psi(G)] = \E[F]\E[\psi(G)] + \E[\psi'(G)\abra{DG, -DL^{-1}F}_{\fH}].
\end{equation}
Here, $L^{-1}$ is extended to a contraction operator on $L^2(\Omega)$.
\end{Lem}

We now recall the dual operator of $D$.
Let $D^{\ast} \colon \cP(\fH \otimes \fE) \to \cP(\fE)$ be a linear operator defined as follows: for $ F \coloneqq \sum_{i=1}^{m}\sum_{j=1}^{n} F_{ij}h_{j}\otimes e_i \in \cP(\fH \otimes \fE)$, 
\begin{equation}
    D^{\ast}F = \sum_{i=1}^{m} \sum_{j=1}^{n}\left\{F_{ij}X(h_j) - \abra{DF_{ij}, h_j}_{\fH}\right\}e_i.
\end{equation}
Then, it can be verified that for any $F \in \cP(\fE)$ and $G \in \cP(\fH \otimes \fE)$, 
\begin{equation}
\label{L = -D^*D}
    LF = - D^{\ast}DF
\end{equation}
and
\begin{equation}
\label{D D^* duality}
    \E[\abra{DF, G}_{\fH \otimes \fE}] = \E[\abra{F, D^{\ast}G}_{\fE}].
\end{equation}
Similar to the Malliavin derivative, $D^{\ast} \colon \cP(\fH \otimes \fE) \to \cP(\fE)$ can be extended uniquely to $D^{\ast} \colon \bW^{-\infty}(\fH \otimes \fE) \to \bW^{-\infty}(\fE)$ so that its restriction $D^{\ast} \colon \bW^{\alpha + 1, p}(\fH \otimes \fE) \to \bW^{\alpha, p}(\fE)$ is continuous for all $\alpha \in \R$ and $p \in (1, \infty)$ (see \cite[Chapter V, Corollary of Theorem 8.5]{IkedaWatanabe}). 
With the extensions of $D$ and $D^{\ast}$, the operator $L$ also admits a unique extension $L \colon \bW^{-\infty}(\fE) \to \bW^{-\infty}(\fE)$, whose restriction $L \colon \bW^{\alpha + 2, p}(\fE) \to \bW^{\alpha, p}(\fE)$ is continuous for every $\alpha \in \R$ and $p \in (1,\infty)$.
Moreover, by approximation, the duality relation \eqref{D D^* duality} extends to the case where $F$ and $G$ belong to appropriate Sobolev spaces.
In particular, $D^{\ast} \colon \bW^{-\alpha, \frac{p}{p-1}}(\fH \otimes \fE) \to \bW^{-\alpha -1, \frac{p}{p-1}}(\fE)$ is the dual operator of $D \colon \bW^{\alpha+1, p}(\fE) \to \bW^{\alpha, p}(\fH \otimes \fE)$ for every $\alpha \in \R$ and $p \in (1,\infty)$.

\begin{Rem}
On the space $L^p(\Omega;\fE)$ with $p \in (1, \infty)$, the Malliavin derivative can be regarded either as a bounded operator $D \colon L^p(\Omega;\fE) \to \bW^{-1, p}(\fH \otimes \fE)$, or as a closed operator $D \colon L^p(\Omega;\fE) \to L^p(\Omega;\fH \otimes \fE)$ with $\Dom(D) = \bD^{1,p}(\fE)$. 
In the latter case, we can also consider its dual operator $\delta \colon L^{\frac{p}{p-1}}(\Omega; \fH \otimes \fE) \to L^{\frac{p}{p-1}}(\Omega;\fE)$.
It should be noted that $\delta$ is a closed operator but not continuous. 
We mainly use $D^{\ast}$ in this paper to make use of its continuity.
\end{Rem}

\subsection{Properties of Wiener chaos} \label{subsection Properties of Wiener chaos}
This section reviews some properties of random variables in Wiener chaos. 
We shall focus on the real-valued setting.

Let $F$ belong to the $m$th Wiener chaos $\cW_m$ for some $m \in \N_0$.
Then $F \in \bW^{\infty}$, and, by the orthogonality, we have $E[F] = 0$ whenever $m \in \N$.
On Wiener chaos, the following equivalence of Sobolev norms is well known. 
\begin{Lem}
\label{Lem Norm equivalence}
For each $m \in \N_0$, all $\bD^{k,p}$-norms and $\bW^{\alpha, p}$-norms, where $k \in \N_0$, $\alpha \in [0,\infty)$, and $p \in [1,\infty)$, are equivalent on $\bigoplus_{i=0}^{m}\cW_i$. 
\end{Lem}
\begin{proof}
Since all $L^p(\Omega)$-norms with $p \in [1, \infty)$ are equivalent on $\bigoplus_{i=0}^{m}\cW_i$ (\textit{cf.} \cite[Section 3.1.3]{Superconvergence}), we find that for any $F = \sum_{i=0}^{m}F_i$ with $F_i \in \cW_i$, 
\begin{equation}
    \norm{F}_{L^2(\Omega)} \lesssim \norm{F}_{L^1(\Omega)} \leq \norm{F}_{\bW^{\alpha, p}} 
\end{equation}
and 
\begin{equation}
    \norm{F}_{\bW^{\alpha, p}} \lesssim (1+m)^{\frac{\alpha}{2}}\sum_{i=0}^{m}\norm{F_i}_{L^2(\Omega)} = (1+m)^{\frac{\alpha + 1}{2}}\norm{F}_{L^2(\Omega)},
\end{equation}
where the last equality follows from the orthogonality of Wiener chaos in $L^2(\Omega)$. 
From these, all $\bW^{\alpha, p}$-norms with $\alpha \in [0,\infty)$ and $p \in [1,\infty)$ are equivalent to $L^2(\Omega)$-norm. 
On the other hand, the equivalence of $L^p(\Omega)$-norms and Meyer's inequality yield
\begin{equation}
    \norm{F}_{L^2(\Omega)} \lesssim \norm{F}_{L^1(\Omega)} \leq \norm{F}_{\bD^{k, p}} \leq \norm{F}_{\bD^{k, p+1}} \lesssim \norm{F}_{\bW^{k, p+1}} \lesssim \norm{F}_{L^2(\Omega)},
\end{equation}
and hence all $\bD^{k, p}$-norms with $k \in \N_0$ and $p \in [1, \infty)$ are equivalent to $L^2(\Omega)$-norm as well. 
Therefore, the lemma follows.
\end{proof}
\begin{Rem}
Lemma \ref{Lem Norm equivalence} includes the case $p = 1$, but this of course does not ensure that Meyer's equivalence holds for $p=1$, since the implicit constants generally depend on $m$. 

\end{Rem}

We now consider a sequence belonging to $\bigoplus_{i=0}^{m}\cW_i$ for some $m \in \N_0$. 
It is known that the tightness of the sequence implies uniform bounds on the Sobolev norms.

\begin{Lem}[\textit{cf.} {\cite[Exercise 2.8.17]{NourdinPeccati}}] \label{Lem Uni bound sobolev}
Let $m \in \N_0$ and $\{F_n\}_{n \in \N} \subset \bigoplus_{i=0}^{m}\cW_i$. 
If the collection of the laws of the sequence $\{F_n\}_{n \in \N}$ is tight, then for any $\alpha \in \R$ and $p \in [1, \infty)$, 
\begin{equation}
\label{uniform sobolev bound}
    \sup_{n\in \N} \norm{F_n}_{\bW^{\alpha, p}} < \infty. 
\end{equation}
\end{Lem}

\begin{proof}
Applying the Paley--Zygmund inequality and Lemma \ref{Lem Norm equivalence}, we obtain
\begin{equation}
    P\left(F_n^2 > \frac{\E[F_n^2]}{2}\right) \geq \frac{\E[F_n^2]^2}{4\E[F_n^4]} \geq C_m > 0,
\end{equation}
where $C_m$ is a constant depending on $m$. 
By the tightness, there is $K>0$ such that $\sup_{n \in \N}P(F_n > K) \leq C_m$, which in turn implies $\sup_{n \in \N}\E[F_n^2] \leq 2 K^2$. 
The lemma now follows from Lemma \ref{Lem Norm equivalence}.
\end{proof}

Let $\{F_n\}_{n \in \N}$ be a sequence belonging to $\cW_m$ for some $m \in \N_0$ and converging in law to $\cN(0,1)$.
Then, we see from Lemma \ref{Lem Uni bound sobolev} that \eqref{uniform sobolev bound} holds for any $\alpha \in \R$ and $p \in [1,\infty)$. 
Moreover, thanks to the recent striking result of Herry, Malicet, and Poly \cite{Superconvergence}, the sequence $\{F_n\}_{n \in \N}$ turns out to be asymptotically nondegenerate in the following sense. 
For the proof and further details, see \cite{Superconvergence}. 
\begin{Lem}[{\cite[Theorem 3]{Superconvergence}}] \label{Lem asymtotic nondegeneracy}
Let $\{F_n\}_{n \in \N} \subset \cW_m$ for some $m \in \N$. 
If $F_n \xrightarrow[]{d} \cN(0,1)$, then 
\begin{equation}
    \limsup_{n \to \infty} \left\lVert \norm{DF_n}^{-2}_{\fH}\right\rVert_{L^k(\Omega)} < \infty, \qquad \text{for any $k \in \N$.} 
\end{equation}
\end{Lem}

\section{Stein's equation in the space of tempered distributions}
\label{section Stein's equation in the space of tempered distributions}

Recall that $\cN$ denotes a random variable with the standard normal distribution $\cN(0,1)$.
For a given test function $h$, Stein's equation for $\cN(0,1)$ is given by
\begin{equation}
\label{Stein's equation}
    f'(x) - xf(x) = h(x) - \E[h(\cN)].
\end{equation}
In this section, we first take the delta function $\delta_a, \ a \in \R$ as a test function $h$ and consider the corresponding Stein's equation in $\cS'(\R)$. 
Namely, for each $a \in \R$, we consider the differential equation
\begin{equation}
\label{Stein's equation in temp}
    \sD T - gT + \rho_{\cN}^{}(a) = \delta_a \quad \text{in $\cS'(\R)$}, 
\end{equation}
where $g$ is the smooth function $g(x)=x$. Recall that $\sD$ denotes the distributional derivative, and that the product $PT$ is a well-defined element of $\cS'(\R)$ for any $T \in \cS'(\R)$ and any polynomial $P$. 

\begin{Prp}
\label{Prp f_a is solution}
A solution to \eqref{Stein's equation in temp} is given by 
\begin{equation}
\label{f_a}
    f_a(x) \coloneqq \rho_{\cN}^{}(a)e^{\frac{x^2}{2}}\left( \int_x^{\infty}e^{-\frac{y^2}{2}}dy \ind_{(a, \infty)}(x) - \int_{-\infty}^{x}e^{-\frac{y^2}{2}}dy \ind_{(-\infty, a)}(x) \right).
\end{equation}
\end{Prp}

\begin{proof}
The function $f_a$ can be viewed as a tempered distribution since $f_a \in L^2(\R)$ by Lemma \ref{Lem f_a uni bound} below.
Let us write
\begin{equation}
    f_{a,+}(x) = \rho_{\cN}^{}(a) e^{\frac{x^2}{2}}\int_{x}^{\infty}e^{-\frac{y^2}{2}}dy \qquad \text{and} \qquad f_{a,-}(x) = -\rho_{\cN}^{}(a) e^{\frac{x^2}{2}}\int_{-\infty}^{x}e^{-\frac{y^2}{2}}dy.
\end{equation}
Then, $f_a(x) = f_{a,+}(x)\ind_{(a,\infty)}(x) + f_{a,-}(x)\ind_{(-\infty, a)}(x)$.  
Since both $f_{a,-}$ and $f_{a,+}$ are solutions of the differential equation $f'(x) - xf(x) + \rho_{\cN}^{}(a) = 0$, it follows that
\begin{equation}
\label{ODE for f_a}
    f_a'(x) - xf_a(x) + \rho_{\cN}^{}(a) = 0 \qquad \text{for all $x \in \R \setminus \{a\}$}.
\end{equation}
For any $\varphi \in \cS(\R)$, it holds that
\begin{align}
    \dpair{\cS'(\R)}{\sD f_a}{\varphi}{\cS(\R)}
    &= - \int_a^{\infty}f_{a,+}(x)\varphi'(x)dx -\int_{-\infty}^{a}f_{a,-}(x)\varphi'(x)dx\\
    &= \left(f_{a,+}(a) - f_{a,-}(a)\right)\varphi(a) + \int_{\R}f_a'(x)\varphi(x)dx\\
    &= \dpair{\cS'(\R)}{\delta_a + f_a'}{\varphi}{\cS(\R)}.
\end{align}
Thus, we obtain $\sD f_a = \delta_a + f_a'$ in $\cS'(\R)$, and this, together with \eqref{ODE for f_a}, yields 
\begin{align}
    \dpair{\cS'(\R)}{\sD f_a - gf_a + \rho_{\cN}^{}(a)}{\varphi}{\cS(\R)}  
    &= \dpair{\cS'(\R)}{\delta_a + f_a' - gf_a + \rho_{\cN}^{}(a)}{\varphi}{\cS(\R)} \\  
    &=  \dpair{\cS'(\R)}{\delta_a}{\varphi}{\cS(\R)}.  
\end{align}
Therefore, $f_a$ is a solution to \eqref{Stein's equation in temp}.
\end{proof}

\begin{Rem}
The method for solving differential equations like \eqref{Stein's equation in temp} is discussed, for example, in a more general framework in \cite[Section 10.1]{FollandFourierAnalysis}.
\end{Rem}

We now justify the earlier claim that $f_a \in L^2(\R)$.

\begin{Lem}\label{Lem f_a uni bound}
For any $a \in \R$, the function $f_a$ belongs to $L^2(\R) \cap L^{\infty}(\R)$ and we have
\begin{equation}
    \sup_{a \in \R}\norm{f_a}_{L^{\infty}(\R)} \leq 1 \qquad \text{and} \qquad  \sup_{a \in \R} \norm{f_a}_{L^2(\R)} \leq \sqrt{5}.
\end{equation}
\end{Lem}

\begin{proof}
By a simple computation, 
\begin{align}
    \rho_{\cN}^{}(a)e^{\frac{x^2}{2}}\int_x^{\infty}e^{-\frac{y^2}{2}}dy \ind_{(a, \infty)}(x) &\leq \rho_{\cN}^{}(a) e^{\frac{a^2}{2}}\int_a^{\infty}e^{-\frac{y^2}{2}}dy = \int_a^{\infty}\frac{1}{\sqrt{2 \pi}}e^{-\frac{y^2}{2}}dy, \\
    \rho_{\cN}^{}(a) e^{\frac{x^2}{2}}\int_{-\infty}^{x}e^{-\frac{y^2}{2}}dy \ind_{(-\infty, a)}(x) &\leq \rho_{\cN}^{}(a) e^{\frac{a^2}{2}}\int_{-\infty}^{a}e^{-\frac{y^2}{2}}dy = \int_{-\infty}^{a}\frac{1}{\sqrt{2 \pi}}e^{-\frac{y^2}{2}}dy,
\end{align}
and we obtain $\sup_{a \in \R}\norm{f_a}_{L^{\infty}(\R)} \leq 1$. 
When $|a| \leq 1$, the Mills ratio bound yields
\begin{align}
    \norm{f_a}_{L^2(\R)}^2 
    &\leq \int_{-\infty}^{-1}\left( e^{\frac{x^2}{2}} \int_{-\infty}^{x}e^{-\frac{y^2}{2}}dy\right)^2 dx + 2 \norm{f_a}_{L^{\infty}(\R)}^2 + \int_{1}^{\infty}\left( e^{\frac{x^2}{2}} \int_{x}^{\infty}e^{-\frac{y^2}{2}}dy\right)^2 dx \\
    &\leq 2\int_{1}^{\infty}\frac{dx}{x^2} + 2.
\end{align}
Similarly, when $a > 1$, 
\begin{align}
    &\norm{f_a}_{L^2(\R)}^2 \\
    &\leq \int_{-\infty}^{-1}\left( e^{\frac{x^2}{2}} \int_{-\infty}^{x}e^{-\frac{y^2}{2}}dy\right)^2 dx + 2 \norm{f_a}_{L^{\infty}(\R)}^2 + e^{-a^2}\int_{1}^{a}e^{x^2}dx + \int_{a}^{\infty}\left( e^{\frac{x^2}{2}} \int_{x}^{\infty}e^{-\frac{y^2}{2}}dy\right)^2 dx\\
    &\leq 2\int_{1}^{\infty}\frac{dx}{x^2} + 2 + \frac{1}{a},  
\end{align}
where the last inequality follows from $e^{x^2} \leq e^{x^2}(2-x^{-2}) = \frac{d}{dx}(x^{-1}e^{x^2})$ valid for $x >1$.
The same estimate also applies when $a< -1$, and therefore, $\sup_{a \in \R} \norm{f_a}_{L^2(\R)} \leq \sqrt{5}$.
\end{proof}

As stated in \cite[Chapter 5 Theorem 9.2]{IkedaWatanabe}, for a nice enough random variable $F$ (see Proposition \ref{Prp density derivative formula proof} for precise conditions), we have  
\begin{equation}
    \rho_{F}^{(n)}(a) = (-1)^{n}\dpairLR{\bW^{-\infty}}{(\sD^{n}\delta_a)(F)}{1}{\bW^{\infty}}.
\end{equation}
To deal with the derivatives of density functions, we next consider the following Stein's equation:
\begin{equation}\label{derivative stein eq}
    \sD T - g T = (-1)^n\sD^n\delta_a - \rho_{\cN}^{(n)}(a) \quad \text{in $\cS'(\R)$},  
\end{equation}
where $g(x) = x$.  

\begin{Prp}
\label{Prp f_a,n is solution}
Let $a \in \R$ and $n \in \N_0$. 
Let $f_a$ be the function in \eqref{f_a}. 
A solution to \eqref{derivative stein eq} is given by 
\begin{align}\label{f_a,n}
    f_{a,n} \coloneqq 
    \begin{cases}
        \dis f_a, &n = 0,\\
        \dis (-1)^n \left\{\sum_{k=0}^{n-1}H_k\sD^{n-1-k}\delta_a + H_n(a)f_a \right\}, &n \in \N.
    \end{cases}
\end{align}
\end{Prp}

\begin{proof}
The case $n = 0$ was shown in Proposition \ref{Prp f_a is solution}. Let $n \in \N$.
Clearly, $f_{a,n} \in \cS'(\R)$. 
In $\cS'(\R)$, we have 
\begin{align}
    \sD f_{a,n} 
    &= (-1)^{n}\left\{\sum_{k=0}^{n-1}\sD(H_k\sD^{n-1-k}\delta_a) + H_n(a)\sD f_a\right\}\\
    &= (-1)^{n}\sum_{k=0}^{n-1}(H_k'\sD^{n-1-k}\delta_a + H_k\sD^{n-k}\delta_a) + (-1)^{n}H_n(a) (g f_a - \rho_{\cN}^{}(a) + \delta_a) \\
    &= (-1)^{n}\sum_{k=0}^{n-1}(H_k'\sD^{n-1-k}\delta_a + H_k\sD^{n-k}\delta_a) + g\left(f_{a,n} - (-1)^n\sum_{k=0}^{n-1}H_k\sD^{n-1-k}\delta_a\right)\\
    &\quad - \rho_{\cN}^{(n)}(a) + (-1)^nH_n\delta_a,
\end{align}
where the last equality follows from (iii) of Lemma \ref{Lem Hermite polynomials properties} and $H_n(a)\delta_a = H_n\delta_a$ in $\cS'(\R)$.
Thus, 
\begin{align}
    \sD f_{a,n} - gf_{a,n} + \rho_{\cN}^{(n)}(a) 
    &= (-1)^n\sum_{k=0}^{n-1} (H_k' - gH_k)\sD^{n-1-k}\delta_a + (-1)^n\sum_{k=0}^{n-1}H_k\sD^{n-k}\delta_a + (-1)^nH_n\delta_a\\
    &= -(-1)^{n}\sum_{k=0}^{n-1}H_{k+1}\sD^{n-(k+1)}\delta_a + (-1)^n\sum_{k=1}^{n}H_k\sD^{n-k}\delta_a + (-1)^nH_0\sD^n\delta_a\\
    &= (-1)^n\sD^n\delta_a,
\end{align}
where in the second line, we used $-H_{k+1} = H_{k}' - gH_k$. 
This completes the proof.  
\end{proof}

The following lemma establishes a relation between $f_{a,n}$ and the derivatives of $\rho_{\cN}^{}$. 
\begin{Lem}\label{Lem f_a calc}
For any $a \in \R$ and $n, q \in \N_0$, 
\begin{equation}
    \dpairLR{\cS'(\R)}{f_{a,n}}{\rho^{(q)}_{\cN}}{\cS(\R)} = \frac{1}{q+1}\rho_{\cN}^{(n+q+1)}(a).
\end{equation}
\end{Lem}

\begin{proof} 
Using (i) and (ii) of Lemma \ref{Lem Hermite polynomials properties}, we obtain
\begin{align}
    &\dpair{\cS'(\R)}{f_{a}}{H_q\rho_{\cN}}{\cS(\R)}\\
    &= \rho_{\cN}^{}(a)\left\{\int_a^{\infty}H_q(x)\left(\int_x^{\infty}\rho_{\cN}^{}(y)dy\right)dx -\int_{-\infty}^{a}H_q(x)\left( \int_{-\infty}^{x}\rho_{\cN}^{}(y)dy \right)dx \right\}\\
    &= \frac{\rho_{\cN}^{}(a)}{q+1}\left\{\int_a^{\infty}H_{q+1}'(x)\left(\int_x^{\infty}\rho_{\cN}^{}(y)dy\right)dx -\int_{-\infty}^{a}H_{q+1}'(x)\left( \int_{-\infty}^{x}\rho_{\cN}^{}(y)dy \right)dx \right\}\\
    &= \frac{\rho_{\cN}^{}(a)}{q+1}\left\{ -H_{q+1}(a)\int_a^{\infty}\rho_{\cN}^{}(y)dy +  \int_a^{\infty}H_{q+1}(x)\rho_{\cN}^{}(x)dx \right.\\
    &\qquad \qquad \quad \left. - H_{q+1}(a)\int_{-\infty}^{a}\rho_{\cN}^{}(y)dy + \int_{-\infty}^{a}H_{q+1}(x) \rho_{\cN}^{}(x) dx \right\}\\
    &= -\frac{\rho_{\cN}^{}(a)H_{q+1}(a)}{q+1}.
\end{align}
It follows from (iii) of Lemma \ref{Lem Hermite polynomials properties} that 
\begin{equation}
    \dpairLR{\cS'(\R)}{f_{a}}{\rho_{\cN}^{(q)}}{\cS(\R)} = \frac{1}{q+1}\rho_{\cN}^{(q+1)}(a).  \label{f_a rho_N relation}
\end{equation}
For general case $n \in \N$, we see from \eqref{f_a,n} and the above computation that 
\begin{align}
    &\dpair{\cS'(\R)}{f_{a,n}}{\rho_{\cN}^{(q)}}{\cS(\R)}\\
    &= (-1)^{n+q}\sum_{j=0}^{n-1}
    \dpairLR{\cS'(\R)}{\sD^{n-1-j}\delta_a}{H_jH_q\rho_{\cN}^{}}{\cS(\R)} + (-1)^{n+q}H_n(a)\left(-\frac{1}{q+1}\rho_{\cN}^{}(a) H_{q+1}(a)\right)\\
    &\eqqcolon \mathbf{A_1} + \mathbf{A_2}
\end{align}
By (iv), (i), (iii) of Lemma \ref{Lem Hermite polynomials properties}, 
\begin{align}
    \mathbf{A_1} 
    &= (-1)^{n+q}\sum_{j=0}^{n-1}\sum_{k=0}^{j \land q}\binom{j}{k}\binom{q}{k}k!\dpair{\cS'(\R)}{\sD^{n-1-j}\delta_a}{H_{j+q-2k}\rho_{\cN}}{\cS(\R)}\\
    &= (-1)^{n+q}\sum_{j=0}^{n-1}\sum_{k=0}^{j \land q}\binom{j}{k}\binom{q}{k}k!\dpair{\cS'(\R)}{\delta_a}{H_{n-1+q-2k}\rho_{\cN}}{\cS(\R)}\\
    &= - \sum_{j=0}^{n-1}\sum_{k=0}^{j \land q}\binom{j}{k}\binom{q}{k}k!\rho_{\cN}^{(n-1+q-2k)}(a)\\
    &= - \sum_{k=0}^{q}\sum_{j=k}^{n-1}\binom{j}{k}\binom{q}{k}k!\rho_{\cN}^{(n-1+q-2k)}(a). 
\end{align}
Since $\dis \sum_{j=k}^{n-1}\binom{j}{k} = \binom{n}{k+1}$, we further obtain 
\begin{align}
    \mathbf{A_1} 
    &= -\sum_{k=0}^{q}\binom{n}{k+1}\binom{q}{k}k!\rho_{\cN}^{(n-1+q-2k)}(a) \\
    &= -\sum_{k=0}^{(n-1) \land q}\binom{n}{k+1}\binom{q}{k}k!\rho_{\cN}^{(n-1+q-2k)}(a). 
\end{align}
Similarly,
\begin{align}
    \mathbf{A_2}
    &= \frac{(-1)^{n+q+1}}{q+1}\rho_{\cN}^{}(a)H_n(a)H_{q+1}(a)\\
    &= \frac{(-1)^{n+q+1}}{q+1}\rho_{\cN}^{}(a)\sum_{k=0}^{n \land (q+1)}\binom{n}{k}\binom{q+1}{k}k!H_{n+q+1-2k}(a)\\
    &= \frac{1}{q+1}\sum_{k=0}^{n \land (q+1)}\binom{n}{k}\binom{q+1}{k}k!\rho_{\cN}^{(n+q+1-2k)}(a)\\
    &= \frac{1}{q+1}\rho_{\cN}^{(n+q+1)}(a) + \frac{1}{q+1}\sum_{k=1}^{n \land (q+1)}\binom{n}{k}\binom{q+1}{k}k!\rho_{\cN}^{(n+q+1-2k)}(a)\\
    &= \frac{1}{q+1}\rho_{\cN}^{(n+q+1)}(a) + \frac{1}{q+1}\sum_{k=0}^{(n-1) \land q}\binom{n}{k+1}\binom{q+1}{k+1}(k+1)!\rho_{\cN}^{(n-1+q-2k)}(a)\\
    &= \frac{1}{q+1}\rho_{\cN}^{(n+q+1)}(a) + \sum_{k=0}^{(n-1) \land q}\binom{n}{k+1}\binom{q}{k}k!\rho_{\cN}^{(n-1+q-2k)}(a).
\end{align}
Therefore, 
\begin{equation}
    \dpair{\cS'(\R)}{f_{a,n}}{\rho_{\cN}^{(q)}}{\cS(\R)}  = \mathbf{A_1} + \mathbf{A_2} = \frac{1}{q+1}\rho_{\cN}^{(n+q+1)}(a).
\end{equation}
\end{proof}

\begin{Rem}
Lemma \ref{Lem f_a calc} can be seen as a distributional analogue to Proposition 2.7 in \cite{NPexact}.
\end{Rem}

\section{Sobolev spaces for tempered distributions}\label{section Sobolev spaces for tempered distributions}
In this section, following \cite{BSimonHermiteExpansions} and \cite[Appendix to V.3]{ReedSimon1}, we introduce Sobolev spaces for tempered distributions via their Hermite coefficients.

For the reader's convenience, we first recall that any tempered distribution $T \in \cS'(\R)$ admits the Hermite expansion 
\begin{equation}
    T = \sum_{n=0}^{\infty} \dpair{\cS'(\R)}{T}{\phi_{n}}{\cS(\R)}\phi_{n},
\end{equation}
where the sum converges in the weak topology of $\cS'(\R)$.
The coefficients $\left\{ \dpair{\cS'(\R)}{T}{\phi_{n}}{\cS(\R)} \right\}_{n \in \N_0}$ are called Hermite coefficients of $T$. 
It is known that for every $T \in \cS'(\R)$, there exists some $\alpha \in \N_0$ and a constant $C>0$ such that  
\begin{equation}
\label{temp dis Hermite coefficients bound}
    |\dpair{\cS'(\R)}{T}{\phi_{n}}{\cS(\R)}| \leq C (n +1)^{\alpha}, \quad n \in \N_0. 
\end{equation}
If $T \in \cS(\R)$, then its Hermite expansion converges in $\cS(\R)$ endowed with its usual Fr\'{e}chet topology, and its Hermite coefficients satisfy
\begin{equation}
\label{rapid dec func Hermite coefficients bound}
    \sup_{n \in \N_0}(n+1)^m |\dpair{\cS'(\R)}{T}{\phi_{n}}{\cS(\R)}| < \infty
\end{equation}
for each $m \in \N_0$.
Conversely, if $T \in \cS'(\R)$ and its Hermite coefficients satisfy \eqref{rapid dec func Hermite coefficients bound} for each $m \in \N_0$, then its truncated Hermite expansion $\sum_{n=0}^{N} \dpair{\cS'(\R)}{T}{\phi_{n}}{\cS(\R)} \phi_{n}$ form a Cauchy sequence in $\cS(\R)$, and hence $T \in \cS(\R)$. 
For a more detailed account of Hermite expansions and coefficients, the reader is referred to \cite[Appendix to V.3]{ReedSimon1}. 

We define, for each $\alpha \in \R$, the space $\cS_{\alpha}$ by
\begin{equation}
    \cS_{\alpha} = \left\{ \left. T \in \cS'(\R)  \relmiddle| \sum_{n =0}^{\infty}(n + 1)^{\alpha}|\dpair{\cS'(\R)}{T}{\phi_{n}}{\cS(\R)}|^2 < \infty \right. \right\}. 
\end{equation}
For simplicity, we shall use the notation 
\begin{equation}
    \cS_{\alpha -} \coloneqq \bigcap_{\beta < \alpha} \cS_{\beta}, \qquad \alpha \in \R.
\end{equation}
Clearly, $\cS_{\beta} \subset \cS_{\alpha}$ if $\alpha \leq \beta$, and we easily see that
\begin{equation}
    \cS(\R) = \bigcap_{\alpha \in \R} \cS_{\alpha}  \qquad \text{and} \qquad \cS'(\R) = \bigcup_{\alpha \in \R} \cS_{\alpha}.
\end{equation}
Moreover, we have $\cS_{0} = L^2(\R)$. 
Indeed, $L^2(\R) \subset \cS_{0}$ follows since $\{\phi_{n}\}_{n \in \N_0}$ is a complete orthonormal system in $L^2(\R)$.
Conversely, if $T \in \cS_{0}$, its Hermite expansion also converges in $L^2(\R)$ to some $\wt{T} \in L^2(\R)$ and we have $\dpair{\cS'(\R)}{T}{\varphi}{\cS(\R)} = \dpair{\cS'(\R)}{\wt{T}}{\varphi}{\cS(\R)}$ for all $\varphi \in \cS(\R)$. 
Hence, $\cS_{0} \subset L^2(\R)$.  

We define a norm on $\cS_{\alpha}$ by 
\begin{equation}
    \norm{T}_{\cS_{\alpha}} \coloneqq 2^{\frac{\alpha}{2}}\left(\sum_{n =0}^{\infty} (n + 1)^{\alpha}|\dpair{\cS'(\R)}{T}{\phi_{n}}{\cS(\R)}|^2\right)^{\frac{1}{2}}, \qquad T \in \cS_{\alpha}. 
\end{equation} 
Since $\norm{T}_{\cS_{\alpha}} = 0$ implies that all Hermite coefficients of $T$ vanish, and hence $T = 0$ in $\cS'(\R)$, this shows that $\norm{\cdot}_{\cS_{\alpha}}$ is indeed a norm.

\begin{Lem}
For every $\alpha \in \R$, the space $(\cS_{\alpha}, \norm{\cdot}_{\cS_{\alpha}})$ is a Banach space in which $\cS(\R)$ is dense.
\end{Lem}

\begin{proof}
Let $\{T_k\}$ be a Cauchy sequence in $\cS_{\alpha}$. 
Then, $\left\{\left((n + 1)^{\frac{\alpha}{2}} \dpair{\cS'(\R)}{T_k}{\phi_{n}}{\cS(\R)}\right)_{n \in \N_0^d}\right\}_k$ is a Cauchy sequence in the Banach space $\ell^2(\N_0)$, so it converges to some $x= \{x_n\}_{n \in \N_0} \in \ell^2(\N_0)$. 
Thus, for each $n \in \N_0$, the limit $y_{n} \coloneqq \lim_{k \to \infty} \dpair{\cS'(\R)}{T_k}{\phi_{n}}{\cS(\R)}$ exists and we have
\begin{equation}
\label{space completenes eq1}
    (n + 1)^{\frac{\alpha}{2}}y_{n} = x_{n}.
\end{equation}
From this, $|y_{n}| \leq \norm{x}_{\ell^2(\N_0)}(n +1)^{-\frac{\alpha}{2}}$ for all $n \in \N_0$, and it follows from \cite[Theorem V. 14]{ReedSimon1} that there is a unique $T \in \cS'(\R)$ with 
$\dpair{\cS'(\R)}{T}{\phi_{n}}{\cS(\R)} = y_{n}$ for every $n \in \N_0$.
This, together with \eqref{space completenes eq1}, implies that $T \in \cS_{\alpha}$ and $T_k$ converges to $T$ in $\cS_{\alpha}$ as $k \to \infty$.
Therefore, $(\cS_{\alpha}, \norm{\cdot}_{\cS_{\alpha}})$ is a Banach space. 

If $T \in \cS_{\alpha}$, then its truncated Hermite expansion $T_N = \sum_{n =0}^{N} \dpair{\cS'(\R)}{T}{\phi_{n}}{\cS(\R)} \phi_{n}$ belongs to $\cS(\R)$ and converges to $T$ in $\cS_{\alpha}$ as $N \to \infty$, so $\cS(\R)$ is dense in $\cS_{\alpha}$.
This completes the proof. 
\end{proof}

Let us define the linear operator $\cA \colon L^2(\R) \to L^2(\R)$ with $\mathrm{Dom} (\cA) = \cS(\R)$ by 
\begin{equation}
\label{operator cA def}
    \cA \varphi(x) = \left( 1 + x^2 - \frac{d^2}{dx^2} \right)\varphi(x), \qquad \varphi \in \cS(\R).
\end{equation}
It is well known that $\cA$ is essentially self-adjoint and has a unique self-adjoint extension $\overline{\cA}$ whose spectrum set coincides with $\{2(n + 1) \mid n \in \N_0 \}$. 
To simplify notation, we will denote the extension $\overline{\cA}$ simply by $\cA$ from now on. 
The corresponding eigenfunctions are Hermite functions $\{\phi_n\}_{n \in \N_0}$: 
\begin{equation}
    \cA \phi_n = 2(n+1)\phi_n, \quad n \in \N_0. 
\end{equation}
From this, we find that
\begin{equation}
\label{S_alpha norm by cA}
    \norm{\varphi}_{\cS_{\alpha}} = \norm{\cA^{\frac{\alpha}{2}} \varphi}_{L^2(\R)}, \qquad \varphi \in \cS(\R).
\end{equation}

For our purposes, it is important to identify the space $\cS_{\alpha}$ to which a given tempered distribution belongs.
This identification can be made using the properties of Hermite functions. 

\begin{Prp}
\label{Prp tempered dis class}
Let $a \in \R$, $\alpha \in \R$, and $k \in \N_0$. Let $g$ be the smooth function $g(x) = x$.
\begin{enumerate}
    \item[\normalfont{(1)}] The delta function $\delta_a$ and the constant function $1$ both belong to $\cS_{-\frac{1}{2}-}$. 
    \item[\normalfont{(2)}] For any $T \in \cS_{\alpha}$, we have $\sD^{k} T \in \cS_{\alpha - k}$ and 
    \begin{equation}
        \norm{\sD^{k} T}_{\cS_{\alpha - k}} \leq \prod_{i=0}^{k-1}(2^{|\alpha - i| - 1} + 2^{-1})^{\frac{1}{2}} \norm{T}_{\cS_{\alpha}}.
    \end{equation}
    \item[\normalfont{(3)}] For any $T \in \cS_{\alpha}$, we have $g^{k}T \in \cS_{\alpha - k}$ and 
    \begin{equation}
        \norm{g^{k} T}_{\cS_{\alpha - k}} \leq \prod_{i=0}^{k-1}(2^{|\alpha - i| - 1} + 2^{-1})^{\frac{1}{2}} \norm{T}_{\cS_{\alpha}}.
    \end{equation}
\end{enumerate}
\end{Prp}

\begin{proof} 
(1)  By (iv) of Lemma \ref{Lem 1d Hermite functions property}, we have $|\dpair{\cS'(\R)}{\delta_a}{\phi_{n}}{\cS(\R)}| = |\phi_{n}(a)| = O\left(n^{-\frac{1}{4}}\right)$ as $n \to \infty$.
It follows that for $\alpha < -\frac{1}{2}$, 
\begin{equation}
    \sum_{n =0}^{\infty} (n + 1)^{\alpha}|\dpair{\cS'(\R)}{\delta_a}{\phi_{n}}{\cS(\R)}|^2 < \infty,
\end{equation}
and consequently $\delta_a \in \cS_{-\frac{1}{2}-}$. 
Similarly, since $\cF \phi_{n} = \sqrt{2\pi}(-\sqrt{-1})^{n}\phi_{n}$ by (iii) of Lemma \ref{Lem 1d Hermite functions property} and $1 = \cF \delta_0$ in $\cS'(\R)$, where $\cF$ denotes the Fourier transform, we see that for $\alpha < -\frac{1}{2}$, 
\begin{equation}
    \sum_{n = 0}^{\infty} (n + 1)^{\alpha}|\dpair{\cS'(\R)}{1}{\phi_{n}}{\cS(\R)}|^2 = 2\pi \sum_{n = 0}^{\infty} (n + 1)^{\alpha}|\phi_{n}(0)|^2 < \infty, 
\end{equation}
which implies $1 \in \cS_{-\frac{1}{2}-}$. 

\vspace{1mm}
\noindent
(2) 
We deduce from (ii) of Lemma \ref{Lem 1d Hermite functions property} that
\begin{align}
    &\sum_{n = 0}^{\infty}(n + 1)^{\alpha - 1}|\dpair{\cS'(\R)}{\sD T}{\phi_n}{\cS(\R)}|^2\\ 
    &\leq \sum_{n = 1}^{\infty} n(n + 1)^{\alpha - 1}|\dpair{\cS'(\R)}{T}{\phi_{n - 1}}{\cS(\R)}|^2 + \sum_{n = 0} (n + 1)^{\alpha}|\dpair{\cS'(\R)}{T}{\phi_{n+1}}{\cS(\R)}|^2\\ 
    &= \sum_{n = 1}^{\infty} n^{\alpha}\left(\frac{n + 1}{n} \right)^{\alpha - 1} |\dpair{\cS'(\R)}{T}{\phi_{n - 1}}{\cS(\R)}|^2 + \sum_{n = 0}^{\infty} (n + 2)^{\alpha} \left(\frac{n + 1}{n + 2}\right)^{\alpha}|\dpair{\cS'(\R)}{T}{\phi_{n + 1}}{\cS(\R)}|^2\\ 
    &\leq (2^{|\alpha|} + 1) \sum_{n = 0}^{\infty} (n + 1)^{\alpha} |\dpair{\cS'(\R)}{T}{\phi_{n}}{\cS(\R)}|^2. 
\end{align}
It follows that 
\begin{equation}
    \norm{\sD T}_{\cS_{\alpha - 1}} \leq \sqrt{2^{|\alpha| - 1} + 2^{-1}} \norm{T}_{\cS_{\alpha}}.
\end{equation}
From this, we can deduce the desired inequality.

\noindent
(3)
By (ii) of Lemma \ref{Lem 1d Hermite functions property}, 
\begin{equation}
    |\dpair{\cS'(\R)}{gT}{\phi_{n}}{\cS(\R)}|^2 \leq n |\dpair{\cS'(\R)}{T}{\phi_{n - 1}}{\cS(\R)}|^2 + (n + 1)|\dpair{\cS'(\R)}{T}{\phi_{n + 1}}{\cS(\R)}|^2,  
\end{equation}
so the same estimate as in (2) can be derived.
\end{proof}

\begin{Rem}
\begin{enumerate}
    \item[(1)] It should be noted that a statement such as ``a tempered distribution $U$ belongs to $\cS_{\beta}$ for some $\beta \in \R$'' does not exclude the possibility that $U$ belongs to a smaller space $\cS_{\gamma}$ with $\beta < \gamma$. 
    For example, we see from Proposition \ref{Prp tempered dis class} that $g \delta_0 \in \cS_{-\frac{3}{2}-}$, but it actually belongs to all $\cS_{\alpha}, \ \alpha \in \R$ since $g\delta_0 = 0$ in $\cS'(\R)$. 
    \item[(2)] We also note that $\sD^{k} T \in \cS_{\alpha - k}$ or $g^{k} T \in \cS_{\alpha - k}$ does not necessarily imply $T \in \cS_{\alpha}$. 
    In other words, the converses of (2) and (3) in Proposition \ref{Prp tempered dis class} are generally false. 
    Indeed, $\sD \ind_{(a, \infty)} =  \delta_a \in \cS_{-1}$ and $x \left(\mathrm{p.v.}\left(\frac{1}{x}\right)\right) = 1 \in \cS_{-1}$ from Proposition \ref{Prp tempered dis class}, but neither $\ind_{(a, \infty)}$ nor $\mathrm{p.v.}\left(\frac{1}{x}\right)$ belongs to $\cS_0 = L^2(\R)$.
\end{enumerate}
\end{Rem}

In order to estimate density functions and their derivatives uniformly on $\R$, the following uniform bound is important.

\begin{Prp}\label{Prp derivative uniform estimate}
For every $a \in \R$ and $k \in \N_0$, we have $\sD^{k}\delta_a \in \cS_{-1-k}$ and
\begin{equation}
    \sup_{a \in \R}\norm{\sD^{k}\delta_a}_{\cS_{-1-k}} < \infty.
\end{equation}
\end{Prp}

\begin{proof}
By (iv) of Lemma \ref{Lem 1d Hermite functions property}, we have  
\begin{equation}
    \sup_{a\in\R}\norm{\delta_a}_{\cS_{-1}} 
    \leq 2^{-\frac{1}{2}}\left( \sum_{n=0}^{\infty} (n+1)^{-1} \norm{\phi_n}_{\infty}^2 \right)^{\frac{1}{2}} < \infty.
\end{equation}
Therefore, from Proposition \ref{Prp tempered dis class} the proposition follows.
\end{proof}

We now identify the space $\cS_{\alpha}$ that contains the solutions to Stein's equations considered in Section \ref{section Stein's equation in the space of tempered distributions}.
We already know from Lemma \ref{Lem f_a uni bound} that $f_{a} \in \cS_{0}$ and $\sup_{a \in \R}\norm{f_a}_{\cS_{0}} < \infty$. 
This result can be sharpened and generalized as follows.  

\begin{Prp}
\label{Prp f_a,n class Dfan uniform bound}
For every $a \in \R$ and $n \in \N_0$, the distribution $f_{a,n}$ defined in $\eqref{f_a,n}$ belongs to $\cS_{-n + \frac{1}{2}-}$.
Moreover, for every $k\in \N_0$, we have
\begin{equation}
\label{uniform bound for Df_a,n}
    \sup_{a \in \R}\norm{\sD^k f_{a,n}}_{\cS_{-n-k}} < \infty.  
\end{equation}
\end{Prp}

\begin{proof}
By (ii) of Lemma \ref{Lem 1d Hermite functions property},
\begin{equation}
    \dpair{\cS'(\R)}{f_a}{\phi_{n}}{\cS(\R)} = -\frac{1}{\sqrt{2(n+1)}} \dpair{\cS'(\R)}{\sD f_a - gf_a}{\phi_{n+1}}{\cS(\R)}.
\end{equation}
For every $\alpha \in (0, \frac{1}{2})$, we obtain 
\begin{align}
    \norm{f_a}^2_{\cS_{\alpha}} 
    &= 2^{\alpha - 1}\sum_{n =0}^{\infty} (n+1)^{\alpha - 1}|\dpair{\cS'(\R)}{\sD f_a - g f_a}{\phi_{n + 1}}{\cS(\R)}|^2 \\ 
    &= 2^{\alpha - 1} \sum_{n = 1}^{\infty} \left(\frac{n+1}{n}\right)^{1 - \alpha} (n+1)^{\alpha - 1}|\dpair{\cS'(\R)}{\sD f_a - g f_a}{\phi_{n}}{\cS(\R)}|^2\\ 
    &\leq 2\norm{\sD f_a - g f_a}^{2}_{\cS_{\alpha - 1}}.
\end{align}
Since $\sD f_a - gf_a = \delta_a - \rho_{\cN}^{}(a) \in \cS_{-\frac{1}{2}-}$ by Proposition \ref{Prp tempered dis class}, the above upper bound is finite, and therefore $f_{a} \in \cS_{\frac{1}{2}-}$.

For $n \in \N$, Proposition \ref{Prp tempered dis class} yields that for each $k = 0, \ldots, n-1$ and for each $j = 0, \ldots, k$, 
\begin{equation}
    g^j\sD^{n-1-k}\delta_a \in \cS_{-n+k-j+\frac{1}{2}-} \subset \cS_{-n+\frac{1}{2}-}.
\end{equation}
Therefore, we conclude that $f_{a,n} \in \cS_{-n + \frac{1}{2}-}$ for every $n \in \N_0$. 

We now prove \eqref{uniform bound for Df_a,n}.
It suffices to consider the case $k = 0$, thanks to Proposition \ref{Prp tempered dis class}. 
Using (ii) of Lemma \ref{Lem 1d Hermite functions property} again, we find that 
\begin{equation}
    \dpair{\cS'(\R)}{f_{a,n}}{\phi_{m}}{\cS(\R)} = -\frac{1}{\sqrt{2(m+1)}} \dpair{\cS'(\R)}{\sD f_{a,n} - gf_{a,n}}{\phi_{m+1}}{\cS(\R)}, 
\end{equation}
and it follows that
\begin{align}
    \norm{f_{a,n}}^{2}_{\cS_{-n}} 
    &= 2^{-n-1} \sum_{m=0}^{\infty} (m+1)^{-n-1} |\dpair{\cS'(\R)}{\sD f_{a,n} - gf_{a,n}}{\phi_{m+1}}{\cS(\R)}|^2 \\ 
    &\leq \norm{\sD f_{a,n} - g f_{a,n}}^2_{\cS_{-n-1}}.
\end{align}
By Propositions \ref{Prp f_a,n is solution} and \ref{Prp derivative uniform estimate}, we have
\begin{align}
    \sup_{a\in\R}\norm{\sD f_{a,n} - g f_{a,n}}_{\cS_{-n-1}} 
    &= \sup_{a\in\R}\norm{(-1)^n\sD^n\delta_a - \rho_{\cN}^{(n)}(a)}_{\cS_{-n-1}}\\
    &\lesssim_n \sup_{a\in\R}\norm{\delta_a}_{\cS_{-1}} + \norm{1}_{\cS_{-1}}\sup_{a \in \R}|\rho_{\cN}^{(n)}(a)| < \infty.  
\end{align}
Therefore, 
\begin{equation}
    \sup_{a \in \R} \norm{f_{a,n}}_{\cS_{-n}} < \infty,
\end{equation}
and the proof is complete.
\end{proof}

\section{Composition of tempered distributions}
\label{section Composition of tempered distributions with Gaussian functionals}

In this section, we justify the composition $T(F)$ of a tempered distribution $T$ with a sufficiently nice Gaussian functional $F$ in our framework. 
In particular, we show that for every $n \in \N_0$ and $a \in \R$, the generalized expectation 
\begin{equation}\label{generalized expectation of D^ndelta_a}
    \E[(\sD^n\delta_a)(F)] = (-1)^n \dpair{\bW^{-\infty}}{(\sD^n\delta_a)(F)}{1}{\bW^{\infty}}
\end{equation}
coincides with the density function $\rho_F^{(n)}(a)$ of $F$.

Such a composition theory is well known in Malliavin calculus, and comprehensive treatments can be found, for example, in \cite[Chapter V, Section 9]{IkedaWatanabe}, \cite[Section 2]{SW-AWF}, and \cite{FractionalSobolev}, and some related results appear in \cite{SW-generalizedWienerfunctionals} and \cite{SW-donskerdelta}.
However, since we work with spaces $(\cS_{\alpha}, \norm{\cdot}_{\cS_{\alpha}}), \ \alpha \in \R$ that differ from those used in the aforementioned references (see Remark \ref{Rem comparison with Watanabe framework}), we provide the necessary details for the sake of completeness and the reader's convenience, although the outline of the argument is essentially the same as in \cite{IkedaWatanabe} and \cite{FractionalSobolev}.

In Section \ref{subsection Integration by parts}, we derive an integration by parts formula for the operator $\cA$ defined in \eqref{operator cA def}. 
Section \ref{subsection Estimates of composition operators} introduces the composition operators and defines the composition. 
Finally, in Section \ref{subsection Composition with delta and constant functions}, we show that \eqref{generalized expectation of D^ndelta_a} coincides with $\rho_F^{(n)}(a)$ and that the composition with a constant function remains a constant function.

\subsection{Integration by parts}\label{subsection Integration by parts}
For a Malliavin differentiable random variable $F$, set 
\begin{equation}
    \Delta_F = \norm{DF}_{\fH}^2.
\end{equation}
We use the following estimate.

\begin{Lem}\label{Lem delta_F^-1 sobolev norm}
Let $k \in \N_0$, $F \in \bW^{k+1, p}$, and $\Delta_F^{-1} \in L^q(\Omega)$ for $p,q,r \in (1,\infty)$ satisfying $\frac{2k}{p} + \frac{k+1}{q} \leq \frac{1}{r}$. 
Then, $\Delta_F^{-1} \in \bW^{k,r}$ and 
\begin{equation}
    \norm{\Delta_F^{-1}}_{\bW^{k,r}} \lesssim_{k,p,q} \norm{F}_{\bW^{k+1, p}}^{2k} \norm{\Delta_F^{-1}}_{L^q(\Omega)}^{k+1}. \label{sec5 flem1}
\end{equation}
\end{Lem}

\begin{proof}
The set $\cP$ is dense in $\bW^{k+1, p}$, and there exists $F_n \in \cP, \ n \in \N$ such that $F_n \to F$ in $\bW^{k+1, p} = \bD^{k+1, p}$. 
By taking a subsequence if necessary, we can extract a sequence $F_{n_i}$ such that
\begin{equation}
    D^jF_{n_i} \to D^jF \quad \text{in $\fH^{\otimes j}$ \  a.s.}, \qquad j = 0, 1, \ldots, k+1.
\end{equation}
For simplicity of notation, we write $F_i$ instead of $F_{n_i}$ from now on. 

For $\varepsilon \in (0,1]$, let $f_\varepsilon\colon \R \to \R$ be a $C^{\infty}_{\mathrm{b}}$-function that coincides with $\frac{1}{x+\epsilon}$ on $[0,\infty)$.
Let $P_n(x_1, \ldots, x_m), \  n \in \N_0, \ m \in \N$ denotes a polynomial in $x_1, \ldots, x_m$ consisting only of terms of degree $n$. 
Let $Q^{(q)}_n(x_1, \ldots, x_m), \  n \in \N_0, \ m \in \N, \  q>0$ denotes a polynomial in $x_1, \ldots, x_m$ of degree at most $n$, with coefficients that may depend on positive powers of $q$. 
We adopt the convention that $Q_{n}^{(q)} \equiv 0$ if $n \in \Z \setminus \N_0$.

By the dominated convergence theorem, 
\begin{equation}
    f_{\varepsilon}(\Delta_{F_i}) \xrightarrow{i \to \infty} f_{\varepsilon}(\Delta_F) 
\end{equation}
in $L^s(\Omega)$ for all $s \in [1,\infty)$. 
Moreover, by applying the chain rule and the Leibniz rule, 
one can easily show that 
\begin{align}
    &\norm{D^j f_{\varepsilon}(\Delta_{F_i})}_{\fH^{\otimes j}} \\
    &\lesssim f_{\varepsilon}(\Delta_{F_i})^{j+1}\left\{ P_{2j}\left(\norm{DF_i}_{\fH}, \ldots, \norm{D^{j+1}F_i}_{\fH^{\otimes (j+1)}}\right) + \varepsilon Q^{(\varepsilon)}_{2j-2}\left(\norm{DF_i}_{\fH}, \ldots, \norm{D^{j+1}F_i}_{\fH^{\otimes (j+1)}}\right) \right\}. 
\end{align}
Here, the implicit constant is independent of $\varepsilon$. 
From this, we obtain
\begin{align}
    &\norm{f_{\varepsilon}(\Delta_{F_i})}_{\bD^{k,r}}\\ 
    &\lesssim \sum_{j=0}^{k} \norm{f_{\varepsilon}(\Delta_{F_i})}_{L^q(\Omega)}^{j+1} \left\{ P_{2j}\left(\norm{DF_i}_{L^p(\Omega;\fH)}, \ldots, \norm{D^{j+1}F_i}_{L^p(\Omega;\fH^{\otimes (j+1)})}\right) \right. \\
    &\qquad \qquad \qquad \qquad \qquad \left. + \varepsilon Q_{2j-2}^{(\varepsilon)}\left(\norm{DF_i}_{L^p(\Omega;\fH)}, \ldots, \norm{D^{j+1}F_i}_{L^p(\Omega;\fH^{\otimes (j+1)})}\right)  \right\}\\
    &\lesssim \sum_{j=0}^{k} \norm{f_{\varepsilon}(\Delta_{F_i})}_{L^q(\Omega)}^{j+1} \left\{ \norm{F_i}_{\bD^{j+1, p}}^{2j} + \varepsilon Q_{2j-2}^{(\varepsilon)}(\norm{F_i}_{\bD^{j+1, p}}) \right\},
\end{align}
which in turn implies 
\begin{equation}
    \sup_{i} \norm{f_{\varepsilon}(\Delta_{F_i})}_{\bD^{k,r}} < \infty
\end{equation}
since $\varepsilon$ is fixed and $f_{\varepsilon}(\Delta_{F_{i}}) \leq \varepsilon^{-1}$. 
It then follows from Lemma \ref{Lem criterion for elements in Sobolev} that $f_{\varepsilon}(\Delta_{F}) \in \bW^{k,r}$ and 
\begin{align}
    \norm{f_{\varepsilon}(\Delta_F)}_{\bW^{k,r}} 
    &\lesssim_{k,r} \liminf_{i \to \infty}\norm{f_{\varepsilon}(\Delta_{F_{i}})}_{\bD^{k,r}}\\
    &\lesssim \sum_{j=0}^{k} \norm{\Delta_{F}^{-1}}_{L^q(\Omega)}^{j+1} \left\{ \norm{F}_{\bD^{j+1, p}}^{2j} + \varepsilon Q_{2j-2}^{(\varepsilon)}(\norm{F}_{\bD^{j+1, p}}) \right\}.
\end{align}
We now let $\varepsilon \to 0$. 
Since $f_{\varepsilon}(\Delta_F)$ converges to $\Delta_F^{-1}$ in $L^r(\Omega)$ by the dominated convergence theorem, we again deduce from Lemma \ref{Lem criterion for elements in Sobolev} that $\Delta_F^{-1} \in \bW^{k,r}$ and 
\begin{align}
    \norm{\Delta_F^{-1}}_{\bW^{k,r}} 
    &\lesssim \sum_{j=0}^{k} \norm{F}_{\bD^{j+1, p}}^{2j} \norm{\Delta_{F}^{-1}}_{L^q(\Omega)}^{j+1} \\
    &\lesssim  \norm{F}_{\bD^{k+1, p}}^{2k} \norm{\Delta_{F}^{-1}}_{L^q(\Omega)}^{k+1},
\end{align}
where the last inequality follows since 
\begin{equation}
    1 = \norm{DF \Delta_{F}^{-\frac{1}{2}}}^{2(k-j)}_{L^r(\Omega; \fH)} \leq \norm{DF}^{2(k-j)}_{L^p(\Omega; \fH)} \norm{\Delta_{F}^{-\frac{1}{2}}}_{L^{2q}(\Omega)}^{2(k-j)} \leq \norm{F}^{2(k-j)}_{\bD^{k+1, p}} \norm{\Delta_{F}^{-1}}_{L^q(\Omega)}^{k-j}. 
\end{equation}
This completes the proof.
\end{proof}

For a Malliavin differentiable random variable $F$ and a random variable $G$, set 
\begin{equation}
    \Theta_F(G) = D^{\ast}\left( G\Delta_F^{-1}DF \right) \qquad \text{and} \qquad \Theta^{\circ n}_F(G) = (\Theta_F \circ \cdots \circ \Theta_F)(G) 
\end{equation}
where $\Theta_F^{\circ n}$ is the $n$-fold composition of $\Theta_F$. 
For instance, $\Theta_F^{\circ 2}(G) = \Theta_F(\Theta_F(G))$.
We set $\Theta_F^{\circ 0}(G) = G$.
\begin{Lem}\label{Lem Theta_F^j estimate}
Let $F \in \bW^{k+1, p}$, $\Delta_F^{-1} \in L^q(\Omega)$, and $G \in \bW^{k, r_0}$ for some $k \in \N_0$ and $p, q, r_0 \in (1, \infty)$. 
Define $r_j, \ j \in \{1, \ldots, k+1\}$ by 
\begin{equation}
    \frac{1}{r_j} = \frac{1}{r_0} + \sum_{i=0}^{j-1} \left( \frac{2k+1-2i}{p} + \frac{k+1-i}{q} \right) = \frac{1}{r_0} + \frac{j(2k+2-j)}{p} + \frac{j(2k+3-j)}{2q}. \label{r_k def}
\end{equation}
If $r_j > 1$ for some $j \in \{1, \ldots, k+1\}$, then for every $i \in \{1, \ldots, j\}$, we have
\begin{equation}
    \Theta_F^{\circ (i-1)}(G) \Delta_F^{-1} DF \in \bW^{k-i+1, r_i}(\fH), \qquad  \Theta_F^{\circ i}(G) \in \bW^{k-i, r_i},
\end{equation}
and 
\begin{equation}
    \norm{\Theta_F^{\circ i}(G)}_{\bW^{k-i, r_i}} \lesssim_{i,k,p,q,r_0} \norm{G}_{\bW^{k,r_0}} \norm{F}_{\bW^{k+1, p}}^{i(2k+2-i)} \norm{\Delta_F^{-1}}_{L^q(\Omega)}^{\frac{i(2k+3-i)}{2}}.
\end{equation}
\end{Lem}
\begin{proof}
We have $r_0 > r_1 > \cdots > r_{j} > 1$. 
If $\Theta_F^{\circ (i-1)}(G) \in \bW^{k-i+1, r_{i-1}}$, then Lemmas \ref{Lem Sobolev holder ineq} and \ref{Lem delta_F^-1 sobolev norm} imply that $\Theta_F^{\circ (i-1)}(G) \Delta_F^{-1} DF \in \bW^{k-i+1, r_i}(\fH)$. 
It follows from the continuity of $D^{\ast}$ that $\Theta_F^{\circ i}(G) \in \bW^{k-i, r_i}$ and 
\begin{align}
    \norm{\Theta_F^{\circ i}(G)}_{W^{k-i, r_i}} 
    &\lesssim \norm{\Theta_F^{\circ (i-1)}\Delta_F^{-1} DF}_{\bW^{k-i+1, r_i}(\fH)}\\
    &\lesssim \norm{\Theta_F^{\circ (i-1)}}_{\bW^{k-i+1,r_{i-1}}} \norm{\Delta_F^{-1}}_{\bW^{k-i+1,s}} \norm{DF}_{\bW^{k-i+1,p}(\fH)}\\
    &\lesssim \norm{\Theta_F^{\circ (i-1)}}_{\bW^{k-i+1,r_{i-1}}} \norm{F}_{\bW^{k+1, p}}^{2(k-i+1)+1} \norm{\Delta_F^{-1}}_{L^q(\Omega)}^{(k-i+1) + 1}, 
\end{align}
where $\frac{1}{s} \coloneqq \frac{2(k-i+1)}{p} + \frac{(k-i+1)+1}{q}$. 
By iterating this estimate, we conclude the proof of the lemma.
\end{proof}

\begin{Rem}\label{Rem notation of r_j}
Of course, the definition of $r_j$ depends on $k$, $p$, $q$, and $r_0$. 
However, to avoid cumbersome notation, we shall not indicate this dependence explicitly. 
When we wish to emphasize it, we will write $r_j[k,p,q,r_0]$.
\end{Rem}

By Lemma \ref{Lem Theta_F^j estimate}, we obtain the following integration by parts formula.

\begin{Prp}\label{Prp IBP}
Let $F \in \bW^{k+1, p}$, $\Delta_F^{-1} \in L^q(\Omega)$, and $G \in \bW^{k, r_0}$ for some $k \in \N_0$ and $p, q, r_0 \in (1, \infty)$. 
Define $r_j, \  j \in \{1, \ldots, k+1\}$ by \eqref{r_k def}. 
If $r_j > 1$ for some $j \in \{0, \ldots, k\}$, then 
\begin{equation}
    \E[\varphi^{(i)}(F)G] = \E[\varphi(F)\Theta_F^{\circ i}(G)]
\end{equation}
holds for all $\varphi \in \cS(\R)$ and $i \in \{0, \ldots, j\}$. 
Moreover, if $\frac{1}{r_{k+1}}+\frac{1}{p} \leq 1$, then we have
\begin{equation}
    \E[\varphi^{(k+1)}(F)G] = \dpairLR{\bW^{1, r_{k+1}'}}{\varphi(F)}{\Theta_F^{\circ (k+1)}(G)}{\bW^{-1, r_{k+1}}}, \qquad \varphi \in \cS(\R), \quad \frac{1}{r_{k+1}'} + \frac{1}{r_{k+1}} = 1.
\end{equation}
\end{Prp}

\begin{proof}
The first equation is trivial when $i = 0$, so let $r_j >1$ for some $j \geq 1$.
In this case, we have $r_{0} > r_1 > \cdots > r_{j} > 1$.
By the chain rule, 
\begin{equation}
    \varphi^{(i)}(F)G = \abra{\varphi^{(i)}(F)DF, G\Delta_{F}^{-1}DF}_{\fH} = \abra{D\varphi^{(i-1)}(F), G\Delta_{F}^{-1}DF}_{\fH}
\end{equation}
holds almost surely. 
This, together with $G\Delta_{F}^{-1}DF \in \bW^{k,r_1}$ from Lemma \ref{Lem Theta_F^j estimate} and $\varphi^{(i-1)}(F) \in L^{r_1'}(\Omega)$, where $\frac{1}{r_1'} + \frac{1}{r_1} =1$, implies that 
\begin{align}
    \E[\varphi^{(i)}(F)G] 
    &= \dpair{\bW^{-k,r_1'}(\fH)}{D\varphi^{(i-1)}(F)}{G\Delta_F^{-1}DF}{\bW^{k,r_1}(\fH)}\\
    &= \dpair{\bW^{-k+1, r_1'}}{\varphi^{(i-1)}(F)}{\Theta_F(G)}{\bW^{k-1,r_1}}\\
    &= \E[\varphi^{(i-1)}(F)\Theta_F(G)].
\end{align}
Therefore, by iterating this procedure, we obtain
\begin{equation}
    \E[\varphi^{(i)}(F)G] = \E[\varphi^{(i-1)}(F)\Theta_F(G)] = \cdots = \E[\varphi(F)\Theta_F^{\circ i}(G)]
\end{equation}
for all$\varphi \in \cS(\R)$ and $i \in \{0, \ldots, j\}$. 

If $\frac{1}{r_{k+1}}+\frac{1}{p} \leq 1$, then $r_k > r_{k+1} > 1$ and $r_{k+1}' \leq p$.
By the above argument, we have
\begin{equation}
    \E[\varphi^{(k+1)}(F)G] = \E[\varphi'(F)\Theta_F^{\circ k}(G)]  
    = \E[\langle D\varphi(F), \Theta_F^{\circ k}(G)\Delta_F^{-1} DF \rangle_{\fH}].
\end{equation}
Since $\varphi(F) \in \bW^{1, p} \subset \bW^{1, r_{k+1}'}$, one further iteration yields 
\begin{align}
    \E[\varphi^{(k+1)}(F)G] 
    &= \dpairLR{L^{r_{k+1}'}(\Omega; \fH)}{D\varphi(F)}{\Theta_F^{\circ k}(G)\Delta_F^{-1} DF}{L^{r_{k+1}}(\Omega;\fH)}\\
    &= \dpairLR{\bW^{1, r_{k+1}'}}{\varphi(F)}{\Theta_F^{\circ (k+1)}(G)}{\bW^{-1, r_{k+1}}}.
\end{align} 
This completes the proof.
\end{proof}

Recall the differential operator 
\begin{equation}
    \cA \varphi(x) = \left( 1 + x^2 - \frac{d^2}{dx^2} \right)\varphi(x), \qquad \varphi \in \cS(\R)
\end{equation}
defined in \eqref{operator cA def}. 
We now establish the integration by parts formula with respect to $\cA$. 
Let 
\begin{equation}
    \cM_F(G) \coloneqq (1+F^2)G, 
\end{equation}
and let 
\begin{equation}
    \Upsilon_F(G) \coloneqq \cM_F(G) - \Theta_F^{\circ 2}(G) \quad \text{and} \quad \Upsilon_F^{\circ n}(G) \coloneqq {(\Upsilon_F \circ \cdots \circ \Upsilon_F)}(G).
\end{equation}
As before, we set $\Upsilon_F^{\circ 0}(G) = G$. 
Observe that for nice enough random variables $F$ and $G$, we have
\begin{equation}
    \E[(\cA^n\varphi)(F)G] = \E[\varphi(F) \Upsilon_F^{\circ n}(G)], \qquad \varphi \in \cS(\R).
\end{equation}

\begin{Lem}\label{Lem Upsilon^j estimate}
Let $F \in \bW^{k+1, p}$, $\Delta_F^{-1} \in L^q(\Omega)$, and $G \in \bW^{k, r_0}$ for some $k \in \N$ and $p, q, r_0 \in (1, \infty)$. 
Define $r_j, \  j \in \{1, \ldots, k+1\}$ by \eqref{r_k def}.
If $r_{2j} > 1$ for some $j \in \{1, \ldots, \floor{\frac{k+1}{2}}\}$, then for every $i \in \{1, \ldots, j\}$, we have $\Upsilon_F^{\circ i}(G) \in \bW^{k-2i, r_{2i}}$ and 
\begin{equation}
    \norm{\Upsilon_F^{\circ i}(G)}_{\bW^{k - 2i, r_{2i}}} 
    \lesssim_{i, k, p, q, r_0} 
    \norm{G}_{\bW^{k,r_0}}(1+\norm{F}_{\bW^{k+1, p}}^2)^{2i(k-i+1)}(1+\norm{\Delta_F^{-1}}_{L^q(\Omega)})^{i(2k-2i+3)}.
\end{equation}
\end{Lem}

\begin{proof}
Let $r_2 > 1$. 
By Lemma \ref{Lem Theta_F^j estimate}, we know that $\Theta_F^{\circ 2}(G) \in \bW^{k-2, r_2}$ and 
\begin{equation}
    \norm{\Theta_F^{\circ 2}(G)}_{\bW^{k-2, r_2}} 
    \lesssim_{k,p,q,r_0} \norm{G}_{\bW^{k,r_0}}\norm{F}_{\bW^{k+1,p}}^{4k}\norm{\Delta_F^{-1}}_{L^q(\Omega)}^{2k+1}.
\end{equation}
On the other hand, since $\frac{1}{r_{2}} \geq \frac{1}{r_0} + \frac{2}{p}$, we deduce from Lemma \ref{Lem Sobolev holder ineq} that $\cM_F(G) \in \bW^{k-2,r_2}$ and
\begin{equation}
    \norm{\cM_F(G)}_{\bW^{k-2,r_2}} \leq \norm{\cM_F(G)}_{\bW^{k,r_2}} \lesssim_{k,p,r_0} \norm{G}_{\bW^{k, r_0}}(1 + \norm{F}^2_{\bW^{k,p}}).  
\end{equation}
From these, we obtain $\Upsilon_F(G) \in \bW^{k-2,r_2}$ and
\begin{align}
    \norm{\Upsilon_F(G)}_{\bW^{k-2,r_2}} 
    &\leq \norm{\cM_F(G)}_{\bW^{k-2,r_2}} + \norm{\Theta_F^{\circ 2}(G)}_{\bW^{k-2,r_2}}\\
    &\lesssim_{k,p,q,r_0} \norm{G}_{\bW^{k,r_0}}(1+ \norm{F}_{\bW^{k+1,p}}^2)^{2k} (1+ \norm{\Delta_F^{-1}}_{L^q(\Omega)})^{2k+1}. \label{Upsilon_F sobolev norm 1}
\end{align}
If $r_{2j} > 1$ for some $j \in \{1, \ldots, \floor{\frac{k+1}{2}}\}$, then we have $r_2 > \cdots > r_{2j} > 1$. 
Observe that 
\begin{equation}
    \frac{1}{r_{2i}} = \frac{1}{r_{2i-2}} + \sum_{l=0}^{1}\left(\frac{2(k-2(i-1)) +1 - 2l}{p} + \frac{(k-2(i-1)) +1 -l}{2q}\right),
\end{equation}
which implies $r_{2i}[k,p,q,r_0] = r_2[k-2(i-1), p, q, r_{2i-2}[k,p,q,r_0]]$ (see Remark \ref{Rem notation of r_j}). 
Thus, repeating the above argument yields $\Upsilon_F^{\circ i}(G) \in \bW^{k-2i, r_{2i}}$ for every $i \in \{1, \ldots, j\}$.
Moreover, by repeatedly applying \eqref{Upsilon_F sobolev norm 1}, we obtain the desired bound for all $i \in \{1, \ldots, j\}$, and the proof is complete.
\end{proof}

Combining Proposition \ref{Prp IBP} and Lemma \ref{Lem Upsilon^j estimate} yields the integration by parts formula for $\cA$. 

\begin{Prp}\label{Prp IBP for cA}
Let $F \in \bD^{k+1, p}$, $\Delta_F^{-1} \in L^q(\Omega)$, and $G \in \bD^{k, r_0}$ for some $k \in \N$ and $p, q, r_0 \in (1, \infty)$. 
Define $r_j, \  j \in \{1, \ldots, k+1\}$ by \eqref{r_k def}.
\begin{enumerate}
    \item[(1)] If $r_{2j} > 1$ for some $j \in \{0, 1, \ldots, \ceil{\frac{k-1}{2}}\}$, then for every $i \in \{0, 1, \ldots, j\}$ and for any $\varphi \in \cS(\R)$, 
    \begin{equation}
        \E[(\cA^{i}\varphi)(F)G] = \E[\varphi(F)\Upsilon_F^{\circ i}(G)].
    \end{equation}
    \item[(2)] If $k$ is odd and $\frac{1}{r_{k+1}} + \frac{1}{p} \leq 1$, then for any $\varphi \in \cS(\R)$, 
    \begin{equation}
        \E\left[\left(\cA^{\frac{k+1}{2}}\varphi \right)(F)G \right] = \dpairLR{\bW^{1,r_{k+1}'}}{\varphi(F)}{\Upsilon_F^{\circ \frac{k+1}{2}}(G)}{\bW^{-1, r_{k+1}}}, \qquad \frac{1}{r_{k+1}} + \frac{1}{r_{k+1}'} = 1.
    \end{equation}
\end{enumerate}
\end{Prp}

\begin{proof}
The first statement follows from Proposition \ref{Prp IBP}. 
If $k$ is odd, then we have 
\begin{align}
    \E\left[\left(\cA^{\frac{k+1}{2}}\varphi\right)(F)G\right] 
    &= \E\left[(\cA \varphi)(F) \Upsilon_F^{\circ \frac{k-1}{2}}(G)\right]\\
    &= \E\left[\varphi(F)\cM_F\left(\Upsilon_F^{\circ \frac{k-1}{2}}(G)\right)\right] - \E\left[\varphi^{(2)}(F)\Upsilon_F^{\circ \frac{k-1}{2}}(G)\right].
\end{align}
Since $\Upsilon_F^{\circ \frac{k-1}{2}}(G) \in \bW^{1,r_{k-1}}$, applying Proposition \ref{Prp IBP} yields 
\begin{equation}
    \E\left[\varphi^{(2)}(F)\Upsilon_F^{\circ \frac{k-1}{2}}(G)\right] = \dpairLR{\bW^{1, r_{k+1}'}}{\varphi(F)}{\Theta_F^{(2)}\left(\Upsilon_F^{\circ \frac{k-1}{2}}\right)}{\bW^{-1, r_{k+1}}}
\end{equation}
provided that $\frac{1}{r_{k+1}} + \frac{1}{p} \leq 1$. 
Therefore, the second claim also follows and the proof is complete.
\end{proof}

\begin{Rem}\label{Rem G=1}
By the same reasoning, when $G=1$, Lemmas \ref{Lem Theta_F^j estimate} and \ref{Lem Upsilon^j estimate} and Propositions \ref{Prp IBP} and \ref{Prp IBP for cA} remain valid by setting $\frac{1}{r_0} = 0$ in \eqref{r_k def}. 
\end{Rem}

\begin{Rem}
Note that the estimate in Lemma \ref{Lem delta_F^-1 sobolev norm} is not sharp, and as a result the statements of Section \ref{subsection Integration by parts} imposes slightly stronger assumptions on $F$ and $G$ than actually required. 
The reason for adopting such an estimate is to simplify the appearance of the statements. 
Moreover, our interest in this paper lies in a sequence $\{F_n\}$ that belongs to a fixed order Wiener chaos and converges to the standard normal distribution. 
In this case, thanks to \cite[Theorem 3]{Superconvergence} (see Lemma \ref{Lem asymtotic nondegeneracy}), one can choose $k, p, q$ arbitrarily large so that $\norm{F_n}_{\bW^{k,p}}$ and $\norm{\Delta_{F_n}^{-1}}_{L^q(\Omega)}$ are uniformly bounded for all sufficiently large $n$. 
Therefore, the sharpness of the estimate is of little concern. 
\end{Rem}

\subsection{Estimates of composition operators}\label{subsection Estimates of composition operators}
To justify the composition of tempered distributions and a nice enough random variable $F$, we first consider the linear operator 
\begin{equation}
    \sC_F^{(\alpha, p)} \colon \cS(\R) \to \bW^{\alpha, p}, \qquad \sC_F^{(\alpha, p)}(\varphi) \coloneqq \varphi(F)
\end{equation}
for some $\alpha \in \Z$ and $p \in (1, \infty)$, and then extend it by density to
\begin{equation}
    \wt{\sC_F^{(\alpha, p)}} \colon \cS_{\alpha} \to \bW^{\alpha, p}.
\end{equation}
Since every tempered distribution 
belongs to some $\cS_{-n}, \ n \in \N$, we can rigorously define the composition using $\wt{\sC_F^{(\alpha, p)}}$ with a suitable $\alpha \in \Z$ and $p \in (1, \infty)$.

In the following, we make use of the boundedness of density functions, which is guaranteed by the following result. 
\begin{Lem}[\textit{cf.} {\cite[THEOREM 5.4]{Shigekawa}}]\label{Lem Shigekawa bdd density criterion}
Let $n \in \N$ and $F$ be a random variable. 
Suppose that, for every $m \in \{1, \ldots, n\}$, there exists $K_m \in L^p(\Omega)$ with $p > 1$ such that for all $\varphi \in C^{\infty}_{\mathrm{c}}(\R)$, 
\begin{equation}
    \E[\varphi^{(m)}(F)] = \E[\varphi(F)K_m]. \label{IBP assumption}
\end{equation}
Then, the law of $F$ has a density function $\rho_F \in C_{\mathrm{b}}^{n-1}(\R)$ that satisfies
\begin{align}
    \norm{\rho_F}_{C_{\mathrm{b}}^{n-1}(\R)} \lesssim_{n,p} 
    \begin{cases}
        \norm{K_1}_{L^p(\Omega)}, &\text{when $n = 1$},\\
        \left( 1 + \sum_{m=1}^{n} \norm{K_m}_{L^p(\Omega)} \right)\norm{K_1}_{L^p(\Omega)}^{1-\frac{1}{p}}, &\text{when $n \geq 2$}.
    \end{cases}
    \label{density bound estimate}
\end{align}
Moreover, $\rho^{(j)}_{F} \in \bigcap_{r \in [1,\infty)}L^r(\R)$ holds for every $j \in \{0, \ldots, n-1\}$.
\end{Lem}

\begin{proof}
For the existence of $\rho_F$ in $C_{\mathrm{b}}^{n-1}(\R)$ and \eqref{density bound estimate}, we refer to \cite[THEOREM 5.4]{Shigekawa}.
The integrability of $\rho_F$ is clear since $\rho_F \in C_{\mathrm{b}}^{n-1}(\R) \cap L^1(\R)$. 
Now, let $j \in \{1, \ldots, n-1\}$.
From \eqref{IBP assumption}, 
\begin{equation}
    \int_{\R}\varphi^{(j)}(x)\rho_F(x)dx = \int_{\R}\varphi(x)\E[K_j \mid F=x]\rho_F(x)dx
\end{equation}
holds for all $\varphi \in C^{\infty}_{\mathrm{c}}(\R)$, and it follows that
\begin{equation}
    \rho_F^{(j)}(x) = (-1)^{j}\E[K_j \mid F=x]\rho_F(x), \qquad \text{a.e.}.
\end{equation}
Therefore, we obtain
\begin{align}
    \int_{\R}\left|\rho_F^{(j)}(x)\right|^r dx 
    &\leq \sup_{a \in \R}\left|\rho_F^{(j)}(a)\right|^{r-1} \int_{\R}|\E[K_j \mid F=x] |\rho_F(x)dx\\
    &\leq \sup_{a \in \R}\left|\rho_F^{(j)}(a)\right|^{r-1} \norm{K_j}_{L^1(\Omega)} < \infty,
\end{align}
and the proof is complete.
\end{proof}

In order to extend the operators, we first check that $\sC^{(\alpha, p)}_{F} \colon \cS_{\alpha} \to W^{\alpha, p}$ with $\Dom(\sC^{(\alpha, p)}_{F}) = \cS(\R)$ is continuous.
We begin with the cases where $\alpha = 0$ and $\alpha = 1$.

\begin{Prp}\label{Prp Pullback estimate nonnegative index}
Let $F \in \bW^{2,p}$ and $\Delta_F^{-1} \in L^q(\Omega)$ for $p \in [2, \infty)$ and $q \in (1,\infty)$ satisfying $\frac{3}{p} + \frac{2}{q} < 1$. 
Then, for all $r \in \left(1, \frac{2p}{2+p}\right]$ and $\varphi \in \cS(\R)$, 
\begin{align}
    \norm{\varphi(F)}_{L^r(\Omega)} 
    &\lesssim_{p,q} \norm{F}_{\bW^{2,p}(\fH)}^{\frac{3}{2}} \norm{\Delta_F^{-1}}_{L^q(\Omega)} \norm{\varphi}_{\cS_0},  \label{g(F) D0,r}\\
    \norm{\varphi(F)}_{\bW^{1, r}} 
    &\lesssim_{p, q, r} \left(1 + \norm{F}_{\bW^{1,p}}\right)\norm{F}_{\bW^{2,p}}^{\frac{3}{2}} \norm{\Delta_F^{-1}}_{L^q(\Omega)}\norm{\varphi}_{\cS_1}. \label{g(F) D1,r}
\end{align}
\end{Prp}

\begin{proof}
By Proposition \ref{Prp IBP} and Lemma \ref{Lem Shigekawa bdd density criterion} (see also Remark \ref{Rem G=1}), the law of $F$ has a density function $\rho_F \in C_{\mathrm{b}}(\R)$, and it satisfies 
\begin{equation}
    \sup_{x \in \R} |\rho_F(x)| 
    \lesssim_{p, q} \norm{\Theta_F(1)}_{L^{\frac{pq}{2p+3q}}(\Omega)} 
    \lesssim_{p,q} \norm{F}_{\bW^{2,p}}^3\norm{\Delta_F^{-1}}_{L^q(\Omega)}^2  \label{rho_F sup norm1}.
\end{equation}
From \eqref{rho_F sup norm1} and $\norm{\cdot}_{L^2(\R)} = \norm{\cdot}_{\cS_0}$, we obtain
\begin{align}
    \norm{\varphi(F)}_{L^r(\Omega)} 
    &\leq \norm{\varphi(F)}_{L^2(\Omega)} = \left(\int_{\R} |\varphi(x)|^2 \rho_F(x) dx\right)^{\frac{1}{2}}\\
    &\lesssim_{p,q}  \norm{F}_{\bW^{2,p}}^\frac{3}{2}\norm{\Delta_F^{-1}}_{L^q(\Omega)}  \norm{\varphi}_{\cS_0}.
\end{align}
If $r \in \left(1, \frac{2p}{2+p}\right]$, then $\frac{1}{r} \geq \frac{1}{p} + \frac{1}{2}$, and  we see from Meyer's norm equivalence that
\begin{align}
    \norm{\varphi(F)}_{\bW^{1, r}}
    &\lesssim_r \norm{\varphi(F)}_{L^r(\Omega)} + \left\lVert\varphi'(F)\norm{DF}_{\fH} \right\rVert_{L^r(\Omega)}\\ 
    &\leq \norm{\varphi(F)}_{L^2(\Omega)} + \norm{\varphi'(F)}_{L^{2}(\Omega)}\norm{DF}_{L^p(\Omega; \fH)}. 
\end{align}
Thus, the same computation yields 
\begin{equation}
    \norm{\varphi(F)}_{\bW^{1, r}} \lesssim_{p,q,r} \norm{F}_{\bW^{2,p}}^\frac{3}{2}\norm{\Delta_F^{-1}}_{L^q(\Omega)}  \norm{\varphi}_{\cS_0} + \norm{F}_{\bW^{2,p}}^\frac{3}{2}\norm{\Delta_F^{-1}}_{L^q(\Omega)} \norm{DF}_{L^p(\Omega; \fH)}\norm{\varphi'}_{\cS_0}.
\end{equation}
Since $\norm{\varphi'}_{\cS_0} \lesssim \norm{\varphi}_{\cS_1}$ by Proposition \ref{Prp tempered dis class} and $\norm{\varphi}_{\cS_0} \leq \norm{\varphi}_{\cS_1}$, we deduce the desired estimate.
\end{proof}

Next, in the case where the differentiability index is a negative integer, we reduce it to the case of Proposition \ref{Prp Pullback estimate nonnegative index} by repeatedly applying the integration by parts formula for $\cA$ (Proposition \ref{Prp IBP for cA}).

\begin{Prp}\label{Prp Pullback est}
Let $F \in \bW^{k+1, p}$ and $\Delta_F^{-1} \in L^q(\Omega)$ for some $k \in \N$ and $p, q\in (1, \infty)$. 
Assume that $p,q \in (1, \infty)$ satisfy
\begin{align}
\label{s cond 1}
    \begin{cases}
        \dis \frac{1}{s} \geq \frac{(k+1)^2}{p} + \frac{(k+1)(k+2)}{2q} + \frac{1}{p} + \frac{1}{2}, &\text{if $k$ is odd},\\[3ex]
        \dis \frac{1}{s} \geq \frac{k(k+2)}{p} + \frac{k(k+3)}{2q} + \frac{1}{p} + \frac{1}{2}, &\text{if $k$ is even},
    \end{cases}
\end{align}
for some $s \in (1, \infty)$. 
Then, for any $\varphi \in \cS(\R)$,  
\begin{equation}
    \norm{\varphi(F)}_{\bW^{-k, s}} 
    \lesssim_{k, p, q, s} C_k \left(\norm{F}_{\bW^{k+1, p}}, \norm{\Delta_{F}^{-1}}_{L^q(\Omega)}\right) \norm{\varphi}_{\cS_{-k}},
\end{equation}
where 
\begin{align}
    &C_k \left(\norm{F}_{\bW^{k+1, p}}, \norm{\Delta_{F}^{-1}}_{L^q(\Omega)}\right) \\
    &= 
    \begin{cases}
        (1+\norm{F}_{\bW^{k+1, p}}^2)^{\frac{2(k+1)^2 + 5}{4}}(1+\norm{\Delta_F^{-1}}_{L^q(\Omega)})^{\frac{(k+1)(k+2)+2}{2}}, &\text{if $k$ is odd},\\[5ex]
        (1+\norm{F}_{\bW^{k+1, p}}^2)^{\frac{2k(k+2) + 3}{4}}(1+\norm{\Delta_F^{-1}}_{L^q(\Omega)})^{\frac{k(k+3)+2}{2}}, &\text{if $k$ is even}.
    \end{cases}
\end{align}
\end{Prp}

\begin{proof}
Let $s' \in (1, \infty)$ satisfy $\frac{1}{s} + \frac{1}{s'} = 1$. 
By Proposition \ref{Prp IBP for cA}, we obtain
\begin{align}
    \norm{\varphi(F)}_{\bW^{-k,s}} 
    &= \sup_{\substack{G \in \bW^{k, s'} \\ \norm{G}_{\bW^{k, s'}} \leq 1 }} |\E[\varphi(F) G]|\\
    &= \sup_{\substack{G \in \bW^{k, s'} \\ \norm{G}_{\bW^{k, s'}} \leq 1 }} \left|\E\left[\left(\cA^{\ceil{\frac{k}{2}}}\cA^{-\ceil{\frac{k}{2}}}\varphi\right)(F)G\right]\right|\\
    &= 
    \begin{cases}
        \dis \sup_{\substack{G \in \bW^{k, s'} \\ \norm{G}_{\bW^{k, s'}} \leq 1 }} \left| \dpairLR{\bW^{1, \alpha}}{\left(\cA^{-\frac{k+1}{2}}\varphi\right)(F)}{\Upsilon_F^{\circ \frac{k+1}{2}}(G)}{\bW^{-1, \alpha'}}\right|, &\text{if $k$ is odd},\\[5ex]
        \dis \sup_{\substack{G \in \bW^{k, s'} \\ \norm{G}_{\bW^{k, s'}} \leq 1 }} \left|\dpairLR{L^{\beta}(\Omega)}{\left(\cA^{-\frac{k}{2}}\varphi\right)(F)}{\Upsilon_F^{\circ \frac{k}{2}}(G)}{L^{\beta'}(\Omega)}\right|, &\text{if $k$ is even},
    \end{cases}
\end{align}
where $1 = \frac{1}{\alpha} + \frac{1}{\alpha'} = \frac{1}{\beta} + \frac{1}{\beta'}$ and 
\begin{align}
    \frac{1}{\alpha'} &\coloneqq \frac{1}{s'} + \sum_{i=0}^{k} \left( \frac{2k+1-2i}{p} + \frac{k+1-i}{q} \right) = \frac{1}{s'} + \frac{(k+1)^2}{p} + \frac{(k+1)(k+2)}{2q}, \\
     \frac{1}{\beta'} &\coloneqq \frac{1}{s'} + \sum_{i=0}^{k-1} \left( \frac{2k+1-2i}{p} + \frac{k+1-i}{q} \right) = \frac{1}{s'} + \frac{k(k+2)}{p} + \frac{k(k+3)}{2q}.
\end{align}
It follows from Lemma \ref{Lem Upsilon^j estimate} that
\begin{align}
    &\norm{\varphi(F)}_{\bW^{-k, s}} \\
    &\leq 
    \begin{cases}
        \dis \norm{(\cA^{-\frac{k+1}{2}}\varphi)(F)}_{\bW^{1, \alpha}} \sup_{\substack{G \in \bW^{k, s'} \\ \norm{G}_{\bW^{k, s'}} \leq 1 }}\norm{\Upsilon^{\circ \frac{k+1}{2}}(G)}_{\bW^{-1, \alpha'}}, &\text{if $k$ is odd},\\[5ex]
        \dis \norm{(\cA^{-\frac{k}{2}}\varphi)(F)}_{L^{\beta}(\Omega)} \sup_{\substack{G \in \bW^{k, s'} \\ \norm{G}_{\bW^{k, s'}} \leq 1 }}\norm{\Upsilon^{\circ \frac{k}{2}}(G)}_{L^{\beta'}(\Omega)}, &\text{if $k$ is even},
    \end{cases}\\
    &\lesssim_{k, p, q, s} 
    \begin{cases}
        (1+\norm{F}_{\bW^{k+1, p}}^2)^{\frac{(k+1)^2}{2}}(1+\norm{\Delta_F^{-1}}_{L^q(\Omega)})^{\frac{(k+1)(k+2)}{2}}\norm{(\cA^{-\frac{k+1}{2}}\varphi)(F)}_{\bW^{1, \alpha}}, &\text{if $k$ is odd},\\[5ex]
        (1+\norm{F}_{\bW^{k+1, p}}^2)^{\frac{k(k+2)}{2}}(1+\norm{\Delta_F^{-1}}_{L^q(\Omega)})^{\frac{k(k+3)}{2}}\norm{(\cA^{-\frac{k}{2}}\varphi)(F)}_{L^{\beta}(\Omega)}, &\text{if $k$ is even}.
    \end{cases}
\end{align}
From \eqref{s cond 1}, we find that $\frac{3}{p} + \frac{2}{q} < 1$ and that
\begin{align}
    \frac{1}{\alpha} &= \frac{1}{s} - \frac{(k+1)^2}{p} - \frac{(k+1)(k+2)}{2q} \geq \frac{1}{p} + \frac{1}{2},\\
    \frac{1}{\beta} &= \frac{1}{s} - \frac{k(k+2)}{p} - \frac{k(k+3)}{2q} \geq \frac{1}{p} + \frac{1}{2},
\end{align}
which imply $\alpha, \beta \in (1, \frac{2p}{2+p}]$. 
Therefore, we can apply Proposition \ref{Prp Pullback estimate nonnegative index} to obtain 
\begin{align}
    &\norm{\varphi(F)}_{\bW^{-k, s}}\\
    &\lesssim_{k, p, q, s} 
    \begin{cases}
        (1+\norm{F}_{\bW^{k+1, p}}^2)^{\frac{2(k+1)^2 + 5}{4}}(1+\norm{\Delta_F^{-1}}_{L^q(\Omega)})^{\frac{(k+1)(k+2)+2}{2}}\norm{\cA^{-\frac{k+1}{2}}\varphi}_{\cS_{1}}, &\text{if $k$ is odd},\\[5ex]
        (1+\norm{F}_{\bW^{k+1, p}}^2)^{\frac{2k(k+2) + 3}{4}}(1+\norm{\Delta_F^{-1}}_{L^q(\Omega)})^{\frac{k(k+3)+2}{2}}\norm{\cA^{-\frac{k}{2}}\varphi}_{\cS_0}, &\text{if $k$ is even},
    \end{cases}\\[2ex]
    &=
    \begin{cases}
        (1+\norm{F}_{\bW^{k+1, p}}^2)^{\frac{2(k+1)^2 + 5}{4}}(1+\norm{\Delta_F^{-1}}_{L^q(\Omega)})^{\frac{(k+1)(k+2)+2}{2}}\norm{\varphi}_{\cS_{-k}}, &\text{if $k$ is odd},\\[5ex]
        (1+\norm{F}_{\bW^{k+1, p}}^2)^{\frac{2k(k+2) + 3}{4}}(1+\norm{\Delta_F^{-1}}_{L^q(\Omega)})^{\frac{k(k+3)+2}{2}}\norm{\varphi}_{\cS_{-k}}, &\text{if $k$ is even},
    \end{cases}
\end{align}
where we use 
\begin{equation}
    (1 + \norm{F}_{\bW^{1,p}})\norm{F}_{\bW^{2,p}}^{\frac{3}{4}} \leq \sqrt{2}(1 + \norm{F}_{\bW^{k+1, p}}^2)^{\frac{5}{4}}
\end{equation}
in the second line, and \eqref{S_alpha norm by cA} in the third line.
This completes the proof.
\end{proof}

In view of Proposition \ref{Prp Pullback est}, for all $j \in \N$ and $p,q \in (1,\infty)$, we define $s_j = s_j[p,q]$ by
\begin{align}
    \frac{1}{s_j[p,q]} = 
    \begin{cases}
        \dis \frac{(j+1)^2}{p} + \frac{(j+1)(j+2)}{2q} + \frac{1}{p} + \frac{1}{2}, &\text{if $j$ is odd},\\[3ex]
        \dis \frac{j(j+2)}{p} + \frac{j(j+3)}{2q} + \frac{1}{p} + \frac{1}{2}, &\text{if $j$ is even}.
    \end{cases}
\label{def s_j}
\end{align}
For the sake of readability, whenever the parameters $p$ and $q$ are clear from the context, we shall omit them and simply write $s_j$. 
The conjugate exponent of $s_j$ is denoted by $s_j'$.

We now extend the composition operators.

\begin{Cor}
\label{Cor pullback map}
Let $F \in \bD^{k+1, p}$ and $\Delta_F^{-1} \in L^q(\Omega)$ for some $k \in \N$ and $p, q\in (1, \infty)$. 
If $s_j = s_j[p,q] > 1$ for some $j \in \{1, \ldots, k\}$, then for any $i \in \{1, \ldots, j\}$, the linear operator
\begin{equation}
    \sC_F^{(-i, s_i)} \colon \cS_{-i} \to \bW^{-i,s_i}, \qquad \Dom \left(\sC_F^{(-i, s_i)}\right) = \cS(\R),  \qquad \varphi \mapsto \varphi(F) 
\end{equation}
is continuous and can be uniquely extended to the continuous linear operator 
\begin{equation}
    \wt{\sC_F^{(-i, s_i)}} \colon \cS_{-i} \to \bW^{-i,s_i}, \qquad  \Dom \left(\wt{\sC_F^{(-i, s_i)}}\right) = \cS_{-i}. 
\end{equation}
Moreover, for any $i_1, i_2 \in \{1, \ldots, j\}$ with $i_1 \leq i_2$ and any $T \in \cS_{-i_1}$, we have
\begin{equation}
    \wt{\sC_F^{(-i_2, s_{i_2})}}(T) = \wt{\sC_F^{(-i_1, s_{i_1})}}(T) \quad \text{in $\bW^{-i_1, s_1}$}. \label{composition compatibility 1}
\end{equation}
\end{Cor}

\begin{proof}
We have $s_1 > \cdots > s_j > 1$, so the extensibility is obvious from Proposition \ref{Prp Pullback est}.
If $T \in \cS_{-i_1}$, then there exists a sequence $\{\varphi_n\}_{n \in \N} \subset \cS(\R)$ such that $\varphi_n \to T$ in $\cS_{-i_1}$ as $n \to \infty$.
In this case, we also have $\varphi_n \to T$ in $\cS_{-i_2}$. 
Since $\wt{\sC_F^{(-i_2, s_{i_2})}}(\varphi_n) = \wt{\sC_F^{(-i_1, s_{i_1})}}(\varphi_n) = \varphi_n(F)$ and 
\begin{equation}
    \wt{\sC_F^{(-i_2, s_{i_2})}}(\varphi_n) \to \wt{\sC_F^{(-i_2, s_{i_2})}}(T) \quad \text{in $\bW^{-i_2, s_{i_2}}$} \quad \text{and} \quad \wt{\sC_F^{(-i_1, s_{i_1})}}(\varphi_n) \to \wt{\sC_F^{(-i_1, s_{i_1})}}(T) \quad \text{in $\bW^{-i_1, s_{i_1}}$}, 
\end{equation}
we obtain \eqref{composition compatibility 1}, and the proof is complete.
\end{proof}

Let $F \in \bD^{k+1, p}$ and $\Delta_F^{-1} \in L^q(\Omega)$ for some $k \in \N$ and $p, q\in (1, \infty)$, and let $s_j > 1$ for some $j \in \{1, \ldots, k\}$, where $s_j = s_j[p,q]$ is defined by \eqref{def s_j}.
In this situation, by Corollary \ref{Cor pullback map}, we define the composition $T(F)$ by $\wt{\sC_F^{-j, s_j}}(T)$ for every $T \in \cS_{-j}$.
In view of \eqref{composition compatibility 1}, if $T \in \cS_{-i} \subset \cS_{-j}$ for some $i \in \{1, \ldots, j\}$, then we have $T(F) \in \bW^{-i, s_i}$. 
For example, by Propositions \ref{Prp tempered dis class}, \ref{Prp derivative uniform estimate}, and \ref{Prp f_a,n class Dfan uniform bound}, we have
\begin{equation}
    1(F) \in \bW^{-1, s_1} \qquad \text{and} \qquad \sD^{i-1}\delta_a(F), \  \sD f_{a, i-1}(F) \in \bW^{-i, s_i}
\end{equation}
for any $a \in \R$ and $i \in \{1, \ldots, j\}$.
As these composition operators are defined for each $F$, whenever it is clear which $s_j[p,q]$ is being used for the definition, we shall, hereafter, simplify the notation by writing the composition as $T(F)$ instead of $\wt{\sC_F^{-j, s_j}}(T)$.

\begin{Rem}
\label{Rem comparison with Watanabe framework}
In \cite{SW-AWF, IkedaWatanabe}, the spaces $\sS_{2k}, \ k \in \Z$, obtained by completing $\cS(\R^d)$ with respect to the norm
\begin{equation}
    \norm{\varphi}_{\sS_{2k}} \coloneqq \norm{\cA^k\varphi}_{L^{\infty}(\R^d)}, \quad \varphi \in \cS(\R^d), 
\end{equation}
is used. 
It is shown that, in the case $d = 1$, the composition $T(F)$ of $T \in \cS'(\R)$ with a random variable $F \in \bW^{\infty}$ satisfying $\Delta_F^{-1} \in \bigcap_{q \in [1,\infty)} L^q(\Omega)$  belongs to $\bigcup_{k=1}^{\infty}\bigcap_{1< s < \infty} \bW^{-k,s}$.
On the other hand, in this paper, we work with the spaces $\cS_{\alpha}, \ \alpha \in \R$, equipped with the norm
\begin{equation}
    \norm{\varphi}_{\cS_{\alpha}} = \norm{\cA^{\frac{\alpha}{2}}\varphi}_{L^2(\R)}, \quad \varphi \in \cS(\R), 
\end{equation}
which is defined in terms of the $L^2(\R)$-norm. 
As a result, even if $F \in \bW^{\infty}$ satisfies the same property, we can only conclude $T(F) \in \bigcup_{k=1}^{\infty}\bigcap_{1< s < 2} \bW^{-k,s}$. 
The restriction on the integrability index $s$ arises from evaluating the $L^r(\Omega)$ (or  $\bW^{1,r}$)-norm in terms of the $\cS_{0}$ (or $\cS_1$)-norm in Proposition \ref{Prp Pullback estimate nonnegative index}.
This is one of the drawbacks of using the spaces $(\cS_{\alpha}, \norm{\cdot}_{\cS_{\alpha}}), \ \alpha \in \R$. 
However, by using these spaces, we can easily identify the specific space $\cS_{\alpha}$ that contains the tempered distributions (such as $\sD f_{a,n}$, \ $g f_{a,n}$, \  $\sD^n\delta_a$, \  $1$) involved in this analysis, as was done in Section \ref{section Sobolev spaces for tempered distributions}. 
In \cite{FractionalSobolev}, instead of considering the composition of all tempered distributions, the composition of tempered distributions contained in the space
\begin{equation}
    \sL \coloneqq \bigcup_{\alpha \in \R} \bigcup_{1 < p < \infty} \sL^{\alpha, p}(\R^d) \subsetneq \cS'(\R^d)
\end{equation}
is studied, where 
\begin{equation}
    \sL^{\alpha, p}(\R^d) \coloneqq \left\{ u \in \cS'(\R^d) \mid \text{there exists $\wt{u} \in L^p(\R^d)$ such that $u = \cF^{-1}(1 + |\cdot|^2)^{-\frac{\alpha}{2}}\cF \wt{u}$ \  in $\cS'(\R^d)$} \right\}
\end{equation}
is the Bessel potential space (\textit{cf.} \cite[Chapter 7]{Adams}). 
However, since $\cF^{-1}(1 + |\cdot|^2)^{\frac{\alpha}{2}}\cF 1 = 1 \notin L^p(\R^d)$ for $1 < p < \infty$, the space $\sL$ does not contain constant functions and may not be suitable for this study.
\end{Rem}

\subsection{Composition with delta and constant functions}\label{subsection Composition with delta and constant functions}
We first show that for a suitable random variable $F$, the generalized expectation \eqref{generalized expectation of D^ndelta_a} coincides with $\rho_F^{(n)}(a)$. 

\begin{Prp}\label{Prp density derivative formula proof}
Let $F \in \bW^{k+1, p}$ and $\Delta_F^{-1} \in L^q(\Omega)$ for some $k \in \N$ and $p, q\in (1, \infty)$. 
Let $s_j = s_j[p,q]$ be defined by \eqref{def s_j}.
If $s_j > 1$ for some $j \in \{1, \ldots, k\}$, then the law of $F$ has a density function $\rho_F \in C_{\mathrm{b}}^{j-1}(\R)$.
Moreover, for all $a \in \R$ and every $i \in \{0, \ldots, j-1\}$,
\begin{equation}
    \rho_F^{(i)}(a) = (-1)^i \dpairLR{\bW^{-i-1, s_{i+1}}}{(\sD^i\delta_a)(F)}{1}{\bD^{i+1, s_{i+1}'}}, \quad \frac{1}{s_{i+1}} + \frac{1}{s_{i+1}'} = 1.
\end{equation}
\end{Prp}

\begin{proof}
The first claim follows from Lemma \ref{Lem Shigekawa bdd density criterion}. 
For the second claim, it suffices to show that the function $x \mapsto \dpair{\bW^{-i-1, s_{i+1}}}{(\sD^i\delta_x)(F)}{1}{\bW^{i+1, s_{i+1}'}}$ is continuous and
\begin{equation}
    \int_{\R}\psi(x) \rho_F^{(i)}(x)dx = (-1)^i\int_{\R} \psi(x) \dpairLR{\bD^{-i-1, s_{i+1}}}{(\sD^i\delta_x)(F)}{1}{\bD^{i+1, s_{i+1}'}}dx  
\end{equation}
holds for all $\psi \in C_{\mathrm{c}}^{\infty}(\R)$.
By Propositions \ref{Prp tempered dis class}, \ref{Prp derivative uniform estimate}, and \ref{Prp Pullback est}, we have $\sD^{n}\delta_x \in \cS_{-n-1}$ and 
\begin{align}
    \left|\dpair{\bW^{-i-1, s_{i+1}}}{\sD^i\delta_x(F)}{1}{\bW^{i+1, s_{i+1}'}}\right| 
    &\leq \norm{(\sD^i\delta_x)(F)}_{\bW^{-i-1, s_{i+1}}} \\
    &\lesssim_{i,k,p,q,F} \norm{\sD^i\delta_x}_{\cS_{-i-1}} \label{D^i bound}\\
    &\lesssim_{i} \norm{\delta_x}_{\cS_{-1}}
\end{align}
for every $x \in \R$. 
Thus, the continuity of $x \mapsto \dpair{\bW^{-i-1, s_{i+1}}}{(\sD^i\delta_x)(F)}{1}{\bW^{i+1, s_{i+1}'}}$ follows once we show the continuity of $x \mapsto \delta_x$ in $\cS_{-1}$.
Since $\dis \lim_{y \to x}|\phi_n(x) - \phi_n(y)|^2 = 0$ and 
\begin{equation}
    \sum_{k=0}^{\infty}(k+1)^{-1}|\phi_k(x) - \phi_k(y)|^2 \leq 4 \sum_{k=0}^{\infty}(k+1)^{-1}\norm{\phi_k}_{\infty}^2 < \infty,
\end{equation}
we see from the dominated convergence theorem that 
\begin{equation}
    \lim_{y \to x}\norm{\delta_x - \delta_y}^2_{\cS_{-1}} = \lim_{y \to x} \frac{1}{2}\sum_{k=0}^{\infty}(k+1)^{-1}|\phi_k(x) - \phi_k(y)|^2 =0, 
\end{equation}
which shows the desired continuity. 
Let $\lambda_m \colon \R \to \R$ be the standard (symmetric) mollifier. 
In $\cS'(\R)$, we have 
\begin{equation}
    (\sD^i \delta_x) \ast \lambda_m = \delta_x \ast \lambda_m^{(i)} = (\delta_x \ast \lambda_m)^{(i)} \in \cS(\R),
\end{equation}
and hence, 
\begin{align}
    \norm{(\sD^i \delta_x) \ast \lambda_m - \sD^{i} \delta_x}_{\cS_{-i-1}}^2 
    &= \norm{\sD^i (\delta_x \ast \lambda_m - \delta_x)}_{\cS_{-i-1}}^2\\
    &\lesssim_{i} \norm{\delta_x \ast \lambda_m - \delta_x}_{\cS_{-1}}^2\\
    &= \frac{1}{2}\sum_{n=0}^{\infty}(n+1)^{-1}|(\phi_n \ast \lambda_m)(x) - \phi_n(x)|^2\\
    &\xrightarrow[]{m \to \infty} 0,
\end{align}
where the last step follows from $|(\phi_n \ast \lambda_m)(x)| \leq \norm{\phi_n}_{\infty}$ and the dominated convergence theorem. 
From this and \eqref{D^i bound}, we see that
\begin{align}
    \dpair{\bW^{-i-1, s_{i+1}}}{\sD^i\delta_x(F)}{1}{W^{i+1, s_{i+1}'}}
    &= \lim_{m \to \infty} \dpair{W^{-i-1, s_{i+1}}}{(\sD^i\delta_x \ast \lambda_m)(F)}{1}{W^{i+1, s_{i+1}'}} \\
    &= \lim_{m \to \infty} \E[\lambda_m^{(i)}(F - x)]\\
    &= \lim_{m\to \infty} (-1)^i\int_{\R} \lambda_m(y-x)\rho_F^{(i)}(y)dy.
\end{align}
Therefore, we obtain
\begin{align}
    (-1)^i\int_{\R} \psi(x) \dpair{\bW^{-i-1, s_{i+1}}}{\sD^i\delta_x(F)}{1}{W^{i+1, s_{i+1}'}} dx
    &= \lim_{m \to \infty} \int_{\R}(\psi \ast \lambda_m)(x)\rho_F^{(i)}(x)dx \\
    &= \int_{\R} \psi(x)\rho_F^{(i)}(x)dx,
\end{align}
which completes the proof.
\end{proof}

Regarding the constant function $1$ as a tempered distribution, we can define $1(F) \coloneqq \wt{\sC_F^{(-1, s_1)}}(1)$ for a suitable random variable $F$ by Corollary \ref{Cor pullback map}.
However, we do not know whether $1(F) \equiv 1$ holds in advance, which must be verified.
To clarify the notation, in the following proposition, the composition of $T \in \cS'(\R)$ with $F$ will be denoted by $T \circ F$, and the constant function with value $c \in \mathbb{R}$ will be denoted by $\bar{c}$.

\begin{Prp}[\textit{cf.} {\cite[PROPOSITION 1]{SW-generalizedWienerfunctionals}}] \mbox{}
\label{Prp const comp}
\begin{enumerate}
    \item[\normalfont{(1)}] Let $F \in \bW^{3, p}$ and $\Delta_F^{-1} \in L^q(\Omega)$ for $p, q\in (1, \infty)$ satisfying 
    \begin{equation}
        \frac{9}{p} + \frac{5}{q} < \frac{1}{2}. \label{s_2 > 1}
    \end{equation}
    Then, for any $h \in C_{\mathrm{b}}(\R)$, $h \circ F = h(F)$.
    \item[\normalfont{(2)}] Let $F \in \bW^{2, p}$ and $\Delta_F^{-1} \in L^q(\Omega)$ for $p, q\in (1, \infty)$ satisfying 
    \begin{equation}
        \frac{5}{p} + \frac{3}{q} < \frac{1}{2}. \label{s_1 > 1}
    \end{equation}
    Then, for any $c \in \R$, $\bar{c} \circ F = (c \cdot \bar{1}) \circ F = c \cdot (\bar{1} \circ F) = c$.
\end{enumerate}
\end{Prp}

\begin{proof}
(1) By (iv) of Lemma \ref{Lem 1d Hermite functions property},
\begin{equation}
        \sum_{n=0}^{\infty}(n+1)^{-2}\left|\int_{\R}h(x)\phi_n(x)dx\right|^2 
    \leq \norm{h}_{\infty}\sum_{n=0}^{\infty}(n+1)^{-2}\norm{\phi_n}_{L^1(\R)}^2 < \infty,
\end{equation}
holds for any $h \in C_{\mathrm{b}}(\R)$, and hence, $C_{\mathrm{b}}(\R) \subset \cS_{-2}$.   
Inequality \eqref{s_2 > 1} implies that $s_2 = s_2[p,q] > 1$, where $s_2$ is defined in \eqref{def s_j}.  
It follows from Corollary \ref{Cor pullback map} that $h \circ F \coloneqq \sC_F^{(-2, s_2)}(h) \in \bW^{-2, s_2}$ is well-defined for all $h \in C_{\mathrm{b}}(\R)$. 
If we take a sequence $\{h_m\}_m \subset \cS(\R)$ such that $\norm{h_m}_{\infty} \leq \norm{h}_{\infty}$ and $h_m(x) \xrightarrow[]{m \to \infty} h(x)$ for all $x \in \R$, then 
\begin{equation}
    \norm{h-h_m}_{\cS_{-2}}^2 = \frac{1}{4} \sum_{n=0}^{\infty}(n+1)^{-2}\left|\int_{\R}(h(x) - h_m(x))\phi_n(x)dx\right|^2  \xrightarrow[]{m \to \infty} 0
\end{equation}
by the dominated convergence theorem. 
Thus, by the continuity of $\sC_F^{(-2, s_2)}$, we obtain
\begin{equation}
    h_m(F) = \sC_F^{(-2, s_2)}(h_m) \xrightarrow[]{m \to \infty} \sC_F^{(-2, s_2)}(h) = h \circ F \quad \text{in $\bW^{-2, s_2}$}.
\end{equation}
On the other hand, it is easy to see that $h_m(F) \to h(F)$ in $L^2(\Omega)$, and therefore, $h \circ F = h(F)$.

\vspace{1ex}
\noindent
(2) Inequality \eqref{s_1 > 1} implies that $s_1 = s_1[p,q] > 1$, where $s_1$ is defined in \eqref{def s_j}. 
Since $\bar{c} = c \cdot \bar{1}$ and $\bar{1} \in \cS_{-1}$ by Proposition \ref{Prp tempered dis class}, we see from Corollary \ref{Cor pullback map} that $\bar{c} \circ F \coloneqq \sC_F^{(-1, s_1)}(\bar{c}) \in \bW^{-1, s_1}$ is well-defined for any $c \in \R$ and 
\begin{equation}
    \bar{c} \circ F = \sC_F^{(-1, s_1)}(\bar{c}) = \sC_F^{(-1, s_1)}(c \cdot\bar{1}) = c \cdot \sC_F^{(-1, s_1)}(\bar{1}) = c\cdot (\bar{1} \circ F).
\end{equation}
It remains to prove that $\bar{1} \circ F = \bar{1}$. 
Let us consider the sequence $\{e^{-\frac{|\cdot|^2}{m}}\}_{m = 3}^{\infty} \subset \cS(\R)$. 
Clearly, $e^{-\frac{|x|^2}{m}} \xrightarrow[]{m \to \infty} \bar{1}(x)$ for all $x \in \R$, and therefore $e^{-\frac{F^2}{m}} \xrightarrow[]{m \to \infty} \bar{1}$ in $L^2(\Omega)$.
By Lemma \ref{Lem Hermite integral}, we obtain
\begin{align}
    \norm{e^{-\frac{|\cdot|^2}{m}} - 1}_{\cS_{-1}}^2 
    &= \frac{1}{2}\sum_{n=0}^{\infty}(n+1)^{-1}\left|\int_{\R}\phi_n(x)dx - \int_{\R}\phi_n(x)e^{-\frac{|x|^2}{m}}dx\right|^2\\
    &\leq \frac{1}{2}\sum_{k=0}^{\infty}(2k+1)^{-1}\left|\int_{\R}\phi_{2k}(x)dx\right|^2\left|1 - \sqrt{\frac{m}{m+2}}\left(\frac{m-2}{m+2}\right)^k\right|^2\\
    &\xrightarrow[]{m \to \infty} 0.
\end{align}
Here, the last step follows from the dominated convergence theorem, which can be justified since
\begin{equation}
    \sum_{k=0}^{\infty}(2k+1)^{-1}\left|\int_{\R}\phi_{2k}(x)dx\right|^2 = 2\norm{\bar{1}}_{\cS_{-1}}^2 < \infty. 
\end{equation}
Therefore, by the continuity of $\sC_F^{(-1, s_1)}$, we have
\begin{equation}
    e^{-\frac{F^2}{m}} = \sC_F^{(-1, s_1)}\left(e^{-\frac{|\cdot|^2}{m}}\right) \xrightarrow[]{m \to \infty} \sC_F^{(-1, s_1)}\left(\bar{1}\right) = \bar{1}\circ F \quad \text{in $\bW^{-1, s_1}$}.
\end{equation}
This, together with the $L^2(\Omega)$-convergence of $e^{-\frac{F^2}{m}}$, implies that $\bar{1}\circ F = \bar{1}$, which completes the proof.
\end{proof}

\section{Optimal local central limit theorems}\label{section Optimal local central limit theorems}

We prove Theorems \ref{Thm Main results 1} and \ref{Thm Main results 2} in Sections \ref{subsection Proof of main results 1} and \ref{subsection Proof of main results 2}, respectively. 
To this end, in Section \ref{subsection Gamma operators and cumulants}, we first recall the Gamma operators and their connection with cumulants. 
We then extend a result from \cite{OptimalBErates} to derive an Edgeworth-type expansion for density functions in Section \ref{subsection Edgeworth-type expansions for density}.
To illustrate the power of Theorems \ref{Thm Main results 1} and \ref{Thm Main results 2}, we study density convergence in the Breuer--Major theorem in Section \ref{subsection Optimal local Breuer--Major CLT} and provide examples of optimal rates.

\subsection{Gamma operators and cumulants}\label{subsection Gamma operators and cumulants}
Gamma operators $\{\Gamma_m\}_{m \in \N_0}$, introduced in \cite{NPCumulants}, are nonlinear operators defined recursively by 
\begin{equation}
    \Gamma_0(F) = F \qquad \text{and} \qquad \Gamma_{m+1}(F) = \abra{DF, -DL^{-1}\Gamma_{m}(F)}_{\fH}, \qquad m \in \N_0,
\end{equation}
for suitable Gaussian functionals $F$.
It is shown in \cite[Lemma 4.2 and Theorem 4.3]{NPCumulants} that if $F \in \bW^{\infty}$, then for every $m \in \N_0$, $\Gamma_m(F) \in \bW^{\infty}$ and the $(m+1)$th cumulant of $F$, denoted by $\kappa_{m+1}(F)$, admits the expression 
\begin{equation}
    \kappa_{m+1}(F) = m!\E[\Gamma_m(F)].
\end{equation}
Moreover, if $F$ belongs to a fixed order Wiener chaos, the following estimates for the Gamma operators are established.

\begin{Prp}[\textit{cf.} {\cite[Proposition 4.3]{OptimalBErates}}]\label{Prp Gamma estimate}
Let $F \in \cW_q$ for some $q \geq 2$ and $\E[F^2] = 1$.
For any $k \in \N_0$ and $p \in [1, \infty)$, we have
\begin{align}
    \norm{\Gamma_2(F) - 2^{-1}\kappa_3(F)}_{\bW^{k, p}} 
    &\lesssim_{k, p,q} \kappa_4(F)^{\frac{3}{4}},\\
    \norm{\Gamma_3(F)}_{\bW^{k, p}} 
    &\lesssim_{k, p,q} \kappa_4(F), \\
    \norm{\Gamma_4(F)}_{\bW^{k, p}} 
    &\lesssim_{k, p, q} \kappa_4(F)^{\frac{5}{4}}.
\end{align}
\end{Prp}

\begin{Rem}
The case $k = 0$ and $p=1$ is shown in \cite[Proposition 4.3]{OptimalBErates}. 
The general case follows from the equivalence of $\bW^{\alpha, p}$-norms on Wiener chaos (see Lemma \ref{Lem Norm equivalence}). 
\end{Rem}

\subsection{Edgeworth-type expansions for density functions} \label{subsection Edgeworth-type expansions for density}

The following Edgeworth-type expansions, originating from Barbour's work \cite{BarbourEdgeworth}, was a main tool to establish optimal convergence rates in \cite{OptimalBErates}.
\begin{Prp}[{\cite[Proposition 3.11]{OptimalBErates}}] \label{Prp Edgeworth-type expansions for regular function}
Let $F \in \bW^{\infty}$ and $M \in \N$. 
For every function $f \colon \R \to \R$ that is $M$ times continuously differentiable with derivatives having at most polynomial growth, we have
\begin{equation}
    \E[Ff(F)] = \sum_{i=0}^{M-1}\frac{\kappa_{i+1}(F)}{i!}\E[f^{(i)}(F)] + \E[f^{(M)}(F)\Gamma_M(F)].
\end{equation}
\end{Prp}

We generalize this expansion to cover the case where $f$ is a tempered distribution.

\begin{Prp}\label{Prp Edgeworth-type expansion tempered case}
Let $M \in \N$, $g(x) \coloneqq x$, and $T \in \cS_{-\alpha}$ with some $\alpha \in \N_0$. 
Take $p,q \in (1, \infty)$ such that $s_{\alpha + M} = s_{\alpha + M}[p,q] > 1$, where $\{s_j\}_{j \in \N}$ is defined by \eqref{def s_j}.
Assume that $F \in \bW^{\infty}$ and $\Delta_F^{-1} \in L^q(\Omega)$. 
Then, we have
\begin{align}
    \dpair{\bW^{-\alpha -1, s_{\alpha+1}}}{(gT)(F)}{1}{\bW^{\alpha +1, s_{\alpha+1}'}}
    &= \sum_{i=0}^{M-1}\frac{\kappa_{i+1}(F)}{i!}\dpair{\bW^{-\alpha - i, s_{\alpha + i}}}{(\sD^{i}T)(F)}{1}{\bW^{\alpha + i, s_{\alpha + i}'}}\\
    &\quad + \dpair{\bW^{-\alpha-M, s_{\alpha + M}}}{(\sD^M T)(F)}{\Gamma_M(F)}{\bW^{\alpha + M, s_{\alpha +M}'}},
\end{align}
where $\dpair{\bW^{0, s_{0}}}{T(F)}{1}{\bW^{0, s_{0}'}}$ is understood as $\E[T(F)]$.
\end{Prp}

\begin{proof}
Taking a sequence $\{T_n\}_n \subset \cS(\R)$ such that $T_n \to T$ in $\cS_{-\alpha}$, we derive from Proposition \ref{Prp Edgeworth-type expansions for regular function} that for each $n$, 
\begin{equation}\label{appro expansion}
    \E[(gT_n)(F)] = \E[F T_n(F)] = \sum_{i=0}^{M-1}\frac{\kappa_{i+1}(F)}{i!}\E[T_n^{(i)}(F)] + \E[T_n^{(M)}(F)\Gamma_M(F)].
\end{equation}
By Proposition \ref{Prp tempered dis class}, we have $gT_n \to gT$ in $\cS_{-\alpha-1}$ and $T_n^{(i)} \to \sD^{i}T$ in $\cS_{-\alpha-i}$ for every $i \in \{1, \ldots, M\}$. 
It follows from Corollary \ref{Cor pullback map} that for every $i \in \{1, \ldots, M\}$, 
\begin{equation}
    (gT_n)(F) \xrightarrow{n\to \infty} gT(F) \quad \text{in $\bW^{-\alpha-1, s_{\alpha+1}}$} \qquad \text{and} \qquad T_n^{(i)}(F) \xrightarrow[]{n \to \infty} (\sD^{i}T)(F) \quad \text{in $\bW^{-\alpha-i, s_{\alpha + i}}$}.    
\end{equation}
When $\alpha \in \N$, we also have $T_n(F) \xrightarrow{n \to \infty} T(F)$ in $\bW^{-\alpha, s_{\alpha}}$ by Corollary \ref{Cor pullback map}. 
In the case $\alpha = 0$, the same convergence still holds in $L^2(\Omega)$ since $s_{1} \geq s_{\alpha + M} >1$ and Proposition \ref{Prp density derivative formula proof} ensures that $\rho_F \in C_{\mathrm{b}}(\R)$.
Therefore, letting $n \to \infty$ in \eqref{appro expansion} completes the proof. 
\end{proof}

Combining Proposition \ref{Prp Edgeworth-type expansion tempered case} with the results of Sections \ref{section Stein's equation in the space of tempered distributions}--\ref{section Composition of tempered distributions with Gaussian functionals}, we obtain Edgeworth-type expansions for the derivatives of density functions. 
This plays a key role in proving our main result.

\begin{Prp}\label{Prp Edgeworth expansions for density functions}
Let $k \in \N_0$ and $M \in \N$. 
Take $p,q \in (1,\infty)$ such that $s_{k+M} = s_{k+M}[p,q] > 1$, where $\{s_i\}_{i \in \N}$ is defined by \eqref{def s_j}. 
Assume that $F \in \bW^{\infty}$ and $\Delta_F^{-1} \in L^q(\Omega)$.
Then, for every $j \in \{0, \ldots, k\}$ and $a \in \R$, we have
\begin{align}
    \rho^{(j)}_{\cN}(a) - \rho^{(j)}_{F}(a)
    &= \sum_{i=0}^{M-1} \frac{\kappa_{i+1}(F)}{i!}\dpair{\bW^{-j-i, s_{j+i}}}{(\sD^{i}f_{a,j})(F)}{1}{\bW^{j+i, s_{j+i}'}}\\
    &\quad + \dpair{\bW^{-j-M, s_{j+M}}}{(\sD^{M}f_{a,j})(F)}{\Gamma_M(F)}{\bW^{j+M, s_{j+M}'}}\\
    &\quad - \dpair{\bW^{-j-1, s_{j+1}}}{(\sD f_{a,j})(F)}{1}{\bW^{j+1,s_{j+1}'}},
\end{align}
where $\dpair{\bW^{0, s_0}}{f_{a,0}(F)}{1}{\bW^{0,s_0'}}$ is understood as $\E[f_{a,0}(F)]$.
In particular, if in addition $\E[F] = 0$ and $\E[F^2] = 1$, then $\kappa_1(F) = 0$ and $\kappa_2(F) = 1$, and consequently  
\begin{align}
    \rho^{(j)}_{\cN}(a) - \rho^{(j)}_{F}(a) 
    &= \sum_{i=2}^{M-1} \frac{\kappa_{i+1}(F)}{i!}\dpair{\bW^{-j-i, s_{j+i}}}{(\sD^{i}f_{a,j})(F)}{1}{\bW^{j+i, s_{j+i}'}}\\
    &\quad + \dpair{\bW^{-j-M, s_{j+M}}}{(\sD^{M}f_{a,j})(F)}{\Gamma_M(F)}{\bW^{j+M, s_{j+M}'}}.
\end{align}
\end{Prp}

\begin{proof}
By Propositions \ref{Prp density derivative formula proof} and \ref{Prp const comp}, 
\begin{equation}
    \rho^{(j)}_{\cN}(a) - \rho^{(j)}_{F}(a) = \dpairLR{\bW^{-j-1,s_{j+1}}}{\left(\rho^{(j)}_{\cN}(a)\right)(F) - (-1)^j(\sD^j\delta_a)(F)}{1}{\bW^{j+1, s_{j+1}'}}.
\end{equation}
Since $f_{a,j}$ solves \eqref{derivative stein eq} with $n = j$, we obtain
\begin{align}
    \rho^{(j)}_{\cN}(a) - \rho^{(j)}_{F}(a) = \dpair{\bW^{-j-1,s_{j+1}}}{(gf_{a,j})(F)}{1}{\bW^{j+1, s_{j+1}'}} -  \dpair{\bW^{-j-1,s_{j+1}}}{(\sD f_{a,j})(F)}{1}{\bW^{j+1, s_{j+1}'}},
\end{align} 
where $g(x) = x$.
Applying Proposition \ref{Prp Edgeworth-type expansion tempered case} yields the desired result.
\end{proof}

\subsection{Proof of Theorem \ref{Thm Main results 1}} \label{subsection Proof of main results 1}
We now prove Theorem \ref{Thm Main results 1}. 
Recall our assumption that $\mathbf{F} = \{F_n\}_{n \in \N} \subset \cW_{m}$ with $m \geq 2$ such that $\E[F_n^2] = 1$ and $F_n \xrightarrow[n \to \infty]{d} \cN(0,1)$.
In this case, $\kappa_4(F_n) > 0$ and 
\begin{equation}
    \kappa_1(F_n) = 0, \qquad \kappa_2(F_n) = 1, \qquad \kappa_3(F_n) = \E[F_n^3] \to 0, \qquad \kappa_4(F_n) = \E[F_n^4] -3 \to 0.
\end{equation}
Moreover, we have
\begin{equation}\label{uniform bounds}
    \sup_{n\in \N} \norm{F_n}_{\bW^{\alpha, p}} < \infty \qquad \text{and} \qquad \limsup_{n \to \infty} \left\lVert \Delta_{F_n}^{-1}\right\rVert_{L^q(\Omega)} < \infty
\end{equation}
for any $\alpha \in \R$ and $p, q \in [1, \infty)$, by Lemmas \ref{Lem Uni bound sobolev} and \ref{Lem asymtotic nondegeneracy}.

\begin{proof}[Proof of Theorem \ref{Thm Main results 1}]
Let $j \in \N_0$. 
Take $p,q \in (1, \infty)$ such that $s_{j+4} = s_{j+4}[p,q] > 1$.
It follows that $s_1 > \cdots > s_{j+4} > 1$.  
For these $p$ and $q$, by \eqref{uniform bounds}, one can choose $\wt{N_{\mathbf{F}, j}} \in \N$ such that for any $\alpha \in \R$, 
\begin{equation}
    \sup_{n \geq \wt{N_{\mathbf{F}, j}}} \norm{F_n}_{\bW^{\alpha, p}} < \infty \qquad \text{and} \qquad \sup_{n \geq \wt{N_{\mathbf{F}, j}}}\norm{\Delta_{F_n}^{-1}}_{L^{q}(\Omega)} < \infty, \label{uni b 10}
\end{equation}
which in turn implies $\rho_{F_n} \in C_{\mathrm{b}}^{j+3}(\R)$ for every $n \geq \wt{N_{\mathbf{F}, j}}$.
By Proposition \ref{Prp Pullback est}, for any  $T \in \cS_{-i}$ with $i \in \{1, \ldots, j+4\}$, we obtain
\begin{equation}
    \norm{T(F_n)}_{\bW^{-i, s_i}} \lesssim_{i, p, q, s_i} C_{i}\left(\norm{F_n}_{\bW^{i+1, p}}, \norm{\Delta_{F_n}^{-1}}_{L^{q}(\Omega)}\right) \norm{T}_{\cS_{-i}},
\end{equation}
where $C_{i}\left(\norm{F_n}_{\bW^{i+1, p}}, \norm{\Delta_{F_n}^{-1}}_{L^{q}(\Omega)}\right)$ is a constant depending on $i$ and some positive powers of $\norm{F_n}_{\bW^{i+1, p}}$ and $\norm{\Delta_{F_n}^{-1}}_{L^{q}(\Omega)}$.

We first prove the upper bound.
Let $n \geq \wt{N_{\mathbf{F}, j}}$. 
By applying Proposition \ref{Prp Edgeworth expansions for density functions} with $M=3$, we obtain 
\begin{align}
    &|\rho_{F_n}^{(j)}(a) - \rho_{\cN}^{(j)}(a)| \\
    &\leq \frac{|\kappa_3(F_n)|}{2} \norm{(\sD^2 f_{a,j})(F_n)}_{\bW^{-j-2, s_{j+2}}} + \norm{\sD^3 f_{a,j}(F_n)}_{\bW^{-j-3, s_{j+3}}} \norm{\Gamma_3(F_n)}_{\bW^{j+3, s_{j+3}'}}\\
    &\lesssim_{j, m} \frac{|\kappa_3(F_n)|}{2} \norm{(\sD^2 f_{a,j})(F_n)}_{\bW^{-j-2, s_{j+2}}} + \norm{\sD^3f_{a,j}(F_n)}_{\bW^{-j-3, s_{j+3}}} \kappa_4(F_n), \label{LCLT UB6}
\end{align}
where the last step follows from Proposition \ref{Prp Gamma estimate}. 
By Propositions \ref{Prp Pullback est} and \ref{Prp f_a,n class Dfan uniform bound}, we have
\begin{equation}
    \sup_{a \in \R}\norm{(\sD^l f_{a,j})(F_n)}_{\bW^{-j-l, s_{j+l}}} 
    \leq \wt{C_{j,l}} \sup_{a \in \R}\norm{\sD^{l} f_{a,j}}_{\cS_{-j-l}} < \infty, \qquad l = 2,3,  \label{LCLT t1 finite60}
\end{equation}
where $\wt{C_{j, l}}, \ l = 2,3$, are constants depending on $j$ and some positive powers of \eqref{uni b 10} with $\alpha = j+l+1$.
Combining \eqref{LCLT UB6} and \eqref{LCLT t1 finite60}, we obtain
\begin{align}
    \sup_{a \in \R} \left|\rho_{F_n}^{(j)}(a) - \rho_{\cN}^{(j)}(a)\right| 
    &\lesssim_{j,m} \frac{\wt{C_{j,2}}}{2}|\kappa_3(F_n)| + \wt{C_{j,3}}\kappa_4(F_n) \leq \wt{C_{j}} \mathbf{M}(F_n).
\end{align}

Next, we prove the lower bound. 
Once again, let $n \geq \wt{N_{\mathbf{F}, j}}$. 
By applying Proposition \ref{Prp Edgeworth expansions for density functions} with $M=4$, we obtain
\begin{align}
    &\rho_{\cN}^{(j)}(a) - \rho_{F_n}^{(j)}(a) \\
    &= \frac{\kappa_{3}(F_n)}{2}\dpairLR{\bW^{-j-2, s_{j+2}}}{(\sD^{2}f_{a,j})(F_n)}{1}{\bW^{j+2, s_{j+2}'}} + \frac{\kappa_{4}(F_n)}{6} \dpairLR{\bW^{-j-3, s_{j+3}}}{(\sD^{3}f_{a,j})(F_n)}{1}{\bW^{j+3, s_{j+3}'}}\\
    &\quad + \dpairLR{\bW^{-j-4, s_{j+4}}}{(\sD^4 f_{a,j})(F_n)}{\Gamma_4(F_n)}{\bW^{j+4, s_{j+4}'}}.
\end{align}
Now, take $h \in \fH$ with $\norm{h}_{\fH} = 1$, and let $X_h \coloneqq X(h)$. 
Since $X_h \in \bW^{\infty}$ and $\Delta_{X_h}^{-1} = 1 \in \bigcap_{q \in [1,\infty]}L^q(\Omega)$, we obtain
\begin{align}
    &\rho_{\cN}^{(j)}(a) - \rho_{F_n}^{(j)}(a) \\
    &- \frac{\kappa_{3}(F_n)}{2}\dpairLR{\bW^{-j-2, s_{j+2}}}{(\sD^{2}f_{a,j})(X_h)}{1}{\bW^{j+2, s_{j+2}'}} - \frac{\kappa_{4}(F_n)}{6} \dpairLR{\bW^{-j-3, s_{j+3}}}{(\sD^{3}f_{a,j})(X_h)}{1}{\bW^{j+3, s_{j+3}'}}\\
    &= \frac{\kappa_{3}(F_n)}{2}\dpairLR{\bW^{-j-2, s_{j+2}}}{(\sD^{2}f_{a,j})(F_n) - (\sD^{2}f_{a,j})(X_h)}{1}{\bW^{j+2, s_{j+2}'}}\\
    &\quad + \frac{\kappa_{4}(F_n)}{6} \dpairLR{\bW^{-j-3, s_{j+3}}}{(\sD^{3}f_{a,j})(F_n) - (\sD^{3}f_{a,j})(X_h)}{1}{\bW^{j+3, s_{j+3}'}} \\
    &\quad + \dpairLR{\bW^{-j-4, s_{j+4}}}{(\sD^4 f_{a,j})(F_n)}{\Gamma_4(F_n)}{\bW^{j+4, s_{j+4}'}}. \label{density diff expansions}
\end{align}
Let $\{\psi_{a,j,l}\}_{l} \subset \cS(\R)$ be a sequence such that $\psi_{a,j,l} \xrightarrow[]{l \to \infty} f_{a,j}$ in $\cS_{-j}$.
It follows that $\psi_{a,j,l} \xrightarrow[]{l \to \infty} f_{a,j}$ in the weak topology of $\cS'(\R)$ and that for each $i = 2, 3$,
\begin{align}
    \dpairLR{\bW^{-j-i, s_{j+i}}}{(\sD^{i}f_{a,j})(X_h)}{1}{\bW^{j+i, s_{j+i}'}} 
    &= \lim_{l \to \infty} \dpairLR{\bW^{-j-i, s_{j+i}}}{\psi_{a,j,l}^{(i)}(X_h)}{1}{\bW^{j+i, s_{j+i}'}}\\
    &= \lim_{l \to \infty} \E[\psi_{a,j,l}^{(i)}(X_h)]\\
    &= \lim_{l \to \infty} \dpairLR{\cS'(\R)}{\psi_{a,j,l}^{(i)}}{\rho_{\cN}}{\cS(\R)}.
\end{align}
We then apply Lemma \ref{Lem f_a calc} to obtain
\begin{align}
    \dpairLR{\bW^{-j-i, s_{j+i}}}{(\sD^{i}f_{a,j})(X_h)}{1}{\bW^{j+i, s_{j+i}'}} 
    &= (-1)^i\lim_{l \to \infty} \dpairLR{\cS'(\R)}{\psi_{a,j,l}}{\rho_{\cN}^{(i)}}{\cS(\R)}\\
    &= (-1)^i\dpairLR{\cS'(\R)}{f_{a,j}}{\rho_{\cN}^{(i)}}{\cS(\R)}\\
    &= \frac{(-1)^i}{i+1}\rho_{\cN}^{(j+i+1)}(a).
\end{align}
This, together with \eqref{density diff expansions} and Proposition \ref{Prp Gamma estimate}, implies that
\begin{align}
    &\left|\rho_{\cN}^{(j)}(a) - \rho_{F_n}^{(j)}(a) - \frac{\kappa_3(F_n)}{6}\rho_{\cN}^{(j+3)}(a) + \frac{\kappa_4(F_n)}{24}\rho_{\cN}^{(j+4)}(a)\right|\\
    &\lesssim_{j,m} \frac{|\kappa_{3}(F_n)|}{2} \left| \dpairLR{\bW^{-j-2, s_{j+2}}}{(\sD^{2}f_{a,j})(F_n) - (\sD^{2}f_{a,j})(X_h)}{1}{\bW^{j+2, s_{j+2}'}}\right| \\
    &\quad + \frac{\kappa_{4}(F_n)}{6} \left| \dpairLR{\bW^{-j-3, s_{j+3}}}{(\sD^{3}f_{a,j})(F_n) - (\sD^{3}f_{a,j})(X_h)}{1}{\bW^{j+3, s_{j+3}'}} \right| \\
    &\quad + \kappa_4(F_n)^{\frac{5}{4}} \norm{(\sD^4 f_{a,j})(F_n)}_{\bW^{-j-4, s_{j+4}}} \\
    &\leq \frac{|\kappa_{3}(F_n)|}{2} \left| \dpairLR{\bW^{-j-2, s_{j+2}}}{(\sD^{2}f_{a,j})(F_n) - (\sD^{2}f_{a,j})(X_h)}{1}{\bW^{j+2, s_{j+2}'}}\right| \\
    &\quad + \frac{\kappa_{4}(F_n)}{6} \left| \dpairLR{\bW^{-j-3, s_{j+3}}}{(\sD^{3}f_{a,j})(F_n) - (\sD^{3}f_{a,j})(X_h)}{1}{\bW^{j+3, s_{j+3}'}} \right| \\
    &\quad + C_{j,4} \kappa_4(F_n)^{\frac{5}{4}},  \label{LCLT LB bound6}
\end{align}
where $C_{j,4}$ is a constant depending on $j$, $\sup_{a \in \R}\norm{\sD^4 f_{a,j}}_{\cS_{-j-4}}$, and some positive powers of \eqref{uni b 10} with $\alpha = j+5$.
Since $F_n \xrightarrow[n \to \infty]{d} X_h \stackrel{d}{=} \cN(0,1)$, we deduce that for $i = 2,3$, 
\begin{align}
    &\lim_{n \to \infty} \dpairLR{\bW^{-j-i, s_{j+i}}}{(\sD^{i}f_{a,j})(F_n) - (\sD^{i}f_{a,j})(X_h)}{1}{\bW^{j+i, s_{j+i}'}}\\
    &= \lim_{n \to \infty} \lim_{l \to \infty} \E[\psi_{a,j,l}^{(i)}(F_n) - \psi_{a,j,l}^{(i)}(X_h)]\\
    &= \lim_{l \to \infty} \lim_{n \to \infty} \E[\psi_{a,j,l}^{(i)}(F_n) - \psi_{a,j,l}^{(i)}(X_h)]\\
    &= 0. \label{LCLT LB lim 06}
\end{align}
Indeed, we have uniform convergence 
\begin{align}
    &\sup_{n\geq N_{j,\mathbf{F}}}\left|\E[\psi_{a,j, l}^{(i)}(F_n) - \psi_{a,j, l}^{(i)}(X_h)] - \dpairLR{\bW^{-j-i, s_{j+i}}}{(\sD^i f_{a,j})(F_n) - (\sD^i f_{a,j})(X_h)}{1}{\bW^{j+i, s_{j+i}'}}\right| \\
    &\leq \sup_{n\geq N_{j,\mathbf{F}}} \norm{\psi_{a,j,l}^{(i)}(F_n) - (\sD^i f_{a,j})(F_n)}_{\bW^{-j-i, s_{j+i}}} + \norm{\psi_{a,j,l}^{(i)}(X_h) - (\sD^i f_{a,j})(X_h)}_{\bW^{-j-i, s_{j+i}}}\\
    &\leq C_{j,i,h} \norm{\psi_{a,j,l}^{(i)} - \sD^i f_{a,j}}_{\cS_{-j-i}} \\
    &\xrightarrow[]{l \to \infty} 0,
\end{align}
so the two limits can be interchanged. 
Here, $C_{j,i,h}$ is a constant depending on $i,j$, $\norm{X_h}_{\bW^{j+i+1, p}}$, and some positive powers of \eqref{uni b 10} with $\alpha = j+i+1$.
Combining \eqref{LCLT LB bound6}, \eqref{LCLT LB lim 06}, and $\lim_{n\to \infty}\kappa_4(F_n) = 0$, we see that there exists $N_{\mathbf{F}, j} \in \N$ larger than $\wt{N_{\mathbf{F}, j}}$ such that for all $n \geq N_{\mathbf{F}, j}$, 
\begin{align}
    &\left|\rho_{\cN}^{(j)}(a) - \rho_{F_n}^{(j)}(a) - \frac{\kappa_3(F_n)}{6}\rho_{\cN}^{(j+3)}(a) + \frac{\kappa_4(F_n)}{24}\rho_{\cN}^{(j+4)}(a)\right|\\
    &\leq \frac{1}{6}\min \left\{ \frac{1}{6}\left|\rho_{\cN}^{(j+3)}(\zeta_{j+4})\right|, \frac{1}{24}\left|\rho_{\cN}^{(j+4)}(\zeta_{j+3})\right| \right\} \mathbf{M}(F_n), 
\end{align}
where $\zeta_{j+3}$ and $\zeta_{j+4}$ are arbitrarily chosen zeros of $\rho_{\cN}^{(j+3)}$ and $\rho_{\cN}^{(j+4)}$, respectively. 
Set 
\begin{equation}
    A_j = \frac{1}{3}\min \left\{ \frac{1}{6}\left|\rho_{\cN}^{(j+3)}(\zeta_{j+4})\right|, \frac{1}{24}\left|\rho_{\cN}^{(j+4)}(\zeta_{j+3})\right| \right\}.
\end{equation}
It follows that for all $n \geq N_{\mathbf{F}, j}$, 
\begin{equation}
    \sup_{a \in \R}|\rho_{\cN}^{(j)}(a) - \rho_{F_n}^{(j)}(a)| 
    \geq |\rho_{\cN}^{(j)}(\zeta_{j+4}) - \rho_{F_n}^{(j)}(\zeta_{j+4})|
    \geq 3 A_j |\kappa_3(F_n)| - \frac{A_j}{2}\mathbf{M}(F_n)
\end{equation}
and 
\begin{equation}
    \sup_{a \in \R}|\rho_{\cN}^{(j)}(a) - \rho_{F_n}^{(j)}(a)| 
    \geq |\rho_{\cN}^{(j)}(\zeta_{j+3}) - \rho_{F_n}^{(j)}(\zeta_{j+3})|
    \geq 3 A_j \kappa_4(F_n) - \frac{A_j}{2}\mathbf{M}(F_n).
\end{equation}
Combining these two estimates, we obtain 
\begin{align}
    \sup_{a \in \R}|\rho_N^{(j)}(a) - \rho_{F_n}^{(j)}(a)| 
    &\geq \frac{3 A_j}{2}(|\kappa_{3}(F_n)| + \kappa_4(F_n)) - \frac{A_j}{2}\mathbf{M}(F_n)\\
    &\geq A_j \mathbf{M}(F_n),
\end{align}
which is the desired lower bound. This completes the proof.
\end{proof}

\subsection{Proof of Theorem \ref{Thm Main results 2}}\label{subsection Proof of main results 2}
We split the proof of Theorem \ref{Thm Main results 2} into two parts. 
The lower bound is proved in Proposition \ref{Prp Sobolev norm lower bound}, and the upper bound in Proposition \ref{Prp Sobolev norm upper bound}.

\begin{Prp}\label{Prp Sobolev norm lower bound}
Let $\mathbf{F} = \{F_n\}_{n \in \N} \subset \cW_m$ with $m \geq 2$ be such that $\E[F_n^2] = 1$ and $F_n \xrightarrow[n \to \infty]{d} \cN(0,1)$.
Then, for any $j \in \N_0$, there exists $N_{\mathbf{F},j} \in \N$ such that $\rho_{F_n} \in C_{\mathrm{b}}^{j}(\R)$ and
\begin{equation}
    \left\lVert \rho_{F_n}^{(j)} - \rho_{\cN}^{(j)}\right\rVert_{L^r(\R)} \geq C_{\mathbf{F},j,m, r} \mathbf{M}(F_n)
\end{equation}
hold for every $n \geq N_{\mathbf{F},j}$ and $r \in [1,\infty)$, where the constant $C_{\mathbf{F},j, m, r}$ depends on $\mathbf{F}$, $j$, $m$, and $r$ but not on $n$.
\end{Prp}

\begin{proof}
Let $j \in \N_0$. 
Then we can choose $N_{\mathbf{F},j}$ such that the statement of Theorem \ref{Thm Main results 1} for $j$ and $j+1$ hold for every $n \geq N_{\mathbf{F},j}$.
Set $M_{j,n} = \sup_{a \in \R}|\rho_{F_n}^{(j)}(a) - \rho_{\cN}^{(j)}(a)|$, and take $\xi \in \R$ such that
\begin{equation}
    \left|\rho_{F_n}^{(j)}(\xi) - \rho_{\cN}^{(j)}(\xi)\right| \geq \frac{M_{j,n}}{2}.
\end{equation}
Since $M_{j+1,n} < \infty$, the Lipschitz continuity of $x \mapsto |\rho_{F_n}^{(j)}(x) - \rho_{\cN}^{(j)}(x)|$ yields
\begin{align}
    \left|\rho_{F_n}^{(j)}(a) - \rho_{\cN}^{(j)}(a)\right| 
    &\geq \left|\rho_{F_n}^{(j)}(\xi) - \rho_{\cN}^{(j)}(\xi)\right| - M_{j+1,n}|a-\xi|\\
    &\geq \frac{M_{j,n}}{2} - M_{j+1,n}|a-\xi|.
\end{align}
If we set $I_{j,n} = \{a \in \R \  : \  |a- \xi| \leq \frac{M_{j,n}}{2M_{j+1,n}}\}$, then $\frac{M_{j,n}}{M_{j+1,n}} > 0$ for all sufficiently large $n$ by Theorem \ref{Thm Main results 1}, and we obtain
\begin{align}
    \int_{\R}\left|\rho_{F_n}^{(j)}(a) - \rho_{\cN}^{(j)}(a)\right|^r da 
    &\geq \int_{I_{j,n}}\left( \frac{M_{j,n}}{2} - M_{j+1,n}|a-\xi| \right)^r da \\
    &= \frac{2}{M_{j+1,n}} \int_{0}^{\frac{M_j,n}{2}}x^r dx \\
    &= \frac{1}{(r+1)2^r} M_{j,n}^{r+1} M_{j+1, n}^{-1}.
\end{align}
From this and Theorem \ref{Thm Main results 1}, we deduce the desired lower bound, and the proof is complete.
\end{proof}

To derive the upper bound, we first establish the following lemma.
\begin{Lem}\label{Lem varphi^j stein sol}
For $j \in \N_0$ and $\varphi \in C_{\mathrm{c}}^{\infty}(\R)$, we define the smooth function 
\begin{equation}
    f_{\varphi^{(j)}}(x) \coloneqq e^{\frac{x^2}{2}}\int_{-\infty}^{x} \left( \varphi^{(j)}(y) - \E\left[\varphi^{(j)}(\cN)\right] \right)e^{-\frac{y^2}{2}}dy.
\end{equation}
Then, $f_{\varphi^{(j)}}$ satisfies 
\begin{equation}
    f_{\varphi^{(j)}}'(x) - xf_{\varphi^{(j)}}(x) = \varphi^{(j)}(x) - \E\left[\varphi^{(j)}(\cN)\right] \label{varphi diff eq}
\end{equation}
and 
\begin{equation}
    \sup_{\substack{\varphi \in C_\mathrm{c}^{\infty}(\R) \\ \norm{\varphi}_{L^{\infty}(\R)} \leq 1}} \norm{f_{\varphi^{(j)}}}_{\cS_{-j-2}} < \infty. \label{f varphi uniform bound}
\end{equation}
\end{Lem}

\begin{proof}
A direct calculation yields \eqref{varphi diff eq}. 
Also, it is easy to check that $f_{\varphi^{(j)}} \in \cS'(\R)$ since $\varphi \in C_{\mathrm{c}}^{\infty}(\R)$.  
To prove \eqref{f varphi uniform bound}, we only need to show that 
\begin{equation}
    \sup_{\substack{\varphi \in C_\mathrm{c}^{\infty}(\R) \\ \norm{\varphi}_{L^{\infty}(\R)} \leq 1}} \left|\int_{\R} f_{\varphi^{(j)}}(x) \phi_n(x)dx \right| \lesssim_j n^{\frac{j}{2}} \label{suff cond for Hermite coeff}
\end{equation}
holds for all sufficiently large $n$, where  $\{\phi_n\}_{n \in \N_0}$ are Hermite functions.
By repeatedly applying integration by parts and (i) of Lemma \ref{Lem Hermite polynomials properties}, we obtain
\begin{align}
    f_{\varphi^{(j)}}(x) 
    &= e^{\frac{x^2}{2}} \int_{-\infty}^{x}\varphi^{(j)}(y)e^{-\frac{y^2}{2}}dy - e^{\frac{x^2}{2}}\int_{-\infty}^{x}e^{-\frac{y^2}{2}}dy \int_{\R} \varphi^{(j)}(z)\rho_{\cN}^{}(z)dz\\
    &= \sum_{i=0}^{j}H_{i-1}(x)\varphi^{(j-i)}(x) + e^{\frac{x^2}{2}}\int_{-\infty}^{x} \left(\varphi(y)H_j(y) - \E[\varphi(\cN)H_j(\cN)] \right)e^{-\frac{y^2}{2}}dy\\
    &\eqqcolon \mathbf{A}(x) + \mathbf{B}(x), \label{A + B}
\end{align}
where we set $H_{-1} \equiv 0$.
We first evaluate $\int_{\R}\mathbf{A}(x)\phi_n(x)dx$. 
Applying integration by parts again, we see that
\begin{align}
    \int_{\R}\mathbf{A}(x)\phi_n(x)dx 
    &= \sum_{i=1}^{j} (-1)^{j-i}\int_{\R}\varphi(x)\frac{d^{j-i}}{dx^{j-i}}\left(H_{i-1}(x)\phi_n(x)\right)dx \\
    &= \sum_{i=1}^{j}(-1)^{j-i}\sum_{l=0}^{(j-i)\land (i-1)}\binom{j-i}{l}\int_{\R}\varphi(x)H_{i-1}^{(l)}(x)\phi_n^{(j-i-l)}(x)dx 
\end{align}
and consequently, 
\begin{equation}
    \sup_{\substack{\varphi \in C_\mathrm{c}^{\infty}(\R) \\ \norm{\varphi}_{L^{\infty}(\R)} \leq 1}} \left|\int_{\R} \mathbf{A}(x) \phi_n(x)dx \right| \lesssim_j \sum_{i=1}^{j}\sum_{l=0}^{(j-i)\land (i-1)} \int_{\R}\left|H_{i-1}^{(l)}(x)\right|\left|\phi_n^{(j-i-l)}(x)\right|dx.
\end{equation}
Since $H_{i-1}^{(l)}$ is a polynomial of degree $i-1-l$, we obtain
\begin{align}
    \int_{\R} \left|H_{i-1}^{(l)}(x)\right| \left|\phi_n^{(j-i-l)}(x)\right|dx
    &\leq \left\lVert \frac{H_{i-1}^{(l)}}{1+|\cdot|^{i-l}} \right\rVert_{L^2(\R)} \norm{(1+|\cdot|^{i-l})\phi_n^{(j-i-l)}}_{L^2(\R)}\\
    &\lesssim_{i,l} \norm{\phi_n^{(j-i-l)}}_{L^2(\R)} + \norm{|\cdot|^{i-l}\phi_n^{(j-i-l)}}_{L^2(\R)}\\
    &\lesssim_{i,j,l} n^{\frac{j-i-l}{2}} + n^{\frac{j-2l}{2}}
\end{align}
for all sufficiently large $n$, where the last line is a consequence of (ii) of Lemma \ref{Lem 1d Hermite functions property} together with the fact that $\norm{\phi_n}_{L^2(\R)} = 1$ for all $n \in \N_0$. 
It follows that for all sufficiently large $n$, 
\begin{equation}
    \sup_{\substack{\varphi \in C_\mathrm{c}^{\infty}(\R) \\ \norm{\varphi}_{L^{\infty}(\R)} \leq 1}} \left|\int_{\R} \mathbf{A}(x) \phi_n(x)dx \right| \lesssim_{j} n^{\frac{j}{2}}. \label{A uni bound}
\end{equation}
As for $\mathbf{B}(x)$, observe that
\begin{equation}
    \int_{-\infty}^{x}\left(\varphi(y)H_j(y) - \E[\varphi(\cN)H_j(\cN)] \right)e^{-\frac{y^2}{2}}dy = - \int_{x}^{\infty} \left(\varphi(y)H_j(y) - \E[\varphi(\cN)H_j(\cN)] \right)e^{-\frac{y^2}{2}}dy.
\end{equation}
From this, we deduce that 
\begin{align}
    &\sup_{\substack{\varphi \in C_\mathrm{c}^{\infty}(\R) \\ \norm{\varphi}_{L^{\infty}(\R)} \leq 1}}|\mathbf{B}(x)| \\
    &\leq e^{\frac{x^2}{2}} \min \left\{ \int_{-\infty}^{x} \left( |H_j(y)| + \E[|H_j(\cN)|] \right)e^{-\frac{y^2}{2}}dy, \  \int_{x}^{\infty} \left( |H_j(y)| + \E[|H_j(\cN)|] \right)e^{-\frac{y^2}{2}}dy \right\}\\
    &\lesssim_j e^{\frac{x^2}{2}}\int_{|x|}^{\infty} (1+y^j)e^{-\frac{y^2}{2}}dy.
\end{align}
A simple computation shows that
\begin{equation}
    e^{\frac{x^2}{2}}\int_{|x|}^{\infty} e^{-\frac{y^2}{2}}dy \leq \sqrt{\frac{e \pi}{2}} \ind_{\{|x| \leq 1\}} + |x|^{-1}\ind_{\{|x|>1\}} \leq \sqrt{\frac{e \pi}{2}}
\end{equation}
and 
\begin{equation}
    \int_{|x|}^{\infty} y^j e^{-\frac{y^2}{2}}dy = |x|^{j-1}e^{-\frac{|x|^2}{2}} + (j-1)\int_{|x|}^{\infty}y^{j-2}e^{-\frac{y^2}{2}}dy
\end{equation}
for $j \in \N$.
It follows that $\sup \{|\mathbf{B}(x)| \ : \  \varphi \in C_\mathrm{c}^{\infty}(\R), \ \norm{\varphi}_{L^{\infty}(\R)} \leq 1 \} \lesssim_j 1+|x|^{j-1}$, and we obtain
\begin{align}
    \sup_{\substack{\varphi \in C_\mathrm{c}^{\infty}(\R) \\ \norm{\varphi}_{L^{\infty}(\R)} \leq 1}}\left|\int_{\R}\mathbf{B}(x)\phi_n(x)dx\right| 
    &\lesssim_j \int_{\R}(1+|x|^{j-1})|\phi_n(x)|dx\\
    &\leq \left\lVert  \frac{1+|\cdot|^{j-1}}{1+|\cdot|^{j}} \right\rVert_{L^2(\R)} \norm{(1+|\cdot|^{j})\phi_n}_{L^2(\R)}\\
    &\lesssim_j n^{\frac{j}{2}} \label{B uni bound}
\end{align}
for all $n$ large enough, where the last step again follows from (ii) of Lemma \ref{Lem 1d Hermite functions property} and the fact that $\norm{\phi_n}_{L^2(\R)} = 1$ for all $n \in \N_0$. 
Therefore, \eqref{suff cond for Hermite coeff} holds, and the proof is complete. 
\end{proof}

We now prove the upper bound and thus complete the proof of Theorem \ref{Thm Main results 2}.
\begin{Prp}\label{Prp Sobolev norm upper bound}
Let $\mathbf{F} = \{F_n\}_{n \in \N} \subset \cW_m$ with $m \geq 2$ be such that $\E[F_n^2] = 1$ and $F_n \xrightarrow[n \to \infty]{d} \cN(0,1)$.
Then, for any $j \in \N_0$, there exists $N_{\mathbf{F},j} \in \N$ such that $\rho_{F_n} \in C_{\mathrm{b}}^{j}(\R)$ and
\begin{equation}
    \left\lVert \rho_{F_n}^{(j)} - \rho_{\cN}^{(j)}\right\rVert_{L^r(\R)} \leq \wt{C}_{\mathbf{F},j,m,r} \mathbf{M}(F_n)
\end{equation}
hold for every $n \geq N_{\mathbf{F},j}$ and $r \in [1,\infty)$, where the constant $\wt{C}_{\mathbf{F},j, m, r}$ depends on $\mathbf{F}$, $j$, $m$, and $r$ but not on $n$.
\end{Prp}

\begin{proof}
Let $j \in \N_0$. Take $p,q \in (1,\infty)$ such that $s_{j+5} = s_{j+5}[p,q] > 1$. 
Then there exists $N_{\mathbf{F},j} \in \N$ such that Theorem \ref{Thm Main results 1} holds for $j$ and, for any $\alpha \in \R$, 
\begin{equation}
    \sup_{n \geq N_{\mathbf{F},j}}\norm{F_n}_{\bW^{\alpha, p}} < \infty \qquad \text{and} \qquad \sup_{n \geq N_{\mathbf{F},j}} \norm{\Delta_{F_n}^{-1}}_{L^{q}(\Omega)} < \infty. \label{uni b 8}
\end{equation} 
To complete the proof, it suffices to show the case $r=1$ since 
\begin{equation}
    \left\lVert \rho_{F_n}^{(j)} - \rho_{\cN}^{(j)}\right\rVert_{L^r(\R)} 
    \lesssim_{\mathbf{F}, j, m} \mathbf{M}(F_n)^{\frac{r-1}{r}}  \left\lVert \rho_{F_n}^{(j)} - \rho_{\cN}^{(j)}\right\rVert_{L^1(\R)}^{\frac{1}{r}} 
\end{equation}
by Theorem \ref{Thm Main results 1}. 
In view of the fact that $C_{\mathrm{c}}^{\infty}(\R)$ is weak-star dense in $L^{\infty}(\R)$, this $L^1(\R)$-norm is given by 
\begin{align}
    \left\lVert \rho_{F_n}^{(j)} - \rho_{\cN}^{(j)}\right\rVert_{L^1(\R)}
    &= \sup_{\substack{\varphi \in L^{\infty}(\R) \\ \norm{\varphi}_{L^{\infty}(\R)} \leq 1}} \left| \int_{\R} \left(\rho_{F_n}^{(j)}(x) - \rho_{\cN}^{(j)}(x)\right) \varphi(x) dx \right|\\
    &= \sup_{\substack{\varphi \in C_\mathrm{c}^{\infty}(\R) \\ \norm{\varphi}_{L^{\infty}(\R)} \leq 1}} \left| \int_{\R} \left(\rho_{F_n}^{(j)}(x) - \rho_{\cN}^{(j)}(x)\right) \varphi(x) dx \right|\\
    &= \sup_{\substack{\varphi \in C_\mathrm{c}^{\infty}(\R) \\ \norm{\varphi}_{L^{\infty}(\R)} \leq 1}} \left|\E\left[\varphi^{(j)}(F_n) - \varphi^{(j)}(\cN)\right]\right|. \label{bb0}
\end{align}
By Lemma \ref{Lem varphi^j stein sol}, we obtain
\begin{align}
    \left|\E\left[\varphi^{(j)}(F_n) - \varphi^{(j)}(\cN)\right]\right|
    &= \left|\E\left[f_{\varphi^{(j)}}'(F_n) - (gf_{\varphi^{(j)}})(F_n)\right]\right|\\
    &= \left| \dpairLR{\bW^{-j-3, s_{j+3}}}{f_{\varphi^{(j)}}'(F_n) - (gf_{\varphi^{(j)}})(F_n)}{1}{\bW^{j+3, s_{j+3}'}}  \right|,
\end{align}
where $g(x) = x$. 
Noting that $\kappa_{1}(F_n) = 0$ and $\kappa_{2}(F_n)=1$, and applying Proposition \ref{Prp Edgeworth-type expansion tempered case} with $M=3$, we obtain
\begin{align}
    &\left|\E\left[\varphi^{(j)}(F_n) - \varphi^{(j)}(\cN)\right]\right| \\
    &\leq \frac{|\kappa_3(F_n)|}{2}\left|\dpairLR{\bW^{-j-4, s_{j+4}}}{f_{\varphi^{(j)}}^{(2)}(F_n)}{1}{\bW^{j+4, s_{j+4}'}}\right| + \left|\dpairLR{\bW^{-j-5, s_{j+5}}}{f_{\varphi^{(j)}}^{(3)}(F_n)}{\Gamma_3(F_n)}{\bW^{j+5, s_{j+5}'}}\right| \\
    &\lesssim_{j,m} \frac{|\kappa_3(F_n)|}{2} \norm{f_{\varphi^{(j)}}^{(2)}(F_n)}_{\bW^{-j-4, s_{j+4}}} + \norm{f_{\varphi^{(j)}}^{(3)}(F_n)}_{\bW^{-j-5, s_{j+5}}} \kappa_4(F_n), \label{bb1}
\end{align}
where we used Proposition \ref{Prp Gamma estimate} in the last line.
By Propositions \ref{Prp Pullback est} and \ref{Prp tempered dis class} and Lemma \ref{Lem varphi^j stein sol}, we have 
\begin{equation}
    \sup_{\substack{\varphi \in C_\mathrm{c}^{\infty}(\R) \\ \norm{\varphi}_{L^{\infty}(\R)} \leq 1}}\norm{f_{\varphi^{(j)}}^{(l)}(F_n)}_{\bW^{-j-2-l, s_{j+2+l}}} 
    \leq \wt{C_{\mathbf{F},j,l}} \sup_{\substack{\varphi \in C_\mathrm{c}^{\infty}(\R) \\ \norm{\varphi}_{L^{\infty}(\R)} \leq 1}}\norm{f_{\varphi^{(j)}}}_{\cS_{-j-2}} < \infty, \qquad l=2,3,  \label{bb2}
\end{equation}
where $\wt{C_{\mathbf{F},j,l}}, \ l=2,3,$ are constants depending on $j$ and some positive powers of \eqref{uni b 8} with $\alpha = j+3+l$.
Combining \eqref{bb0}, \eqref{bb1}, and \eqref{bb2} shows that
\begin{align}
    \left\lVert \rho_{F_n}^{(j)} - \rho_{\cN}^{(j)}\right\rVert_{L^1(\R)} \lesssim_{\mathbf{F}, j, m} |\kappa_3(F_n)| + \kappa_4(F_n) \lesssim \mathbf{M}(F_n),
\end{align}
and therefore, the proposition follows.
\end{proof}

\subsection{Optimal local Breuer--Major CLT}\label{subsection Optimal local Breuer--Major CLT}
Let $\{X_l\}_{l \in \Z}$ be a centered stationary Gaussian sequence with unit variance. 
Note that this sequence can be regarded as an isonormal Gaussian process over an appropriate separable Hilbert space.
For an integer $m \geq 2$ and a constant $\nu \neq 0$, consider the nonlinear functional of the sequence 
\begin{equation}
    V_n \coloneqq \frac{1}{\sqrt{n}}\sum_{l=1}^n \nu H_{m}(X_l). 
\end{equation}
The Breuer--Major central limit theorem, first established in full generality in \cite{BreuerMajor}, implies that if $\sum_{l \in \Z} |\E[X_0X_l]|^m < \infty$, then 
\begin{equation}
    \sigma^2 \coloneqq m!\nu^2\sum_{l \in \Z}\E[X_0X_l]^m \in [0,\infty)
\end{equation}
and we have 
\begin{equation}
    V_n \xrightarrow[n \to \infty]{d} \cN(0, \sigma^2), \label{BM special}
\end{equation}
where $\cN(0,0)$ denotes the degenerate distribution $\delta_0$.
We refer the reader to \cite[Section 7]{NourdinPeccati} for proofs in more general settings via Malliavin--Stein techniques and for references to related works.

The convergence of the corresponding density functions in \eqref{BM special} has been studied from several perspectives. 
Let  $F_n  = V_n \norm{V_n}^{-1}_{L^2(\Omega)}$. 
Upper bounds for $\norm{\rho_{F_n} - \rho_{\cN}}_{L^1(\R)}$ can be obtained from estimates of the total variation distance, for instance, in \cite{NPStein'smethodonWienerchaos, NPoptimalTV}, where \cite{NPoptimalTV} provides the matching upper and lower bounds.
As for the uniform convergence, \cite{FisherNourdinNualart} shows that $\norm{\rho_{F_n} - \rho_{\cN}}_{\infty}$ is eventually dominated by $\sqrt{\kappa_4(F_n)}$ when $m = 2$ and $\{X_l\}_{l \in \Z}$ is the so-called fractional Gaussian noise (see the paragraph below Corollary \ref{Cor BM applic}).
The uniform convergence in more general settings is addressed in \cite{DensityBreuerMajor}, and, by applying the result from \cite{HuLuNualart}, it is shown that for all $j \in \N_0$, $\norm{\rho_{F_n}^{(j)} - \rho_{\cN}^{(j)}}_{\infty} \lesssim_j \sqrt{\kappa_4(F_n)}$ holds for all $n$ large enough under an additional assumption on the spectral measure of the sequence.
Recently, \cite{Superconvergence} removes this assumption on the spectral measure and also shows that $\rho_{F_n}$ converges to $\rho_{\cN}$ in $W^{k,r}(\R)$ for every $k \in \N_0$ and $r \in [1,\infty]$, although the rate of this convergence remains unknown. 

Our main results, Theorems \ref{Thm Main results 1} and \ref{Thm Main results 2}, immediately yield the optimal convergence rate in this context without any additional assumptions, thus sharpening the previous results and filling the gap.

\begin{Cor}\label{Cor BM applic}
Under the above setting, 
if $\sum_{l \in \Z} |\E[X_0X_l]|^m < \infty$ and $\sigma^2 \in (0, \infty)$, then for every $k \in \N_0$ and $r \in [1,\infty]$,   
\begin{equation}
    C\mathbf{M}(F_n) \leq \norm{\rho_{F_n} - \rho_{\cN}}_{W^{k,r}(\R)} \leq \wt{C} \mathbf{M}(F_n)
\end{equation}
holds for all $n$ large enough, 
where $C$ and $\wt{C}$ are constants depending on $\{F_n\}_{n \in \N}$, $k$, $m$, and $r$ but not on $n$.
\end{Cor}

For example, let $B_H$ be a fractional Brownian motion with Hurst index $H \in (0,1)$ and consider the centered stationary Gaussian sequence $\{X_{l}\}_{l \in \Z}$ given by 
\begin{equation}
    X_{l} = B_{H}(l+1) - B_H(l). \label{fGn def}
\end{equation}
This sequence is often referred to as fractional Gaussian noise and satisfies 
\begin{equation}
    \E[X_{l}X_{l+r}] = \frac{1}{2} \left(|r+1|^{2H} - 2|r|^{2H} + |r-1|^{2H}\right) 
\end{equation}
for every $l, r \in \Z$. 
It follows that  
\begin{equation}
    \E[X_{l}X_{l+r}] = H(2H-1)|r|^{2H-2} + o\left(|r|^{2H-2}\right) \qquad \text{as $|r| \to \infty$},
\end{equation}
and the convergence \eqref{BM special} holds whenever $H \in (0,1-\frac{1}{2m})$.
We note that the sum still converges to a normal distribution under a different scaling if $H = 1 - \frac{1}{2m}$, whereas the limit is non-Gaussian for $H \in (1-\frac{1}{2m}, 1)$.
See \cite[Section 7]{NourdinPeccati} for more details.
When $H \in (0,1-\frac{1}{2m})$, it is calculated in \cite{OptimalBErates} that 
\begin{align}
    \kappa_3(F_n) \propto 
    \begin{cases}
        n^{-\frac{1}{2}} &\quad  \text{if \  $H \in (0,1-\frac{2}{3m})$},\\
        n^{-\frac{1}{2}}(\log n)^2 &\quad \text{if \  $H = 1-\frac{2}{3m}$},\\
        n^{\frac{3}{2}-3m + 3mH} &\quad \text{if \  $H \in (1-\frac{2}{3m}, 1-\frac{1}{2m})$},
    \end{cases}
    \qquad \text{when $m$ is even},
\end{align}
\begin{align}
    \kappa_4(F_n) \propto 
    \begin{cases}
        n^{-1} &\quad  \text{if \  $H \in (0,1-\frac{3}{4m})$},\\
        n^{-1}(\log n)^3 &\quad \text{if \  $H = 1-\frac{3}{4m}$},\\
        n^{2-4m+4mH} &\quad \text{if \  $H \in (1-\frac{3}{4m}, 1-\frac{1}{2m})$},
    \end{cases}
    \qquad \text{when $m \in \{2,3\}$},
\end{align}
and
\begin{align}
    \kappa_4(F_n) \propto 
    \begin{cases}
        n^{-1} &\quad  \text{if \  $H \in (0,\frac{3}{4})$},\\
        n^{-1}(\log n) &\quad \text{if \  $H = \frac{3}{4}$},\\
        n^{4H-4} &\quad \text{if \  $H \in (\frac{3}{4}, 1-\frac{1}{2m-2})$},\\
        n^{4H-4}(\log n)^2 &\quad \text{if \  $H = 1-\frac{1}{2m-2}$},\\
        n^{2 -4m + 4mH} &\quad \text{if \  $H \in (1-\frac{1}{2m-2}, 1-\frac{1}{2m})$},
    \end{cases}
    \qquad \text{for every $m \geq 4$},
\end{align} 
where we use the notation $a_n \propto b_n$ for two nonnegative sequences $\{a_n\}$ and $\{b_n\}$ to mean that $0 < \liminf_{n \to \infty} \frac{a_n}{b_n} \leq \limsup_{n \to \infty} \frac{a_n}{b_n} < \infty$.
Combining these and Corollary \ref{Cor BM applic}, we can easily deduce the exact order of $\norm{\rho_{F_n} - \rho_{\cN}}_{W^{k,r}(\R)}$ as $n \to \infty$. 

\begin{Prp}
Under the above setting, where $\{X_l\}_{l \in \Z}$ is given by \eqref{fGn def}, the following holds for every $k \in \N_0$ and $r \in [1,\infty]$.
\begin{enumerate}[\normalfont(1)]
    \item When $m = 2$, 
    \begin{align}
        \norm{\rho_{F_n}^{} - \rho_{\cN}^{}}_{W^{k,r}(\R)} \propto
        \begin{cases}
            n^{-\frac{1}{2}} &\quad  \text{if \  $H \in (0,\frac{2}{3})$},\\
            n^{-\frac{1}{2}}(\log n)^2 &\quad \text{if \  $H = \frac{2}{3}$},\\
            n^{-\frac{9}{2}+ 6H} &\quad \text{if \  $H \in (\frac{2}{3}, \frac{3}{4})$}.
        \end{cases}
    \end{align}
    \item When $m = 3$, 
    \begin{align}
        \norm{\rho_{F_n}^{} - \rho_{\cN}^{}}_{W^{k,r}(\R)} \propto
        \begin{cases}
            n^{-1} &\quad  \text{if \  $H \in (0, \frac{3}{4})$},\\
            n^{-1}(\log n)^3 &\quad \text{if \  $H = \frac{3}{4}$},\\
            n^{-10+12H} &\quad \text{if \  $H \in (\frac{3}{4}, \frac{5}{6})$}.
        \end{cases}
    \end{align}
    \item When $m = 4$, 
    \begin{align}
        \norm{\rho_{F_n}^{} - \rho_{\cN}^{}}_{W^{k,r}(\R)} \propto
        \begin{cases}
            n^{-\frac{1}{2}} &\quad  \text{if \  $H \in (0,\frac{5}{6})$},\\
            n^{-\frac{1}{2}}(\log n)^2 &\quad \text{if \  $H = \frac{5}{6}$},\\
            n^{-\frac{21}{2} + 12H} &\quad \text{if \  $H \in (\frac{5}{6}, \frac{7}{8})$}.
        \end{cases}
    \end{align}
    \item When $m \geq 5$ is an odd integer,
    \begin{align}
        \norm{\rho_{F_n}^{} - \rho_{\cN}^{}}_{W^{k,r}(\R)} \propto
        \begin{cases}
            n^{-1} &\quad  \text{if \  $H \in (0,\frac{3}{4})$},\\
            n^{-1}(\log n) &\quad \text{if \  $H = \frac{3}{4}$},\\
            n^{4H-4} &\quad \text{if \  $H \in (\frac{3}{4}, 1-\frac{1}{2m-2})$},\\
            n^{4H-4}(\log n)^2 &\quad \text{if \  $H = 1-\frac{1}{2m-2}$},\\
            n^{2 -4m + 4mH} &\quad \text{if \  $H \in (1-\frac{1}{2m-2}, 1-\frac{1}{2m})$}.
        \end{cases}
    \end{align} 
    \item When $m \geq 6$ is an even integer, 
    \begin{align}
        \norm{\rho_{F_n}^{} - \rho_{\cN}^{}}_{W^{k,r}(\R)} \propto
        \begin{cases}
            n^{-\frac{1}{2}} &\quad  \text{if \  $H \in (0,\frac{7}{8}]$},\\
            n^{4H-4} &\quad \text{if \  $H \in [\frac{7}{8}, \frac{6m-11}{6m-8}]$},\\
            n^{\frac{3}{2} - 3m + 3mH} &\quad \text{if \  $H \in [\frac{6m-11}{6m-8}, 1-\frac{1}{2m})$}.
        \end{cases}
    \end{align}
\end{enumerate}
\end{Prp}

\section{Exact asymptotics}\label{section Exact asymptotics}
This section proves Theorem \ref{Thm Main results 3}.
We first derive the rate of convergence of the one-term Edgeworth approximation for the densities, which is of independent interest.

\subsection{One-term Edgeworth approximation for density functions}\label{subsection One-term Edgeworth approximation}

We first prove the lemma. 

\begin{Lem}\label{Lem 71}
For $j \in \N_0$ and $\varphi \in C_{\mathrm{c}}^{\infty}(\R)$, set 
\begin{equation}
    f_{\varphi^{(j)}}(x) \coloneqq e^{\frac{x^2}{2}}\int_{-\infty}^{x} \left( \varphi^{(j)}(y) - \E\left[\varphi^{(j)}(\cN)\right] \right)e^{-\frac{y^2}{2}}dy
\end{equation}
as in Lemma \ref{Lem varphi^j stein sol}. 
Then we have
\begin{equation}
    \frac{1}{3}\E[\varphi^{(j+3)}(\cN)] = - \E[f_{\varphi^{(j)}}^{(2)}(\cN)].
\end{equation}
\end{Lem}
\begin{proof}
Note that $e^{-\frac{x^2}{2}}f_{\varphi^{(j)}}(x)$ is a rapidly decreasing smooth function.
A simple computation using (iii) of Lemma \ref{Lem Hermite polynomials properties} yields
\begin{align}
    \frac{1}{3}\E[\varphi^{(j+3)}(\cN)] 
    &= \frac{1}{3}\int_{\R}(x^3-3x)\varphi^{(j)}(x)\rho_{\cN}(x)dx\\
    &= \frac{1}{3}\int_{\R}(x^3-3x)\left(\varphi^{(j)}(x) - \E[\varphi^{(j)}(\cN)]\right)\rho_{\cN}(x)dx\\
    &= \frac{1}{3}\int_{\R}(x^3-3x)\left(\int_{-\infty}^{x}\left\{\varphi^{(j)}(y) - \E[\varphi^{(j)}(\cN)]\right\} \rho_{\cN}(y)dy \right)'dx\\
    &= - \int_{\R}(x^2 - 1)f_{\varphi^{(j)}}(x)\rho_{\cN}(x)dx.
\end{align}
On the other hand, since 
\begin{align}
    f_{\varphi^{(j)}}^{(2)}(x) 
    &= f_{\varphi^{(j)}}(x) + x f_{\varphi^{(j)}}'(x) + \varphi^{(j+1)}(x)\\
    &= (x^2 + 1)f_{\varphi^{(j)}}(x) + x\varphi^{(j)}(x) + \varphi^{(j+1)}(x) -x\E[\varphi^{(j)}(\cN)], 
\end{align}
we obtain 
\begin{align}
    - \E[f_{\varphi^{(j)}}^{(2)}(\cN)]
    &= - \int_{\R}\left( (x^2 + 1)f_{\varphi^{(j)}}(x) + 2\varphi^{(j+1)}(x) \right) \rho_{\cN}(x) dx\\
    &= - \int_{\R} (x^2 - 1)f_{\varphi^{(j)}}(x)\rho_{\cN}(x) dx - 2 \int_{\R}\left( f_{\varphi^{(j)}}(x) +  \varphi^{(j+1)}(x)  \right)\rho_{\cN}(x) dx\\
    &= \frac{1}{3}\E[\varphi^{(j+3)}(\cN)] - 2 \int_{\R}\left( f_{\varphi^{(j)}}^{(2)}(x) - xf_{\varphi^{(j)}}'(x)  \right)\rho_{\cN}(x) dx\\
    &= \frac{1}{3}\E[\varphi^{(j+3)}(\cN)],
\end{align}
and the lemma follows. 
\end{proof}

In the next proposition, we assume that $\E[F_n] = 0$ and $\E[F_n^2] = 1$.
In this case, note that the third-order Edgeworth expansion of $\rho_{F_n}^{}$ is given by $\rho_{\cN}^{} - \frac{\kappa_3(F_n)}{3!}\rho_{\cN}^{(3)}$. 
For an introduction to and further details on the Edgeworth expansion, see \cite{McCullagh} and \cite{Hall}. 
For a remark on the next result, see Remark \ref{Rem end}.

\begin{Prp}\label{Prp conv rates in one-term approx}
Let $\mathbf{F} = \{F_n\}_{n \in \N} \subset \cW_m$ with $m \geq 2$ such that $\E[F^2_n] = 1$ and $F_n \xrightarrow[n \to \infty]{d} \cN(0,1)$. 
Then for every $k \in \N_0$ and $r \in [1, \infty]$, we have $\rho_{F_n} \in C_{\mathrm{b}}^{k}(\R) \cap W^{k,r}(\R)$ and 
\begin{equation}
    \left\lVert \rho_{F_n}^{} - \rho_{\cN}^{} + \frac{\kappa_3(F_n)}{3!} \rho_{\cN}^{(3)} \right\rVert \leq C_{\mathbf{F},k,m,r} \left( |\kappa_{3}(F_n)|^2 + |\kappa_3(F_n)|\kappa_4(F_n) + \kappa_4(F_n) \right) \label{uniform Edgeworth conv}
\end{equation} 
for all sufficiently large $n$, where $\norm{\cdot} \in \{ \norm{\cdot}_{C_\mathrm{b}^{k}(\R)}, \norm{\cdot}_{W^{k,r}(\R)} \}$ and $C_{\mathbf{F},k,m,r}$ is a constant depending on $\mathbf{F}$, $k$, $m$, and $r$, but not on $n$.
\end{Prp}

\begin{proof}
Take $p,q \in (1, \infty)$ such that $s_{k+6} = s_{k+6}[p,q] > 1$.
Then, there exists $N_{\mathbf{F}, k} \in \N$ such that for any $\alpha \in \R$, 
\begin{equation}
    \sup_{n \geq {N_{\mathbf{F}, k}}} \norm{F_n}_{\bW^{\alpha, p}} < \infty \qquad \text{and} \qquad \sup_{n \geq {N_{\mathbf{F}, k}}}\norm{\Delta_{F_n}^{-1}}_{L^{q}(\Omega)} < \infty. \label{constant dep}
\end{equation}
Let $n \geq N_{\mathbf{F}, k}$. 
By Lemma \ref{Lem Shigekawa bdd density criterion}, we actually have $\rho_{F_n} \in C_{\mathrm{b}}^{k+5}(\R) \cap W^{k+5,r}(\R)$ in this case. 
To complete the proof, it suffices to show that for every $j \in \{0,\ldots, k\}$, 
\begin{equation}
    \sup_{a \in \R}\left| \rho_{F_n}^{(j)}(a) - \rho_{\cN}^{(j)}(a) + \frac{\kappa_3(F_n)}{3!} \rho_{\cN}^{(j+3)}(a) \right| \lesssim_{\mathbf{F},j,m} \left( |\kappa_{3}(F_n)|^2 + |\kappa_3(F_n)|\kappa_4(F_n) + \kappa_4(F_n) \right) \label{Edge bound uniform}
\end{equation}
and 
\begin{equation}
    \left\lVert \rho_{F_n}^{(j)} - \rho_{\cN}^{(j)} + \frac{\kappa_3(F_n)}{3!} \rho_{\cN}^{(j+3)} \right\rVert_{L^1(\R)} \lesssim_{\mathbf{F},j,m} \left( |\kappa_{3}(F_n)|^2 + |\kappa_3(F_n)|\kappa_4(F_n) + \kappa_4(F_n) \right). \label{Edge bound Sobolev}
\end{equation}

We first establish \eqref{Edge bound uniform}. 
Applying Proposition \ref{Prp Edgeworth expansions for density functions} with $M = 3$, we obtain 
\begin{align}
    &\left| \rho_{F_n}^{(j)}(a) - \rho_{\cN}^{(j)}(a) + \frac{\kappa_3(F_n)}{3!} \rho_{\cN}^{(j+3)}(a) \right| \\
    &\leq \frac{|\kappa_3(F_n)|}{2}\left| \dpairLR{\bW^{-j-2, s_{j+2}}}{(\sD^2f_{a,j})(F_n)}{1}{\bW^{j+2, s_{j+2}'}} - \frac{1}{3}\rho_{\cN}^{(j+3)}(a) \right|\\
    &\quad + \left| \dpairLR{\bW^{-j-3, s_{j+3}}}{(\sD^3f_{a,j})(F_n)}{\Gamma_3(F_n)}{\bW^{j+3, s_{j+3}'}} \right|. \label{aa1}
\end{align}
This, together with Propositions \ref{Prp Pullback est} and \ref{Prp Gamma estimate}, yields 
\begin{align}
    &\sup_{a \in \R} \left| \rho_{F_n}^{(j)}(a) - \rho_{\cN}^{(j)}(a) + \frac{\kappa_3(F_n)}{3!} \rho_{\cN}^{(j+3)}(a) \right|\\
    &\lesssim_{\mathbf{F}, j, m} \frac{|\kappa_3(F_n)|}{2}\sup_{a\in\R}\left| \dpairLR{\bW^{-j-2, s_{j+2}}}{(\sD^2f_{a,j})(F_n)}{1}{\bW^{j+2, s_{j+2}'}} - \frac{1}{3}\rho_{\cN}^{(j+3)}(a) \right| \\
    &\quad + \kappa_4(F_n) \sup_{a \in \R} \norm{\sD^{3}f_{a,j}}_{\cS_{-j-3}},
\end{align}
where the implicit constant depends on $j$ and some positive powers of \eqref{constant dep} with $\alpha = j+4$.
Let $\{\psi_{a,j,l}\}_{l} \subset \cS(\R)$ be a sequence such that $\psi_{a,j,l} \xrightarrow{l \to \infty} f_{a,j}$ in $\cS_{-j}$. 
Then, $\psi^{(2)}_{a,j,l} \xrightarrow{l \to \infty} \sD^{2}f_{a,j}$ in both $\cS_{-j-2}$ and in the weak topology of $\cS'(\R)$, and we thus obtain
\begin{equation}
    \dpairLR{\bW^{-j-2, s_{j+2}}}{(\sD^2f_{a,j})(F_n)}{1}{\bW^{j+2, s_{j+2}'}} = \lim_{l \to \infty} \E[\psi_{a,j,l}^{(2)}(F_n)] \label{lmj1}
\end{equation}
and 
\begin{equation}
    \frac{1}{3}\rho_{\cN}^{(j+3)}(a) = \dpair{\cS'(\R)}{\sD^2f_{a,j}}{\rho_{\cN}}{\cS(\R)} = \lim_{l \to \infty}\E[\psi_{a,j,l}^{(2)}(\cN)], \label{lmj2}
\end{equation}
where the first equality in \eqref{lmj2} follows from Lemma \ref{Lem f_a calc}.
Now, we introduce the smooth function $f_{\psi_{a,j,l}^{(2)}}$ defined by 
\begin{equation}
    f_{\psi_{a,j,l}^{(2)}}(x) = e^{\frac{x^2}{2}} \int_{-\infty}^{x} \left( \psi_{a,j,l}^{(2)}(y) - \E\left[\psi_{a,j,l}^{(2)}(\cN)\right] \right)e^{-\frac{y^2}{2}}dy.
\end{equation}
This function satisfies  
\begin{equation}
    f_{\psi_{a,j,l}^{(2)}}'(x) - xf_{\psi_{a,j,l}^{(2)}}(x) = \psi_{a,j,l}^{(2)}(x) - \E\left[\psi_{a,j,l}^{(2)}(\cN)\right],
\end{equation}
so by the same argument as in the proof of Proposition \ref{Prp f_a,n class Dfan uniform bound}, we obtain 
\begin{align}
    \norm{f_{\psi_{a,j,l}^{(2)}}}_{\cS_{-j-1}} 
    &\leq \sqrt{2} \left\lVert \psi_{a,j,l}^{(2)} - \E\left[\psi_{a,j,l}^{(2)}(\cN)\right] \right\rVert_{\cS_{-j-2}}\\
    &\leq \sqrt{2} \left\lVert \psi_{a,j,l}^{(2)} \right\rVert_{\cS_{-j-2}} + \sqrt{2}\left|\E[\psi_{a,j,l}^{(2)}(\cN)]\right| \norm{1}_{\cS_{-1}},
\end{align}
which in turn implies that for every $i\in \N_0$, 
\begin{align}
    \limsup_{l \to \infty} \left\lVert f_{\psi_{a,j,l}^{(2)}}^{(i)} \right\rVert_{\cS_{-j-1-i}} 
    &\lesssim_{i,j} \limsup_{l \to \infty} \left\lVert f_{\psi_{a,j,l}^{(2)}} \right\rVert_{\cS_{-j-1}} \\
    &\leq \sqrt{2}\norm{\sD^2f_{a,j}}_{\cS_{-j-2}} + \frac{\sqrt{2}}{3} \left|\rho_{\cN}^{(j+3)}(a)\right|\norm{1}_{\cS_{-1}}.
\end{align}
Combining \eqref{lmj1} and \eqref{lmj2} and applying Proposition \ref{Prp Edgeworth-type expansion tempered case} with $M=3$, we deduce that
\begin{align}
    &\left| \dpairLR{\bW^{-j-2, s_{j+2}}}{(\sD^2f_{a,j})(F_n)}{1}{\bW^{j+2, s_{j+2}'}} - \frac{1}{3}\rho_{\cN}^{(j+3)}(a) \right| \\
    &=  \lim_{l \to \infty} \left| \E\left[f_{\psi_{a,j,m}^{(2)}}'(F_n) - F_nf_{\psi_{a,j,m}^{(2)}}(F_n)\right] \right|\\
    &\leq \frac{|\kappa_3(F_n)|}{2} \limsup_{l \to \infty} \left| \dpairLR{\bW^{-j-3, s_{j+3}}}{f_{\psi_{a,j,m}^{(2)}}^{(2)}(F_n)}{1}{\bW^{j+3, s_{j+3}'}} \right|\\
    &\quad + \limsup_{l \to \infty} \left| \dpairLR{\bW^{-j-4, s_{j+4}}}{f_{\psi_{a,j,m}^{(2)}}^{(3)}(F_n)}{\Gamma_3(F_n)}{\bW^{j+4, s_{j+4}'}} \right|\\
    &\leq \frac{|\kappa_3(F_n)|}{2} \limsup_{l \to \infty} \left\lVert f_{\psi_{a,j,m}^{(2)}}^{(2)}(F_n) \right\rVert_{\bW^{-j-3, s_{j+3}}} + \kappa_4(F_n)\limsup_{l \to \infty} \left\lVert f_{\psi_{a,j,m}^{(2)}}^{(3)}(F_n) \right\rVert_{\bW^{-j-4, s_{j+4}}} \\
    &\lesssim_{\mathbf{F},j} \sup_{a \in \R}\left(\norm{\sD^2f_{a,j}}_{\cS_{-j-2}} + \left|\rho_{\cN}^{(j+3)}(a)\right|\right) \left(|\kappa_3(F_n)| + \kappa_4(F_n)\right),
\end{align}
where, in the last line, the constant depends on $j$ and some positive powers of \eqref{constant dep} with $\alpha = j+5$.
This, together with \eqref{aa1}, implies \eqref{Edge bound uniform}.

We next prove \eqref{Edge bound Sobolev}.
Following a similar argument to that in the proof of Proposition \ref{Prp Sobolev norm upper bound}, we see that 
\begin{align}
    &\left\lVert \rho_{F_n}^{(j)} - \rho_{\cN}^{(j)} + \frac{\kappa_3(F_n)}{3!} \rho_{\cN}^{(j+3)} \right\rVert_{L^1(\R)}\\
    &= \sup_{\substack{\varphi \in C_\mathrm{c}^{\infty}(\R) \\ \norm{\varphi}_{L^{\infty}(\R)} \leq 1}} \left|\E\left[\varphi^{(j)}(F_n) - \varphi^{(j)}(\cN) - \frac{\kappa_3(F_n)}{3!}\varphi^{(j+3)}(\cN) \right]\right|\\
    &\leq \frac{|\kappa_3(F_n)|}{2} \sup_{\substack{\varphi \in C_\mathrm{c}^{\infty}(\R) \\ \norm{\varphi}_{L^{\infty}(\R)} \leq 1}} \left|\dpairLR{\bW^{-j-4, s_{j+4}}}{f_{\varphi^{(j)}}^{(2)}(F_n)}{1}{\bW^{j+4, s_{j+4}'}}  + \frac{1}{3}\E[\varphi^{(j+3)}(\cN)]\right| \\
    &\quad + \sup_{\substack{\varphi \in C_\mathrm{c}^{\infty}(\R) \\ \norm{\varphi}_{L^{\infty}(\R)} \leq 1}} \left|\dpairLR{\bW^{-j-5, s_{j+5}}}{f_{\varphi^{(j)}}^{(3)}(F_n)}{\Gamma_3(F_n)}{\bW^{j+5, s_{j+5}'}}\right| \\
    &\lesssim_{\mathbf{F}, j,m} \frac{|\kappa_3(F_n)|}{2}\sup_{\substack{\varphi \in C_\mathrm{c}^{\infty}(\R) \\ \norm{\varphi}_{L^{\infty}(\R)} \leq 1}} \left|\E\left[f_{\varphi^{(j)}}^{(2)}(F_n)\right]  - \E\left[f_{\varphi^{(j)}}^{(2)}(\cN)\right]\right|  +  \kappa_4(F_n),
\end{align}
where the first inequality follows by applying Proposition \ref{Prp Edgeworth-type expansion tempered case} with $M =3$, and the last line follows from Lemma \ref{Lem 71}, Proposition \ref{Prp Gamma estimate}, and \eqref{bb2}.
As before, we introduce the smooth function $\mathfrak{f}$ given by 
\begin{equation}
    \mathfrak{f}(x) = e^{\frac{x^2}{2}}\int_{-\infty}^{x}\left( f_{\varphi^{(j)}}^{(2)}(y) - \E[f_{\varphi^{(j)}}^{(2)}(\cN)] \right)e^{-\frac{y^2}{2}}dy. 
\end{equation}
Then, in a similar way, we see that
\begin{align}
    \sup_{\substack{\varphi \in C_\mathrm{c}^{\infty}(\R) \\ \norm{\varphi}_{L^{\infty}(\R)} \leq 1}} \norm{\mathfrak{f}}_{\cS_{-j-3}}
    &\lesssim \sup_{\substack{\varphi \in C_\mathrm{c}^{\infty}(\R) \\ \norm{\varphi}_{L^{\infty}(\R)} \leq 1}} \left( \norm{f_{\varphi^{(j)}}^{(2)}}_{\cS_{-j-4}} + \left| \E[f_{\varphi^{(j)}}^{(2)}(\cN)] \right| \right)\\
    &\lesssim \sup_{\substack{\varphi \in C_\mathrm{c}^{\infty}(\R) \\ \norm{\varphi}_{L^{\infty}(\R)} \leq 1}} \left(\norm{f_{\varphi^{(j)}}}_{\cS_{-j-2}} + \left|\E[\varphi^{(j+3)}(\cN)]\right| \right)\\
    &\leq \sup_{\substack{\varphi \in C_\mathrm{c}^{\infty}(\R) \\ \norm{\varphi}_{L^{\infty}(\R)} \leq 1}} \norm{f_{\varphi^{(j)}}}_{\cS_{-j-2}} + \left\lVert \rho_{\cN}^{(j+3)} \right\rVert_{L^1(\R)} < \infty,
\end{align}
where we use Lemma \ref{Lem 71} in the second line.
It follows that for any $i \in \N_0$, 
\begin{equation}
    \sup_{\substack{\varphi \in C_\mathrm{c}^{\infty}(\R) \\ \norm{\varphi}_{L^{\infty}(\R)} \leq 1}} \norm{\mathfrak{f}^{(i)}}_{\cS_{-j-3-i}} < \infty.
\end{equation}
Now, by applying Proposition \ref{Prp Edgeworth-type expansion tempered case} with $M=3$ and proceeding as before, we obtain 
\begin{align}
    \sup_{\substack{\varphi \in C_\mathrm{c}^{\infty}(\R) \\ \norm{\varphi}_{L^{\infty}(\R)} \leq 1}} \left|\E\left[f_{\varphi^{(j)}}^{(2)}(F_n)\right]  - \E\left[f_{\varphi^{(j)}}^{(2)}(\cN)\right]\right| 
    &= \sup_{\substack{\varphi \in C_\mathrm{c}^{\infty}(\R) \\ \norm{\varphi}_{L^{\infty}(\R)} \leq 1}} \left|\E\left[ \mathfrak{f}'(F_n) - F_n\mathfrak{f}(F_n) \right]  \right|\\
    &\leq \frac{|\kappa_3(F_n)|}{2}\sup_{\substack{\varphi \in C_\mathrm{c}^{\infty}(\R) \\ \norm{\varphi}_{L^{\infty}(\R)} \leq 1}} 
    \left| \dpairLR{\bW^{-j-5,s_{j+5}}}{\mathfrak{f}^{(2)}(F_n)}{1}{\bW^{j+5,s_{j+5}'}} \right| \\
    &\quad + \sup_{\substack{\varphi \in C_\mathrm{c}^{\infty}(\R) \\ \norm{\varphi}_{L^{\infty}(\R)} \leq 1}} 
    \left| \dpairLR{\bW^{-j-6,s_{j+6}}}{\mathfrak{f}^{(3)}(F_n)}{\Gamma_3(F_n)}{\bW^{j+6,s_{j+6}'}} \right|\\
    &\lesssim_{\mathbf{F},j, m} \sup_{\substack{\varphi \in C_\mathrm{c}^{\infty}(\R) \\ \norm{\varphi}_{L^{\infty}(\R)} \leq 1}} \norm{\mathfrak{f}}_{\cS_{-j-3}} \left(|\kappa_3(F_n)| + \kappa_{4}(F_n)\right),
\end{align}
and therefore \eqref{Edge bound Sobolev} holds.
The proof is complete.
\end{proof}

\subsection{Proof of Theorem \ref{Thm Main results 3}}

We are now in a position to prove Theorem \ref{Thm Main results 3}.

\begin{proof}[Proof of Theorem \ref{Thm Main results 3}]
Let $k \in \N_0$ and $r \in [1,\infty]$.
Since $F_n \xrightarrow[n \to \infty]{d} \cN(0,1)$, we can take $p,q \in (1, \infty)$ and $N_{\mathbf{F},k} \in \N$ such that $s_{k+6} = s_{k+6}[p,q]>1$ and for any $\alpha \in \R$, 
\begin{equation}
    \sup_{n \geq {N_{\mathbf{F}, k}}} \norm{F_n}_{\bW^{\alpha, p}} < \infty \qquad \text{and} \qquad \sup_{n \geq {N_{\mathbf{F}, k}}}\norm{\Delta_{F_n}^{-1}}_{L^{q}(\Omega)} < \infty. \label{uni b 1}
\end{equation}
It follows that $\rho_{F_n} \in C_{\mathrm{b}}^{k+5}(\R) \cap W^{k+5,r}(\R)$ for all $n \geq N_{\mathbf{F},k}$. 
From now on, let $n \geq N_{\mathbf{F},k}$.
We deduce from Proposition \ref{Prp conv rates in one-term approx} that for each norm $\norm{\cdot} \in \{\norm{\cdot}_{C_{\mathrm{b}}^k(\R)}, \norm{\cdot}_{W^{k,r}(\R)} \}$,
\begin{align}
    \left\lVert \frac{\rho_{F_n}^{} - \rho_{\cN}^{}}{\varphi(n)} - \frac{\zeta}{3}\rho_{\cN}^{(3)} \right\rVert 
    &\leq \left\lVert \frac{\rho_{F_n}^{} - \rho_{\cN}^{}}{\varphi(n)} + \frac{\kappa_3(F_n)}{3! \varphi(n)} \rho_{\cN}^{(3)} \right\rVert + \frac{1}{3} \left| \frac{\kappa_3(F_n)}{2\varphi(n)} + \zeta \right| \left\lVert\rho_{\cN}^{(3)}\right\rVert \\
    &\lesssim_{\mathbf{F},k,m,r} \frac{|\kappa_3(F_n)|^2 + |\kappa_3(F_n)|\kappa_4(F_n)  + \kappa_{4}(F_n)}{\varphi(n)} + \left| \frac{\kappa_3(F_n)}{2\varphi(n)} + \zeta \right|,
\end{align}
where the implicit constant depends on $j$ and some positive powers of \eqref{uni b 1} with $\alpha = k + 7$.
Since $|\kappa_3(F_n)|^2 \lesssim \kappa_4(F_n)$, it remains to show that $\lim_{n \to \infty} \kappa_{4}(F_n)\varphi(n)^{-1} = 0$ and $\lim_{n \to \infty}\kappa_3(F_n)\varphi(n)^{-1} = -2\zeta$.
By the assumption \eqref{two dim conv assump}, $F_n$ and $(1- \abra{DF_n, -DL^{-1}F_n}_{\fH})\varphi(n)^{-1}$ converges to normal distributions, so, for each $l \in \N$, the sequence $\left\{F_n^{l} \left(1- \abra{DF_n, -DL^{-1}F_n}_{\fH}\right) \varphi(n)^{-1}\right\}_{n \in \N}$ is uniformly integrable by Lemma \ref{Lem Uni bound sobolev}.
Combining this and Lemma \ref{Lem NP IBP}, we obtain
\begin{equation}
    \frac{\kappa_{3}(F_n)}{2\varphi(n)} = \frac{\E[F_n^3] - 2\E[F_n]}{2\varphi(n)} =  - \E\left[ F_n \frac{1- \abra{DF_n, -DL^{-1}F_n}_{\fH}}{\varphi(n)} \right] \xrightarrow[]{n \to \infty}  -\zeta \label{711}
\end{equation}
and 
\begin{equation}
    \frac{\kappa_4(F_n)}{\varphi(n)} = \frac{\E[F_n^4] - 3\E[F_n^2]}{\varphi(n)} = - 3\E\left[ F_n^2 \frac{1- \abra{DF_n, -DL^{-1}F_n}_{\fH}}{\varphi(n)} \right] \xrightarrow{n \to \infty} 0,  \label{712}
\end{equation}
and the proof is complete.
\end{proof}

We end the paper with some remarks on Theorem \ref{Thm Main results 3} and Proposition \ref{Prp conv rates in one-term approx}.

\begin{Rem}\label{Rem end}
\begin{enumerate}
    \item[(1)] If $m$ is odd, then $\kappa_{3}(F_n) = \E[F_n^3] = 0$, and consequently $\zeta = 0$ by \eqref{711}.
    \item[(2)] It is known (see \textit{e.g.}, \cite[Proposition 3.14]{NPsurveywithnewestimates}) that $|\kappa_3(F_n)|^2 \lesssim \kappa_4(F_n)$. 
    Thus, Proposition \ref{Prp conv rates in one-term approx} actually implies that 
    \begin{equation}
        \left\lVert \rho_{F_n}^{} - \rho_{\cN}^{} + \frac{\kappa_3(F_n)}{3!} \rho_{\cN}^{(3)} \right\rVert = O(\kappa_4(F_n)), \qquad \norm{\cdot} \in \{\norm{\cdot}_{C_{\mathrm{b}}^k(\R)}, \norm{\cdot}_{W^{k,r}(\R)}\}.
    \end{equation}
    This estimate for the case $\norm{\cdot} = \norm{\cdot}_{L^1(\R)}$ has been established only recently in \cite{mansanarez2025edgeworthexpansionwienerchaos}.
    \item[(3)] Under the setting of Theorem \ref{Thm Main results 3}, if $\zeta \neq 0$, then $\{\varphi(n)\}_{n \in \N}$ becomes an optimal convergence rate of $\norm{\rho^{}_{F_n} - \rho^{}_N}_{C_{\mathrm{b}}^{k}(\R)} \to 0$ and $\norm{\rho^{}_{F_n} - \rho^{}_N}_{W^{k,r}(\R)} \to 0$ for every $k \in \N_0$ and $r \in [1,\infty]$. 
    Moreover, by Proposition \ref{Prp conv rates in one-term approx} and \eqref{712}, we have 
    \begin{equation}
        \left\lVert \rho_{F_n}^{} - \rho_{\cN}^{} + \frac{\kappa_3(F_n)}{3!} \rho_{\cN}^{(3)} \right\rVert = o(\varphi(n)), \qquad \norm{\cdot} \in \{\norm{\cdot}_{C_{\mathrm{b}}^k(\R)}, \norm{\cdot}_{W^{k,r}(\R)}\}.
    \end{equation} 
    Thus, adding the term $\frac{\kappa_3(F_n)}{3!} \rho_{\cN}^{(3)}$ actually improves the rate of convergence to zero.
\end{enumerate}

\end{Rem}

\bibliographystyle{amsalpha}
\bibliography{main}

\end{document}